\RequirePackage{fix-cm}
\documentclass[twocolumn]{svjour3}          
\smartqed  
\usepackage{graphicx}
\usepackage{epstopdf}
\usepackage{amssymb}
\usepackage{subfigure}
\usepackage{flushend,cuted}
\usepackage[ruled]{algorithm2e}
\usepackage{amsmath}
\usepackage{url}
\usepackage[colorlinks=true,citecolor=blue,linkcolor=blue]{hyperref}
\usepackage{multirow}
\usepackage{booktabs}
\usepackage{mathrsfs}
\allowdisplaybreaks
\begin{document}

\title{An advanced meshless approach for the high-dimensional multi-term time-space-fractional PDEs {\color{blue}on convex domains}}

\titlerunning{An advanced meshless method for fractional TSFPDEs}        

\author{X. G. Zhu \and Y. F. Nie \and J. G. Wang \and Z. B. Yuan}

\authorrunning{X. G. Zhu et al.} 

\institute{  X. G. Zhu \at
               School of Science, Shaoyang University, Shaoyang, Hunan 422000, P.R. China \\
              \email{zhuxg590@yeah.net}
              \and
             Y. F. Nie \and J. G. Wang \and Z. B. Yuan  \at
               Department of Applied Mathematics, Northwestern  Polytechnical University, Xi'an, Shaanxi 710129, P. R. China \\
              \email{yfnie@nwpu.edu.cn}      
}

\date{Received: date / Accepted: date}

\maketitle

\begin{abstract}
In this article, an advanced differential quadrature (DQ) approach is proposed for
the high-dimensional multi-term time-space-fractional partial differential equations (TSFPDEs) {\color{blue}on convex domains}.
Firstly, a family of high-order difference schemes  is
introduced to discretize the time-fractional derivative and a semi-discrete scheme
for the considered problems is presented.
We strictly prove its unconditional stability and error estimate. Further,
we derive a class of DQ formulas to evaluate the fractional derivatives, which employs
radial basis functions (RBFs) as test functions. Using these DQ formulas in spatial discretization,
a fully discrete DQ scheme is then proposed. Our approach provides a flexible and
high accurate alternative to solve the high-dimensional multi-term TSFPDEs {\color{blue}on convex domains} and its actual
performance is illustrated by contrast to the other methods available in the open literature.
The numerical results {\color{red}confirm the theoretical analysis and the capability of our proposed method finally}.
\end{abstract}

\section{Introduction}\label{intro}
\indent

{\color{red}The fractional partial differential equations (PDEs) are the results of mathematical modelling based on the fractional calculus,
which provide} a new powerful tool to study some complex systems associated with temporal memory or long-range space interaction
{\color{red}in mathematical physics.} 
Nevertheless, {\color{red}it is challenging to solve these type of equations} and few of fractional PDEs can be solved by analytic techniques.
As a result, numerical methods {\color{red}are priorities for development.}
The fractional PDEs can be divided into three categories: the time-fractional PDEs, the space-fractional PDEs, and a mixture of both.
Up to present, the numerical algorithms for the time-fractional PDEs {\color{red} have gradually moved towards maturely after years of
development} \cite{dq08,dq09,dq04,dq02,dq03,dq06,dq05,dq07}, while these for the high-dimensional space-fractional or time-space-fractional
PDEs have to be further developed and improved.
Regardless of the difficulties in constructing numerical algorithms for the space-fractional PDEs,
{\color{red}many methods have been invented to solve them, which cover} the mainstream of today's numerical methods,
such as finite difference (FD) methods
\cite{dq12,dq13,dq10,dq11}, spectral methods \cite{dq18,dq19}, finite element (FE) methods \cite{dq15,dq14,dq16,dq17},
and finite volume methods \cite{dq20,da58}. Liu et al. considered a space-fractional FitzHugh-Nagumo monodomain model by an implicit
semi-alternative direction FD scheme on approximate irregular domains \cite{dq21}. Qiu et al. developed a nodal
discontinuous Galerkin method for the two-dimensional space-fractional diffusion equations \cite{dq23}.
Yang et al. studied a FE method for the two-dimensional nonlinear space-fractional diffusion equations
on convex domains \cite{dq22}. Bhrawy and Zaky developed a Jacobi spectral tau method for the multi-term
TSFPDEs \cite{dq26}. In \cite{dq25}, an efficient numerical algorithm based on FD and FE methods was addressed
for the two-dimensional multi-term time-space-fractional Bloch-Torrey equations.
Qin et al. gave a fully discrete FE scheme {\color{red}for the multi-term time-space-fractional Bloch-Torrey equation}
by bilinear rectangular finite elements \cite{da54}.
Fan et al. proposed a fully discrete FE scheme with unstructured meshes for the two-dimensional multi-term
{\color{red}time-space-fractional \\ diffusion-wave equations} {\color{blue}on convex domains} \cite{dq24}.

Although the study of numerical methods for the fractional PDEs
is an active area of research, few works have been reported for the high-dimensional multi-term TSFPDEs {\color{blue}on convex domains}.
Actually, this topic is full of challenges because of
the non-locality and weakly singular integral kernel of fractional derivatives,
which is the biggest difficulty in the design of their numerical algorithms.
This situation motivates us to seek another efficient numerical methods to solve these equations.
Meshless methods have the ability to construct functional approximation
entirely {\color{red}from the information at a set of scattered nodes or particles,}
and thereby can reduce the computational cost in practise. They serve {\color{red}as promising tools} in dealing with
the structure destruction, high-dimensional crack propagation and large deformation problems
without {\color{red}the expensive meshes generation.}
In the past two decades, meshless methods have achieved significant advances and many meshless methods
have been developed, such as diffuse element method \cite{dq27},
reproducing kernel particle (RKP) method \cite{dq28}, hp-cloud method, element-free Galerkin method \cite{dq29},
meshless local Petrov-Galerkin method \cite{dq31}, boundary element-free method, Kansa's methods \cite{dq32,dq33},
point interpolation (PI) method \cite{dq30}, and so forth. Meshless methods offer more advantages over the mesh-dependent methods in treating the
space-fractional PDEs {\color{red}but seldom works have been reported in this field, let alone} the multi-term TSFPDEs.
Liu et al. addressed a meshless PI technique in strong form for the space-fractional diffusion equation \cite{dq35}.
Cheng et al. proposed an improved moving least-squares collocation scheme for the two-dimensional two-sided space-fractional wave
equation \cite{dq36}. In \cite{da71},  the meshless PI method based on weak form was further developed for
{\color{red}the space-fractional} advection dispersion equations.
{\color{red} The RBFs-based methods have been applied to solve fractional problems but most are limited to the time-fractional cases.
Liu et al. extended an implicit RBFs collocation method to the time-fractional diffusion equation \cite{rbf1}.
Piret and Hanert proposed a RBFs method called by RBF-QR algorithm for the one-dimensional space-fractional
diffusion equation \cite{rbf2}.  Mohebbi et al. proposed a RBFs approach
for the fractional sub-diffusion equation \cite{rbf4}. Pang et al. applied the Kansa's method to  the two-dimensional
space-fractional advection diffusion equation \cite{dq34}. Zafarghandi and  Mohammadi developed a RBFs collocation
method for the one-dimensional space-fractional advection diffusion equation \cite{rbf6}.
Moreover, in \cite{rbf7,rbf9,rbf8,rbf3}, the authors reported the works on RBFs-based approaches for the time-fractional PDEs. }

The basic principle of DQ method was suggested by Bellman and Casti in 1971 \cite{dq38}, and due to the
rapid development of computer simulation, this method has recently aroused extensive concern.
Many scholars applied different basis functions to develop different DQ methods,
such as polynomial DQ method \cite{dq39}, Hermite polynomial DQ \cite{da61}, splines-based DQ method,
orthogonal polynomial DQ method \cite{da60}, and RBFs-based DQ method \cite{dq40}.
At present, {\color{red}a certain progress was made in the new development of DQ methods.
Some works are} committed to the researches of the method itself
and {\color{red}some works} are committed to its new areas of application in structural engineering.
In fractional cases,
Pang et al. extended the classical polynomial DQ method to the steady-state space-fractional
advection diffusion equations \cite{dq37}, where the Lagrange interpolation basis functions are used as
test functions to determine weighted coefficients. In \cite{da69}, we proposed an effective DQ method for
the two-dimensional space-fractional diffusion equations by using RBFs as test functions.
Liu et al. applied the RBFs-based DQ method for the numerical simulation of two-dimensional variable-order
time-fractional advection diffusion equation \cite{da62}.

{\color{red}
Fractional diffusion equations are new numerical models built based on the
continuous time random walk (CTRW) model of statistical physics \cite{rbf23}.
Comparing with classical diffusion equations, they have incomparable advantages in modeling
the physical phenomena exhibiting anomalous diffusion, which is ubiquitous in real life.
Because of this, these models have been widely recognized and used in various applications,
such as thermonuclear reaction, turbulent flow \cite{rbf20}, biological population growth \cite{rbf22},
contaminant transport in groundwater \cite{rbf21}, semiconductor device simulation, and so forth. }
In this work, we {\color{red}propose} an efficient RBFs-based DQ method for the high-dimensional
multi-term TSFPDEs of {\color{red}diffusion equation form} {\color{blue}on convex domains}:

\begin{itemize}
  \item[(i)] two-dimensional multi-term TSFPDE:
\begin{align}\label{gs01}
\left\{
             \begin{array}{lll}\displaystyle
P_{\theta_1,\theta_2,\ldots,\theta_s}\big({^C_0}\mathcal{D}_t\big)u(x,y,t)-
    \varepsilon^+_\alpha(x,y)\frac{\partial^{\alpha}_+ u(x,y,t)}{\partial x^{\alpha}_+} \\
    \quad\displaystyle -\varepsilon^-_\alpha(x,y)\frac{\partial^{\alpha}_- u(x,y,t)}{\partial x^{\alpha}_-}
    -\varepsilon^+_\beta(x,y)\frac{\partial^{\beta}_+ u(x,y,t)}{\partial y^{\beta}_+} \\
     \quad\displaystyle  -\varepsilon^-_\beta(x,y)\frac{\partial^{\beta}_- u(x,y,t)}{\partial y^{\beta}_-} \\
     \quad\displaystyle =f(x,y,t),  \quad (x,y;t)\in\Omega\times(0,T],\\
       u(x,y,0)=u_0(x,y),\quad (x,y)\in\Omega, \\
   u(x,y,t)=g(x,y,t), \quad (x,y;t)\in\partial\Omega\times(0,T],
             \end{array}
        \right.
\end{align}
where $0<\theta_1,\theta_2,\ldots,\theta_s\leq 1$, $s\in \mathbb{Z}^+$, $1<\alpha, \beta\leq 2$, $\Omega\subset \mathbb{R}^2$ with
$\partial\Omega$ being its boundary, {\color{red} and
$\varepsilon^\pm_{\varsigma}(x,y)$ are} the diffusion coefficients with $\varsigma=\alpha, \beta$.
\item[(ii)] three-dimensional multi-term TSFPDE:
\begin{align}\label{gs02}
\small
\left\{
             \begin{array}{lll}\displaystyle
P_{\theta_1,\theta_2,\ldots,\theta_s}\big({^C_0}\mathcal{D}_t\big)u(x,y,z,t)-
    \varepsilon^+_\alpha(x,y,z)\frac{\partial^{\alpha}_+ u(x,y,z,t)}{\partial x^{\alpha}_+} \\
   \quad\displaystyle -\varepsilon^-_\alpha(x,y,z)\frac{\partial^{\alpha}_- u(x,y,z,t)}{\partial x^{\alpha}_-}
    -\varepsilon^+_\beta(x,y,z)\frac{\partial^{\beta}_+ u(x,y,z,t)}{\partial y^{\beta}_+}\\
   \quad\displaystyle-\varepsilon^-_\beta(x,y,z)\frac{\partial^{\beta}_- u(x,y,z,t)}{\partial y^{\beta}_-} -
    \varepsilon^+_\gamma(x,y,z)\frac{\partial^{\gamma}_+ u(x,y,z,t)}{\partial z^{\gamma}_+} \\
    \quad\displaystyle-\varepsilon^-_\gamma(x,y,z)\frac{\partial^{\gamma}_- u(x,y,z,t)}{\partial z^{\gamma}_-} \\
     \quad\displaystyle  =f(x,y,z,t),  \quad (x,y,z;t)\in\Omega\times(0,T],\\
       u(x,y,z,0)=u_0(x,y,z),\quad (x,y,z)\in\Omega, \\
   u(x,y,z,t)=g(x,y,z,t), \quad (x,y,z;t)\in\partial\Omega\times(0,T],
             \end{array}
        \right.
\end{align}
where $0<\theta_1,\theta_2,\ldots,\theta_s\leq 1$, $s\in \mathbb{Z}^+$, $1<\alpha, \beta, \gamma\leq 2$,
$\Omega\subset \mathbb{R}^3$ with $\partial\Omega$ being its boundary, {\color{red}and
$\varepsilon^\pm_{\varsigma}(x,y,z)$ are} the diffusion coefficients with $\varsigma=\alpha, \beta, \gamma$.
\end{itemize}

Assuming $C_1:x=x_L(y,z)$, $C_2:x=x_R(y,z)$ are the forward and backward boundaries,
$C_3:y=y_L(x,z)$, $C_4:y=y_R(x,z)$ are the left and right boundaries, and
$C_5:z=z_L(x,y)$, $C_6:z=z_R(x,y)$ are the lower and upper boundaries of $\Omega\subset\mathbb{R}^3$,
in which
\begin{align*}
  x_L(y,z)&=\min\{x: (x,\eta,\zeta),\eta=y,\zeta=z\},\\
  x_R(y,z)&=\max\{x: (x,\eta,\zeta),\eta=y,\zeta=z\},\\
  y_L(x,z)&=\min\{y: (\eta,y,\zeta),\eta=x,\zeta=z\},\\
  y_R(x,z)&=\max\{y: (\eta,y,\zeta),\eta=x,\zeta=z\},\\
  z_L(x,y)&=\min\{z: (\eta,\zeta,z),\eta=x,\zeta=y\},\\
  z_R(x,y)&=\max\{z: (\eta,\zeta,z),\eta=x,\zeta=y\},
\end{align*}
the space-fractional derivatives 
are defined as follows:
\begin{small}
\begin{align*}
  \frac{\partial^{\alpha}_+ u(x,y,z,t)}{\partial x^{\alpha}_+}&=\frac{1}{\Gamma(2-\alpha)}
    \int^x_{x_L(y,z)}\frac{\partial^2 u(\xi,y,z,t)}{\partial \xi^2}\frac{d\xi}{(x-\xi)^{\alpha-1}},\\
  \frac{\partial^{\alpha}_- u(x,y,z,t)}{\partial x^{\alpha}_-}&=\frac{1}{\Gamma(2-\alpha)}
    \int^{x_R(y,z)}_x\frac{\partial^2 u(\xi,y,z,t)}{\partial \xi^2}\frac{d\xi}{(\xi-x)^{\alpha-1}},\\
   \frac{\partial^{\beta}_+ u(x,y,z,t)}{\partial y^{\beta}_+}&=\frac{1}{\Gamma(2-\beta)}
    \int^y_{y_L(x,z)}\frac{\partial^2 u(x,\xi,z,t)}{\partial \xi^2}\frac{d\xi}{(y-\xi)^{\beta-1}},\\
  \frac{\partial^{\beta}_- u(x,y,z,t)}{\partial y^{\beta}_-}&=\frac{1}{\Gamma(2-\beta)}
    \int^{y_R(x,z)}_y\frac{\partial^2 u(x,\xi,z,t)}{\partial \xi^2}\frac{d\xi}{(\xi-y)^{\beta-1}},\\
   \frac{\partial^{\gamma}_+ u(x,y,z,t)}{\partial z^{\gamma}_+}&=\frac{1}{\Gamma(2-\gamma)}
    \int^z_{z_L(x,y)}\frac{\partial^2 u(x,y,\xi,t)}{\partial \xi^2}\frac{d\xi}{(z-\xi)^{\gamma-1}},\\
  \frac{\partial^{\gamma}_- u(x,y,z,t)}{\partial z^{\gamma}_-}&=\frac{1}{\Gamma(2-\gamma)}
    \int^{z_R(x,y)}_z\frac{\partial^2 u(x,y,\xi,t)}{\partial \xi^2}\frac{d\xi}{(\xi-z)^{\gamma-1}},
\end{align*}
\end{small}
with the Gamma function $\Gamma(\cdot)$.
$P_{\theta_1,\theta_2,\ldots,\theta_s}\big({^C_0}\mathcal{D}_t\big)$ denotes the multi-term fractional derivative operator:
\begin{align}\label{gs03}
\begin{aligned}
P_{\theta_1,\theta_2,\ldots,\theta_s}&\big({^C_0}\mathcal{D}_t\big)u(x,y,z,t)=\sum_{r=1}^{s}a_r{^C_0}D_t^{\theta_r}u(x,y,z,t)\\
   &=a_1{^C_0}D_t^{\theta_1}u(x,y,z,t)+a_2{^C_0}D_t^{\theta_2}u(x,y,z,t) \\
   &\quad+\cdots+a_s{^C_0}D_t^{\theta_s}u(x,y,z,t),
\end{aligned}
\end{align}
where
\begin{align*}
\small
  {^C_0}D_t^{\theta_r}u(x,y,z,t)=\left\{
\begin{aligned}
      &\frac{\partial u(x,y,z,t)}{\partial t}, \quad  \theta_r=1,\\
      &\frac{1}{\Gamma(1-\theta_r)}\int^t_0\frac{\partial u(x,y,z,\xi)}{\partial \xi}\frac{d\xi}{(t-\xi)^{\theta_r}}, \\
      & \qquad\qquad\qquad\qquad\qquad\quad  0<\theta_r<1,
\end{aligned}
\right.
\end{align*}
with $a_1\in \mathbb{R}^+$ and $a_r\geq0$, $r=2,3,\ldots,s$.

{\color{red}As compared to the existing meshless methods, the proposed method
is not only simple, flexible to perform, but also can use a few nodes to achieve high accuracy, because
the fractional derivatives are discretized by the weighted linear sums of the functional values at scattered nodes.}
The layout is as follows.
In Section \ref{s2}, some preliminaries on the fractional calculus and RBFs are introduced.
In Section \ref{sx},  a family of high-order schemes of the time-fractional derivative is
presented and a semi-discrete scheme is obtained.
In Section \ref{sk}, the unconditional stability and convergence are discussed.
In Section \ref{s4}, the DQ formulas for fractional derivatives are presented based on RBFs,
and with these RBFs-based DQ formulas, we further construct a fully discrete DQ scheme for the
multi-term TSFPDEs {\color{blue}on convex domains}.
In Section \ref{s5}, some illustrative examples are carried out to confirm its validity and convergence.
A concluding remark is given in Section \ref{s6}.

\section{Preliminaries} \label{s2}
\indent

{\color{red}We recall some} basic preliminaries on the fractional calculus and RBFs required for further discussions.
Letting $\Omega\subset \mathbb{R}^d$, $d=2,3$, we define
\begin{equation*}
  (u,v)=\int_\Omega uvd\Omega, \quad||u||_{0,\Omega}=\sqrt{(u,u)}.
\end{equation*}
\subsection{Fractional calculus}
\begin{definition}\label{def01}
 The left and right Riemann-Liouville fractional integrals of order $\alpha$ are defined by
 \begin{align*}
  {_a}J^{\alpha}_{x}u(x)=\frac{1}{\Gamma(\alpha)}\int^x_{a}\frac{u(\xi)d\xi}{(x-\xi)^{1-\alpha}},\quad x>a,\\
  {_x}J^{\alpha}_{b}u(x)=\frac{1}{\Gamma(\alpha)}\int^{b}_x\frac{u(\xi)d\xi}{(\xi-x)^{1-\alpha}},\quad x<b,
 \end{align*}
 and if $\alpha=0$, ${_a}J^{\alpha}_{x}u(x)=u(x)$ and ${_x}J^{\alpha}_{b}u(x)=u(x)$.
\end{definition}

\begin{definition}\label{def02}
 The left and right Riemann-Liouville fractional derivatives of order $\alpha$ are defined by
 \begin{align*}
  {^{RL}_a}D_x^{\alpha}u(x)&=\bigg(\frac{d}{dx}\bigg)^m{_a}J^{m-\alpha}_{x}u(x)\\
      &=\frac{1}{\Gamma(m-\alpha)}
  \frac{d^m}{d x^m}\int^x_{a}\frac{u(\xi)d\xi}{(x-\xi)^{\alpha-m+1}},\quad x>a,\\
  {^{RL}_x}D_b^{\alpha}u(x)&=(-1)^m\bigg(\frac{d}{dx}\bigg)^m{_x}J^{m-\alpha}_{b}u(x)\\
      &=\frac{(-1)^m}{\Gamma(m-\alpha)}
  \frac{d^m}{d x^m}\int^b_{x}\frac{u(\xi)d\xi}{(\xi-x)^{\alpha-m+1}},\quad x<b,
 \end{align*}
where $m-1<\alpha<m$, $m\in \mathbb{Z}^+$, and if $\alpha=m$, ${^{RL}_a}D_x^{\alpha}u(x)=u^{(m)}(x)$
and ${^{RL}_x}D_b^{\alpha}u(x)=(-1)^{m}u^{(m)}(x)$.
\end{definition}

\begin{definition}\label{def03}
 The left and right Caputo fractional derivatives of order $\alpha$ are defined by
 \begin{align*}
  {^C_a}D_x^{\alpha}u(x)&={_a}J^{m-\alpha}_{x}u^{(m)}(x)\\
          &=\frac{1}{\Gamma(m-\alpha)}
  \int^x_{a}\frac{u^{(m)}(\xi)d\xi}{(x-\xi)^{\alpha-m+1}},\quad x>a,\\
  {^C_x}D_b^{\alpha}u(x)&=(-1)^m{_x}J^{m-\alpha}_{b}u^{(m)}(x)\\
          &=\frac{(-1)^m}{\Gamma(m-\alpha)}
  \int^b_{x}\frac{u^{(m)}(\xi)d\xi}{(\xi-x)^{\alpha-m+1}},\quad x<b,
 \end{align*}
where $m-1<\alpha<m$, $m\in \mathbb{Z}^+$, and if $\alpha=m$, ${^C_a}D_x^{\alpha}u(x)=u^{(m)}(x)$
and ${^C_x}D_b^{\alpha}u(x)=(-1)^{m}u^{(m)}(x)$.
\end{definition}

{\color{blue}Define the norms
\begin{align*}
&||u||_{J^\alpha_L(\Omega)}=\big( ||u||^2_{0,\Omega}
    +||{^{RL}_a}D_x^{\alpha}u||^2_{0,\Omega} \big)^{1/2},\\
&||u||_{J^\alpha_R(\Omega)}=\big( ||u||^2_{0,\Omega}
    +||{^{RL}_x}D_b^{\alpha}u||^2_{0,\Omega} \big)^{1/2},
\end{align*}
and let $J^{\alpha}_{L,0}(\Omega)$, $J^{\alpha}_{R,0}(\Omega)$ 
be the closures of $C^\infty_0(\Omega)$ with respect to $||\cdot||_{J^\alpha_L(\Omega)}$
and $||\cdot||_{J^\alpha_R(\Omega)}$. 
By \cite{dq14}, for $\alpha\in \mathbb{R}^+$, $\alpha\neq n-1/2$, $n\in \mathbb{Z}^+$,
the $J^{\alpha}_{L,0}(\Omega)$, $J^{\alpha}_{R,0}(\Omega)$ spaces  
are equal with equivalent norms.}

\begin{lemma}\cite{dq14}\label{xl01}
For $\alpha\in \mathbb{R}^+$ {\color{blue}, $u\in J^{\alpha}_{L,0}(\Omega)$
and $J^{\alpha}_{R,0}(\Omega)$, there exists}
\begin{align*}
( {^{RL}_a}D_x^{\alpha}u,{^{RL}_x}D_b^{\alpha}u)=\cos(\alpha\pi)||{^{RL}_a}D_x^{\alpha}u||^2_{0,\Omega}.
\end{align*}
\end{lemma}

\begin{lemma}\cite{dq14} \label{xl02}
For $0<\alpha<1$, {\color{blue}$u,\ v\in J^{2\alpha}_{L,0}(\Omega)$ and $J^{2\alpha}_{R,0}(\Omega)$,} there exist
\begin{align*}
( {^{RL}_a}D_x^{2\alpha}u,v)&=( {^{RL}_a}D_x^{\alpha}u,{^{RL}_x}D_b^{\alpha}v),\\
\quad ( {^{RL}_x}D_b^{2\alpha}u,v)&=( {^{RL}_x}D_b^{\alpha}u,{^{RL}_a}D_x^{\alpha}v).
\end{align*}
\end{lemma}

The Caputo and Riemann-Liouville fractional derivatives are  converted to one another by the relation
\begin{align*}
  {^C_a}D^\alpha_xu(x)&={^{RL}_a}D^\alpha_xu(x)-\sum^{m-1}_{l=0}\frac{u^{(l)}(a)}{\Gamma(l+1-\alpha)}(x-a)^{l-\alpha},\\
  {^C_x}D^\alpha_bu(x)&={\color{red}{^{RL}_x}D^\alpha_bu(x)}-\sum^{m-1}_{l=0}\frac{u^{(l)}(b)}{\Gamma(l+1-\alpha)}(b-x)^{l-\alpha}.
\end{align*}
In addition, we have the properties:
\begin{align}
  &{^{C}_a}D^\alpha_x C=0,\\
  &{_a}J^{\alpha_1}_{x}{_a}J^{\alpha_2}_{x}u(x)={_a}J^{\alpha_2}_{x}{_a}J^{\alpha_1}_{x}u(x)={_a}J^{\alpha_1+\alpha_2}_{x}u(x),
\end{align}
\begin{align}\label{gs19}
\begin{aligned}
  {^{C}_a}D^\alpha_x (x-a)^\beta&={^{RL}_a}D^\alpha_x (x-a)^\beta \\
        &=\frac{\Gamma(1+\beta)}{\Gamma(1+\beta-\alpha)}(x-a)^{\beta-\alpha},
\end{aligned}
\end{align}
where $C$ is a constant and $\beta>[\alpha]+1$ with $[x]$ being the ceiling function, which outputs the smallest integer greater than or equal to $x$.
Another point need to be noticed is that these properties are also true for the right-side fractional calculus.
For more details, we refer the readers to \cite{dq41} for deeper insights.

\subsection{Radial basis functions}

\indent

{\color{red}
Letting $\textbf{x}=(x_1,x_2,\ldots,x_d)$, $\{\textbf{x}_i\}_{i=0}^{M}\in \Omega\subset \mathbb{R}^d$, }
the approximation of $u(\textbf{x})$ can be written as a weighted sum of RBFs in the form:
\begin{equation}\label{gs05}
    u(\textbf{x})\approx\sum_{i=0}^{M}\lambda_i\varphi(r_i)+\sum_{q=1}^{Q}\mu_{q}p_q(\textbf{x}), \quad \textbf{x}\in \Omega,
\end{equation}
{\color{blue}where $\{p_q(\textbf{x})\}_{q=1}^{Q}$ are the basis functions of the polynomial space of degree less than $n$ with
$Q$ being the number of $d$-variate polynomial basis, i.e.,
\begin{align*}
  Q=\binom {n+d-1} {d}, \ d=1,2,3,
\end{align*}
$\{\lambda_i\}_{i=0}^{M}$, $\{\mu_{q}\}_{q=1}^{Q}$ are unknown weights,
and $\{\varphi(r_i)\}_{i=0}^{M}$ are the RBFs with $r_i=||\textbf{x}-\textbf{x}_i||$. }
To ensure that the interpolant properly behave at infinity, the above equations are augmented by
\begin{equation}\label{gs06}
   \sum_{i=0}^{M}\lambda_ip_{q}(\textbf{x}_i)=0, \quad q=1,2,\ldots,Q.
\end{equation}

Put Eqs. (\ref{gs05})-(\ref{gs06}) in the below form
\begin{align}\label{gs07}
\left[      \begin{array}{ll}
             \textbf{A} & \textbf{B}^T\\
             \textbf{B} & \textbf{0}\\
             \end{array}
        \right]
\left(      \begin{array}{l}
             \boldsymbol{\lambda} \\
             \boldsymbol{\mu}\\
             \end{array}
        \right)
=\left(     \begin{array}{l}
             \textbf{u} \\
             \textbf{0}\\
             \end{array}
        \right),
\end{align}
where the matrix elements for $[\textbf{A}]$ are $\varphi(r_{ij})$ and for [\textbf{B}] are $p_j(\textbf{x}_i)$
with $r_{ij}=||\textbf{x}_j-\textbf{x}_i||$. $\boldsymbol{\lambda}$, $\boldsymbol{\mu}$, and $\textbf{u}$
are column vectors {\color{red}and the vector elements of them} are $\lambda_i$, $\mu_i$, and $u(\textbf{x}_i)$, respectively.

\begin{table*}[!htb]
\centering
\caption{Some commonly used RBFs.} \label{tb01}
\begin{tabular}{cl}
\hline \textbf{Name} & \textbf{RBF}\\
\hline Multiquadric (MQ)   &  $\varphi(r)=(r^2+c^2)^{1/2}, \ c>0$\\
       Inverse Multiquadric (IMQ) & $\varphi(r)=1/(r^2+c^2)^{1/2}, \ c>0$ \\
       Inverse Quadratic (IQ) & $\varphi(r)=1/(r^2+c^2), \ c>0$ \\
       Gaussian (GA) &  $\varphi(r)=e^{-r^2/c^2}, \ c>0$\\
       Polyharmonic Spline (PS)  & $\varphi(r)=(-1)^{s+1}r^{2s}\ln r, \ s\in\mathbb{Z}^+$\\
\hline
\end{tabular}
\end{table*}

\begin{definition}\label{def04}
A function is completely monotonic if and only if $(-1)^{\ell}f^{(\ell)}(r)\geq0$
for $\ell=0,1,2,\ldots$, and $r\geq0$.
\end{definition}

\begin{theorem}\label{th1}
\cite{da42} Let an univariate function $\psi(r)\in C^\infty[0,+\infty)$ be such that
$\psi$ is completely monotonic, but not a constant. Suppose further that
$\psi(0)\geq0$. Then the interpolation matrix $[\textbf{A}]$ of the basis function
$\varphi(r)=\psi(r^2)$ is positive definite.
\end{theorem}

The Eq. (\ref{gs07}) is solved to obtain $\{\lambda_i\}_{i=0}^{M}$, $\{\mu_{q}\}_{q=1}^{Q}$.
If $r=||\cdot||$, we give the commonly used RBFs in Table \ref{tb01}, {\color{red}where $c$ is the shape parameter.}
According to Theorem \ref{th1}, except Multiquadrics and Polyharmonic Splines,
the interpolation matrices $[\textbf{A}]$ of the other three RBFs have all positive eigenvalues,
i.e., $[\textbf{A}]$ is invertible and the polynomial term in Eq. (\ref{gs05}) can be removed.
Multiquadrics do not fulfill the above theorem and actually the eigenvalues remain the ones that are
negative. To keep the well-posedness of the resulting algebraic system, the polynomial term
is thus necessary, i.e., $Q\geq1$.

\section{The temporal discretization}\label{sx}
\indent

{\color{red}In this section, we introduce} a family of high-order difference schemes to
discretize {\color{red}the fractional derivatives in time.}
Define a grid $t_n=n\tau$, $n=0,1,\ldots,N$, $T=\tau N$, $N\in\mathbb{Z}^+$, on the time interval $[0,T]$,
and denote the Lubich's difference operator by
\begin{equation*}
\mathscr{L}^{\theta}_qu(\textbf{x},t_n)=\frac{1}{\tau^\theta}\sum_{k=0}^{n}\omega^{q,\theta}_ku(\textbf{x},t_{n-k}),
\end{equation*}
which is always used to discretize the Riemann-Liouville of order $\theta$ \cite{da43},
where $m-1<\theta<m$, $m\in\mathbb{Z}^+$, and $q=1,2,3,4,5$. By virtue of
the relation of both fractional derivatives, there exists
\begin{equation}\label{gs20}
    {^C_0}D^\theta_tu(\textbf{x},t)={_0^{RL}}D^\theta_tu(\textbf{x},t)-\sum^{m-1}_{l=0}\frac{u^{(l)}_t(\textbf{x},0)t^{l-\theta}}{\Gamma(l+1-\theta)}.
\end{equation}

Using the property (\ref{gs19}) and applying the operator $\mathscr{L}^{\theta}_q$ {\color{red}to discretize
the Riemann-Liouville derivative in Eq. (\ref{gs20})} lead to
\begin{align}\label{gs21}
\begin{aligned}
{^C_0}D^\theta_tu(\textbf{x},t_n)&\approx\frac{1}{\tau^\theta}\sum_{k=0}^{n}\omega^{q,\theta}_ku(\textbf{x},t_{n-k})\\
 &\quad-\frac{1}{\tau^\theta}\sum^{m-1}_{l=0}\sum^{n}_{k=0}\frac{\omega^{q,\theta}_k u^{(l)}_t(\textbf{x},0)t_{n-k}^{l}}{l!},
\end{aligned}
\end{align}
where $\{\omega^{q,\theta}_k\}_{k=0}^{n}$ are the discrete coefficients. For example, when $q=1$, we have
$\omega^{1,\theta}_k=\frac{\Gamma{(k-\theta)}}{\Gamma{(-\theta)}\Gamma{(k+1)}}$, $k=0,1,2,\ldots$,
and denote $\omega^{1,\theta}_k$ by  $\omega^{\theta}_k$ hereafter.
\begin{lemma}\label{le01}
The coefficients $\omega^{\theta}_k$ satisfy the properties
\begin{itemize}
   \item[(i)] $\displaystyle\omega^{\theta}_0=1, \quad \omega^{\theta}_k< 0, \quad \sum_{k=0}^{\infty}\omega^{\theta}_k=0,
        \quad \sum_{k=0}^{n-1}\omega^{\theta}_k>0,\ \ \forall k=1,2,\ldots$,
   \item[(ii)] $\displaystyle\omega^{\theta}_0=1,\ \  \omega^{\theta}_k=\bigg(1-\frac{\theta+1}{k}\bigg)\omega^{\theta}_{k-1},\ \ \forall k=1,2,\ldots$,
\end{itemize}
{\color{blue} where $\theta\in(0,1]$. }
\end{lemma}

Moreover, by virtue of $\sum_{k=1}^{\infty}\omega^{\theta}_k=-1$ and $\omega^{\theta}_k< 0$, $k\neq0$,
we can derive $-1<\omega^{\theta}_k< 0$ for $k=1,2,3,\ldots$

\begin{lemma}\label{le02} \cite{da43}
Assume that $u(\textbf{x},t)$, ${_{-\infty}^{RL}}D^{\theta+q}_tu(\textbf{x},t)$ and their
Fourier transforms belong to $L^1(\mathbb{R})$ with regard to $t$, then the Lubich's difference operators satisfy
\begin{equation}\label{gs24}
  {_0^{RL}}D^\theta_tu(\textbf{x},t_n)=\mathscr{L}^{\theta}_qu(\textbf{x},t_n)+\mathscr{O}(\tau^q).
\end{equation}
\end{lemma}

\begin{theorem}\label{th01}
Assume that $u(\textbf{x},t)$ with $\theta>0$ is smooth enough with regard to $t$,
then we have
\begin{equation}\label{gs24}
  {^C_0}D^\theta_tu(\textbf{x},t_n)=\mathscr{A}^{\theta}_qu(\textbf{x},t_n)+\mathscr{O}(\tau^q),
\end{equation}
where
\begin{align*}
  \mathscr{A}^{\theta}_qu(\textbf{x},t_n)&=\frac{1}{\tau^\theta}\sum_{k=0}^{n}\omega^{q,\theta}_ku(\textbf{x},t_{n-k})\\
   &\quad-\frac{1}{\tau^\theta}\sum^{m-1}_{l=0}\sum^{n}_{k=0}\frac{\omega^{q,\theta}_k u^{(l)}_t(\textbf{x},0)t_{n-k}^{l}}{l!}.
\end{align*}
\end{theorem}

The proof is trivial by following Lemma \ref{le02}. 
Letting $0<\theta_1,\theta_2,\ldots,\theta_s\leq1$ and applying the operators $\mathscr{A}^{\theta_r}_q$
to the multi-term fractional derivative operator, we finally have
\begin{align}\label{gs25}
\begin{aligned}
  &P_{\theta_1,\theta_2,\ldots,\theta_s}\big({^C_0}\mathcal{D}_t\big)u(\textbf{x},t_n)\\
   &\qquad=a_1{^C_0}D_t^{\theta_1}u(\textbf{x},t_n)+a_2{^C_0}D_t^{\theta_2}u(\textbf{x},t_n)\\
   &\qquad\qquad\qquad +\cdots+a_s{^C_0}D_t^{\theta_s}u(\textbf{x},t_n)\\
   &\qquad=a_1\mathscr{A}^{\theta_1}_qu(\textbf{x},t_n)+a_2\mathscr{A}^{\theta_2}_qu(\textbf{x},t_n) \\
   &\qquad\qquad\qquad  +\cdots+a_s\mathscr{A}^{\theta_s}_qu(\textbf{x},t_n)+\mathscr{O}(\tau^q)\\
   &\qquad=\sum_{r=1}^{s}a_r\mathscr{A}^{\theta_r}_qu(\textbf{x},t_n)+\mathscr{O}(\tau^q)\\
   &\qquad= \sum_{r=1}^{s}\frac{a_r}{\tau^{\theta_r}}\sum_{k=0}^{n-1}\omega^{q,\theta_r}_ku(\textbf{x},t_{n-k})\\
   &\qquad\qquad\qquad   -\sum_{r=1}^{s}\frac{a_r}{\tau^{\theta_r}}\sum_{k=0}^{n-1}\omega^{q,\theta_r}_ku(\textbf{x},0)+\mathscr{O}(\tau^q).
\end{aligned}
\end{align}



We derive the semi-discretization in time for Eqs. (\ref{gs01})-(\ref{gs02}) and for simplicity, letting
$\textbf{x}=(x_1,x_2,\ldots,x_d)$, $\Omega\subset \mathbb{R}^d$, we
put Eqs. (\ref{gs01})-(\ref{gs02}) in a unified form:
\begin{align*}
\left\{
             \begin{array}{lll}\displaystyle
    \begin{aligned}
    &P_{\theta_1,\theta_2,\ldots,\theta_s}\big({^C_0}\mathcal{D}_t\big)u(\textbf{x},t)-\sum_{l=1}^d
    \varepsilon^+_{\alpha_l}(\textbf{x})\frac{\partial^{\alpha_l}_+ u(\textbf{x},t)}{\partial x^{\alpha_l}_{l,+}} \\
   &\quad -\sum_{l=1}^d\varepsilon^-_{\alpha_l}(\textbf{x})\frac{\partial^{\alpha_l}_- u(\textbf{x},t)}{\partial x^{\alpha_l}_{l,-}}
       =f(\textbf{x},t),  \quad (\textbf{x};t)\in\Omega\times(0,T],\\
      & u(\textbf{x},0)=u_0(\textbf{x}),\quad \textbf{x}\in\Omega, \\
   &u(\textbf{x},t)=g(\textbf{x},t), \quad (\textbf{x};t)\in\partial\Omega\times(0,T],
    \end{aligned}
             \end{array}
        \right.
\end{align*}
where $0<\theta_1,\theta_2,\ldots,\theta_s\leq 1$, $1<\alpha_l\leq 2$ and $l=1,2,\ldots,d$.

Using the operator $\mathscr{A}^{\theta}_q$ to
discretize $P_{\theta_1,\theta_2,\ldots,\theta_s}\big({^C_0}\mathcal{D}_t\big)$ arrives at
\begin{equation}\label{gs28}
\begin{aligned}
   &\sum_{r=1}^{s}a_r\mathscr{A}^{\theta_r}_qu(\textbf{x},t_n)-\sum_{l=1}^d
    \varepsilon^+_{\alpha_l}(\textbf{x})\frac{\partial^{\alpha_l}_+ u(\textbf{x},t_n)}{\partial x^{\alpha_l}_{l,+}} \\
   &\quad -\sum_{l=1}^d\varepsilon^-_{\alpha_l}(\textbf{x})\frac{\partial^{\alpha_l}_- u(\textbf{x},t_n)}{\partial x^{\alpha_l}_{l,-}}
       =f(\textbf{x},t_n)+\mathscr{O}(\tau^q),
\end{aligned}
\end{equation}
Omitting the truncation error $\mathscr{O}(\tau^q)$, we obtain the following semi-discrete scheme:
\begin{align}\label{gs290}
\begin{aligned}
    &\displaystyle\sum_{r=1}^{s}\frac{a_r\omega^{q,\theta_r}_0}{\tau^{\theta_r}}U^n(\textbf{x})
      -\sum_{l=1}^d\varepsilon^+_{\alpha_l}(\textbf{x})\frac{\partial^{\alpha_l}_+ U^n(\textbf{x})}{\partial x^{\alpha_l}_{l,+}} \\
    &\displaystyle\quad -\sum_{l=1}^d\varepsilon^-_{\alpha_l}(\textbf{x})\frac{\partial^{\alpha_l}_- U^n(\textbf{x})}{\partial x^{\alpha_l}_{l,-}}
      =f(\textbf{x},t_n) \\
     & \displaystyle\quad-\sum_{r=1}^{s}\frac{a_r}{\tau^{\theta_r}}\sum_{k=1}^{n-1}\omega^{q,\theta_r}_kU^{n-k}(\textbf{x}) \\
     & \displaystyle\quad +\sum_{r=1}^{s}\frac{a_r}{\tau^{\theta_r}}\sum_{k=0}^{n-1}\omega^{q,\theta_r}_kU^0(\textbf{x}),
       \quad n=1,2,\ldots,N.
\end{aligned}
\end{align}

\section{Stability and convergent analysis}\label{sk}
\indent

{\color{red}In this section, we attempt to} strictly prove the stability and error estimate of the semi-discrete scheme (\ref{gs290}).
For simplicity, we assume $\varepsilon^+_{\alpha_l}(\textbf{x})$, $\varepsilon^-_{\alpha_l}(\textbf{x})$
are all constants $\varepsilon^+_{\alpha_l}$, $\varepsilon^-_{\alpha_l}$ {\color{blue}and $u(\textbf{x},t)$ has zero gradient on $\partial\Omega$.}
We mainly focus on the case of $q=1$. By Lemma \ref{xl01}, let us define the energy norm as

\begin{align*}
  &|u|_{\alpha_l,\Omega}=\Bigg\{\sum_{l=1}^d\varepsilon^+_{\alpha_l}\Bigg|\Bigg(\frac{\partial^{\alpha_l/2}_+ u}{\partial x^{\alpha_l/2}_{l,+}},
        \frac{\partial^{\alpha_l/2}_- u}{\partial x^{\alpha_l/2}_{l,-}}\Bigg)\Bigg|\\
  &\qquad\qquad\qquad\qquad +\sum_{l=1}^d\varepsilon^-_{\alpha_l}\Bigg|\Bigg(\frac{\partial^{\alpha_l/2}_- u}{\partial x^{\alpha_l/2}_{l,-}},
         \frac{\partial^{\alpha_l/2}_+ u}{\partial x^{\alpha_l/2}_{l,+}}\Bigg)\Bigg|\Bigg\}^{1/2},\\
  &||u||_{\alpha_l,\Omega}=\big(||u||^2_{0,\Omega}+|u|^2_{\alpha_l,\Omega}\big)^{1/2}.
\end{align*}

\begin{lemma}\label{le05} \cite{da73}
The coefficients $\omega^{\theta}_k$ enjoy the property
\begin{equation}\label{gs60}
  \frac{1}{n^\theta\Gamma(1-\theta)}<\sum_{k=0}^{n-1}\omega^{\theta}_k=-\sum_{k=n}^{\infty}\omega^{\theta}_k\leq \frac{1}{n^\theta},
\end{equation}
where $n=1,2,3,\ldots$.
\end{lemma}

In the sequel, we carry out the stability and convergent analysis for the semi-discrete scheme.
\begin{theorem}\label{th05}
The discrete scheme (\ref{gs290}) is unconditionally stable.
\end{theorem}
\begin{proof}
Let $\tilde{U}^n(\textbf{x})$ be the perturbation solution of the exact solution $u^n(\textbf{x})$
and $e^n=u^n(\textbf{x})-\tilde{U}^n(\textbf{x})$. According to Eq. (\ref{gs290}), {\color{red}we have}
\begin{align*}
  \begin{aligned}
    &\displaystyle\sum_{r=1}^{s}\frac{a_r}{\tau^{\theta_r}}e^n
      -\sum_{l=1}^d\varepsilon^+_{\alpha_l}\frac{\partial^{\alpha_l}_+ e^n}{\partial x^{\alpha_l}_{l,+}}
     -\sum_{l=1}^d\varepsilon^-_{\alpha_l}\frac{\partial^{\alpha_l}_- e^n}{\partial x^{\alpha_l}_{l,-}} \\
     &  \quad =-\sum_{r=1}^{s}\frac{a_r}{\tau^{\theta_r}}\sum_{k=1}^{n-1}\omega^{\theta_r}_ke^{n-k}
     +\sum_{r=1}^{s}\frac{a_r}{\tau^{\theta_r}}\sum_{k=0}^{n-1}\omega^{\theta_r}_ke^0.
\end{aligned}
\end{align*}
By multiplying this equation by $e^n$ and integrating it over $\Omega$, we obtain
\begin{align*}
  \begin{aligned}
    &\sum_{r=1}^{s}\frac{a_r}{\tau^{\theta_r}}(e^n,e^n)
      -\sum_{l=1}^d\varepsilon^+_{\alpha_l}\Bigg(\frac{\partial^{\alpha_l}_+ e^n}{\partial x^{\alpha_l}_{l,+}},e^n\Bigg)\\
    & -\sum_{l=1}^d\varepsilon^-_{\alpha_l}\Bigg(\frac{\partial^{\alpha_l}_- e^n}{\partial x^{\alpha_l}_{l,-}},e^n\Bigg)
     =-\sum_{r=1}^{s}\frac{a_r}{\tau^{\theta_r}}\sum_{k=1}^{n-1}\omega^{\theta_r}_k(e^{n-k},e^n) \\
    &   +\sum_{r=1}^{s}\frac{a_r}{\tau^{\theta_r}}\sum_{k=0}^{n-1}\omega^{\theta_r}_k(e^0,e^n).
\end{aligned}
\end{align*}
In view of Lemma \ref{xl01}, Lemma \ref{xl02}, we have
\begin{align*}
  \begin{aligned}
   \Lambda(e^n,e^n)&=-\sum_{l=1}^d\varepsilon^+_{\alpha_l}\Bigg(\frac{\partial^{\alpha_l}_+ e^n}{\partial x^{\alpha_l}_{l,+}},e^n\Bigg)
-\sum_{l=1}^d\varepsilon^-_{\alpha_l}\Bigg(\frac{\partial^{\alpha_l}_- e^n}{\partial x^{\alpha_l}_{l,-}},e^n\Bigg) \\
   & =-\sum_{l=1}^d\varepsilon^+_{\alpha_l}\Bigg(\frac{\partial^{\alpha_l/2}_+ e^n}{\partial x^{\alpha_l/2}_{l,+}},
      \frac{\partial^{\alpha_l/2}_- e^n}{\partial x^{\alpha_l/2}_{l,-}}\Bigg)\\
   & \qquad -\sum_{l=1}^d\varepsilon^-_{\alpha_l}\Bigg(\frac{\partial^{\alpha_l/2}_- e^n}{\partial x^{\alpha_l/2}_{l,-}},
    \frac{\partial^{\alpha_l/2}_+ e^n}{\partial x^{\alpha_l/2}_{l,+}}\Bigg),
\end{aligned}
\end{align*}
which implies $\Lambda(e^n,e^n)=|u|_{\alpha_l,\Omega}\geq0$, and consequently
\begin{align}\label{gs66}
  \begin{aligned}
    \sum_{r=1}^{s}\frac{a_r}{\tau^{\theta_r}}(e^n,e^n)
     \leq-\sum_{r=1}^{s}\frac{a_r}{\tau^{\theta_r}}\sum_{k=1}^{n-1}\omega^{\theta_r}_k(e^{n-k},e^n) \\
      +\sum_{r=1}^{s}\frac{a_r}{\tau^{\theta_r}}\sum_{k=0}^{n-1}\omega^{\theta_r}_k(e^0,e^n).
\end{aligned}
\end{align}

{\color{red}On the other hand, we have}
\begin{align*}
  \sum_{k=1}^{n-1}\omega^{\theta_r}_k(e^{n-k},e^n)&\geq \sum_{k=1}^{n-1}\omega^{\theta_r}_k\frac{||e^{n-k}||^2_{0,\Omega}+||e^n||^2_{0,\Omega}}{2}, \\
  \sum_{k=0}^{n-1}\omega^{\theta_r}_k(e^0,e^n)&\leq \sum_{k=0}^{n-1}\omega^{\theta_r}_k\frac{||e^{0}||^2_{0,\Omega}+||e^n||^2_{0,\Omega}}{2}. 
\end{align*}
Substituting the above inequalities into (\ref{gs66}) reaches to
\begin{align}\label{gs69}
  \begin{aligned}
    \sum_{r=1}^{s}\frac{a_r}{\tau^{\theta_r}}||e^n||^2_{0,\Omega}
     &\leq\sum_{r=1}^{s}\frac{a_r}{\tau^{\theta_r}}\Bigg(1+\frac{1}{2}\sum_{k=1}^{n-1}\omega^{\theta_r}_k\Bigg)||e^0||^2_{0,\Omega} \\
     &\quad -\frac{1}{2}\sum_{r=1}^{s}\frac{a_r}{\tau^{\theta_r}}\sum_{k=1}^{n-1}\omega^{\theta_r}_k||e^{n-k}||^2_{0,\Omega}.
\end{aligned}
\end{align}

Next, we use the mathematical induction to prove the unconditional stability. Noticing
$||e^1||_{0,\Omega}\leq ||e^0||_{0,\Omega}$ and supposing $||e^\ell||_{0,\Omega}\leq ||e^0||_{0,\Omega}$,
$\ell=1,2,\cdots,n-1$, from the inequality (\ref{gs69}), it follows that
\begin{align*}
   ||e^n||^2_{0,\Omega}
     &\leq\Bigg(\sum_{r=1}^{s}\frac{a_r}{\tau^{\theta_r}}\Bigg)^{-1}\Bigg[
      \sum_{r=1}^{s}\frac{a_r}{\tau^{\theta_r}}\Bigg(1+\frac{1}{2}\sum_{k=1}^{n-1}\omega^{\theta_r}_k\Bigg)||e^0||^2_{0,\Omega} \\
     &\quad-\frac{1}{2}\sum_{r=1}^{s}\frac{a_r}{\tau^{\theta_r}}\sum_{k=1}^{n-1}\omega^{\theta_r}_k||e^{n-k}||^2_{0,\Omega}\Bigg]\\
     &\leq\Bigg(\sum_{r=1}^{s}\frac{a_r}{\tau^{\theta_r}}\Bigg)^{-1}\Bigg[
      \sum_{r=1}^{s}\frac{a_r}{\tau^{\theta_r}}||e^{0}||^2_{0,\Omega}\Bigg]\\
     &= ||e^{0}||^2_{0,\Omega},
\end{align*}
and this completes the mathematical induction.
\end{proof}

We have the following convergent results.
\begin{theorem}\label{th06}
Letting $u^n(\textbf{x})$ be the exact solution of Eqs. (\ref{gs01})-(\ref{gs02})
and $U^n(\textbf{x})$ be the approximate solution obtained by the discrete scheme (\ref{gs290}),
we have
\begin{equation}\label{gs70}
  ||u^n(\textbf{x})-U^n(\textbf{x})||_{0,\Omega}\leq C\tau,
\end{equation}
where $C$ is a constant and $n=1,2,3,\ldots$
\end{theorem}
\begin{proof}
Let $\vartheta^n=u^n(\textbf{x})-U^n(\textbf{x})$. Subtracting Eq. (\ref{gs290})
from Eqs. (\ref{gs01})-(\ref{gs02}) and using Theorem \ref{th01}, we obtain
\begin{align*}
    &\displaystyle\sum_{r=1}^{s}\frac{a_r}{\tau^{\theta_r}}\vartheta^n
      -\sum_{l=1}^d\varepsilon^+_{\alpha_l}\frac{\partial^{\alpha_l}_+ \vartheta^n}{\partial x^{\alpha_l}_{l,+}} \\
     &  \quad -\sum_{l=1}^d\varepsilon^-_{\alpha_l}\frac{\partial^{\alpha_l}_- \vartheta^n}{\partial x^{\alpha_l}_{l,-}}
     =-\sum_{r=1}^{s}\frac{a_r}{\tau^{\theta_r}}\sum_{k=1}^{n-1}\omega^{\theta_r}_k\vartheta^{n-k}+ R,
\end{align*}
where $\vartheta^0=0$ and $R=\mathscr{O}(\tau)$.
Do the multiplication and integration as before, we have
\begin{align}\label{gs80}
    \sum_{r=1}^{s}\frac{a_r}{\tau^{\theta_r}}||\vartheta^n||^2_{0,\Omega}
     \leq-\sum_{r=1}^{s}\frac{a_r}{\tau^{\theta_r}}\sum_{k=1}^{n-1}\omega^{\theta_r}_k(\vartheta^{n-k},\vartheta^n)
    +(R,\vartheta^n).
\end{align}
In view of the following inequalities
\begin{align*}
  \sum_{k=1}^{n-1}&\omega^{\theta_r}_k(\vartheta^{n-k},\vartheta^n)\geq
         \sum_{k=1}^{n-1}\omega^{\theta_r}_k\frac{||\vartheta^{n-k}||^2_{0,\Omega}+||\vartheta^n||^2_{0,\Omega}}{2},
\end{align*}
\begin{align*}
\begin{aligned}
  &(R,\vartheta^n)\leq \sum_{r=1}^{s}\frac{a_r}{2\tau^{\theta_r}}\sum_{k=0}^{n-1}\omega^{\theta_r}_k||\vartheta^n||^2_{0,\Omega}
         +\frac{||R||^2_{0,\Omega}}{2\displaystyle\sum_{r=1}^{s}\frac{a_r}{\tau^{\theta_r}}\sum_{k=0}^{n-1}\omega^{\theta_r}_k}\\
  &\quad \leq \sum_{r=1}^{s}\frac{a_r}{2\tau^{\theta_r}}\sum_{k=0}^{n-1}\omega^{\theta_r}_k||\vartheta^n||^2_{0,\Omega}+
         \frac{||R||^2_{0,\Omega}}{\displaystyle\sum_{r=1}^{s}\frac{2a_r}{\tau^{\theta_r}n^{\theta_r}\Gamma{(1-\theta_r)}}}  \\
  &\quad \leq \sum_{r=1}^{s}\frac{a_r}{2\tau^{\theta_r}}\sum_{k=0}^{n-1}\omega^{\theta_r}_k||\vartheta^n||^2_{0,\Omega}
           +\frac{||R||^2_{0,\Omega}}{\displaystyle\sum_{r=1}^{s}\frac{2a_r}{T^{\theta_r}\Gamma{(1-\theta_r)}}},
\end{aligned}
\end{align*}
we further obtain
\begin{align*}
   & \sum_{r=1}^{s}\frac{a_r}{\tau^{\theta_r}}||\vartheta^n||^2_{0,\Omega} \\
   &\quad  \leq-\sum_{r=1}^{s}\frac{a_r}{2\tau^{\theta_r}} \sum_{k=1}^{n-1}\omega^{\theta_r}_k
        (||\vartheta^{n-k}||^2_{0,\Omega}+||\vartheta^n||^2_{0,\Omega}) \nonumber \\
   &\quad\ \ + \sum_{r=1}^{s}\frac{a_r}{2\tau^{\theta_r}}\sum_{k=0}^{n-1}\omega^{\theta_r}_k||\vartheta^n||^2_{0,\Omega}
           +\frac{||R||^2_{0,\Omega}}{\displaystyle\sum_{r=1}^{s}\frac{2a_r}{T^{\theta_r}\Gamma{(1-\theta_r)}}} \nonumber \\
   &\quad \leq-\sum_{r=1}^{s}\frac{a_r}{2\tau^{\theta_r}} \sum_{k=1}^{n-1}\omega^{\theta_r}_k
        ||\vartheta^{n-k}||^2_{0,\Omega} \nonumber \\
   &\quad\ \ + \sum_{r=1}^{s}\frac{a_r}{2\tau^{\theta_r}}||\vartheta^n||^2_{0,\Omega}
           +\frac{||R||^2_{0,\Omega}}{\displaystyle\sum_{r=1}^{s}\frac{2a_r}{T^{\theta_r}\Gamma{(1-\theta_r)}}} \nonumber
\end{align*}
by using Lemma \ref{le05} and substituting the above inequalities into (\ref{gs80}).
Then, there exists
\begin{align}\label{gs85}
\begin{aligned}
   & \sum_{r=1}^{s}\frac{a_r}{\tau^{\theta_r}}||\vartheta^n||^2_{0,\Omega} \\
   &\quad \leq-\sum_{r=1}^{s}\frac{a_r}{\tau^{\theta_r}} \sum_{k=1}^{n-1}\omega^{\theta_r}_k
        ||\vartheta^{n-k}||^2_{0,\Omega}
           +C_e||R||^2_{0,\Omega},
\end{aligned}
\end{align}
where $C_e=1\big/\sum_{r=1}^{s}\frac{a_r}{T^{\theta_r}\Gamma{(1-\theta_r)}}$.

Let $a_m=\min\{a_1,a_2,\cdots,a_s\}$,  $\theta_m=\min\{\theta_1,\theta_2,\cdots,\theta_s\}$,
and continue to prove the following inequality
\begin{equation}\label{87}
  \sum_{r=1}^{s}\frac{a_r}{\tau^{\theta_r}}||\vartheta^n||^2_{0,\Omega} \leq C_e
    \Bigg(\sum_{k=0}^{n-1}\omega^{\theta_m}_k\Bigg)^{-1}||R||^2_{0,\Omega}
\end{equation}
via the mathematical induction to achieve the error bound (\ref{gs70}). If $0<\theta_m<1$, then for
$n=1$, we can easily check its correctness. Supposing
\begin{equation}\label{88}
  \sum_{r=1}^{s}\frac{a_r}{\tau^{\theta_r}}||\vartheta^\ell||^2_{0,\Omega} \leq C_e
    \Bigg(\sum_{k=0}^{\ell-1}\omega^{\theta_m}_k\Bigg)^{-1}||R||^2_{0,\Omega}
\end{equation}
with $\ell=1,2,\cdots,n-1$, for $\ell=n$, we can prove
\begin{align*}
   & \sum_{r=1}^{s}\frac{a_r}{\tau^{\theta_r}}||\vartheta^n||^2_{0,\Omega} \\
   &\quad \leq-\sum_{r=1}^{s}\frac{a_r}{\tau^{\theta_r}} \sum_{k=1}^{n-1}\omega^{\theta_r}_k
        ||\vartheta^{n-k}||^2_{0,\Omega}+C_e||R||^2_{0,\Omega}  \\
   &\quad \leq-C_e\sum_{k=1}^{n-1}\omega^{\theta_m}_k\Bigg(\sum_{k=0}^{n-k-1}\omega^{\theta_m}_k\Bigg)^{-1}
        ||R||^2_{0,\Omega}+C_e||R||^2_{0,\Omega}  \\
   &\quad \leq-C_e\sum_{k=1}^{n-1}\omega^{\theta_m}_k\Bigg(\sum_{k=0}^{n-1}\omega^{\theta_m}_k\Bigg)^{-1}
        ||R||^2_{0,\Omega}+C_e||R||^2_{0,\Omega}  \\
   &\quad \leq C_e\Bigg(1-\sum_{k=0}^{n-1}\omega^{\theta_m}_k\Bigg)\Bigg(\sum_{k=0}^{n-1}\omega^{\theta_m}_k\Bigg)^{-1}
        ||R||^2_{0,\Omega}+C_e||r^n||^2_{0,\Omega}  \\
   &\quad \leq C_e\Bigg(\sum_{k=0}^{n-1}\omega^{\theta_m}_k\Bigg)^{-1}||R||^2_{0,\Omega},
\end{align*}
with the aid of (\ref{gs85}) and $-1<\omega_k^{\theta_r}<0$ for $k\neq0$.
Further, letting $\tau<1$ and using Lemma \ref{le05}, we have
\begin{align*}
   \frac{sa_m}{\tau^{\theta_m}}||\vartheta^n||^2_{0,\Omega}\!\!\leq\!\!
    \sum_{r=1}^{s}\frac{a_r}{\tau^{\theta_r}}||\vartheta^n||^2_{0,\Omega}
   \!\!\leq\!\! C_e n^{\theta_m}\Gamma(1-\theta_m)||R||^2_{0,\Omega}.
\end{align*}
After simple operations, we obtain
\begin{align*}
   ||\vartheta^n||^2_{0,\Omega}\leq \frac{C_e T^{\theta_m}\Gamma(1-\theta_m)}{sa_m}||R||^2_{0,\Omega}\leq C\tau^2.
\end{align*}

On the other hand, if $\theta_m=1$, then $\theta_1=\theta_2=\cdots=\theta_s=1$
and the above inequality fails to evaluate the error estimate because $\Gamma(1-\theta_m)$ is infinity
when $\theta_m$ is infinitely close to 1. Therefore, other form is needed to be found. Supposing
\begin{equation}\label{xzz06}
  \sum_{r=1}^{s}\frac{a_r}{\tau}||\vartheta^\ell||^2_{0,\Omega} \leq C_e \ell||R||^2_{0,\Omega},
\end{equation}
with $\ell=1,2,\cdots,n-1$, and noticing (\ref{xzz06}) is correct for $n=1$, we check
by (\ref{gs85}) that (\ref{xzz06}) holds for $\ell=n$, i.e.,
\begin{align*}
   & \sum_{r=1}^{s}\frac{a_r}{\tau}||\vartheta^n||^2_{0,\Omega} \\
   &\qquad \leq-\sum_{r=1}^{s}\frac{a_r}{\tau} \sum_{k=1}^{n-1}\omega^{\theta_r}_k
          ||\vartheta^{n-k}||^2_{0,\Omega}+C_e||R||^2_{0,\Omega},\\
   &\qquad \leq-C_e\sum_{k=1}^{n-1}\omega^{\theta_m}_k(n-k)||R||^2_{0,\Omega}+C_e||R||^2_{0,\Omega}\\
   &\qquad \leq C_e\Bigg(1-\sum_{k=0}^{n-1}\omega^{\theta_m}_k\Bigg)(n-k)||R||^2_{0,\Omega}+C_e||R||^2_{0,\Omega}\\
   &\qquad \leq C_e(n-1)||R||^2_{0,\Omega}+C_e||R||^2_{0,\Omega}\\
   &\qquad = C_en||R||^2_{0,\Omega},
\end{align*}
which further yields
\begin{align*}
   ||\vartheta^n||^2_{0,\Omega}\leq \frac{C_e T}{sa_m}||R||^2_{0,\Omega} \leq C\tau^2,
\end{align*}
and this finally completes the proof.
\end{proof}

\section{Spatial discretization by RBFs-based DQ method} \label{s4}

\subsection{Discretization of fractional derivatives in space} \label{s3}
\indent

{\color{red}Supposing $u(\textbf{x})\in C^m(\Omega)$, $m\in \mathbb{Z}^+$,} we have the following DQ formula \cite{dq40}:
\begin{align}\label{xxzz01}
    u^{(m)}_{x_l}(\textbf{x}_i,t)\approx\sum\limits_{j=0}^M {w_{ij}^{m_l}u(\textbf{x}_j,t)},\ \ i=0,1,\ldots, M,
\end{align}
where $1\leq l\leq d$, $w_{ij}^{m_l}$, $i,j=0,1,\ldots,M$, are the weighted coefficients. Realizing its essence
and a function in the linear space $V_h$ spanned by
$\{\phi_k(\textbf{x})\}_{k=0}^M$ can be approximately expanded
as a weighted sum of $\{\phi_k(\textbf{x})\}_{k=0}^M$, i.e.,
$u(\textbf{x},t)\approx\sum_{k=0}^{M}\delta_k(t)\phi_k(\textbf{x})$, \ $\textbf{x}\in \Omega\subset\mathbb{R}^d$.
Then for $\frac{\partial^{\alpha}_+ u(\textbf{x},t)}{\partial x^{\alpha}_{l,+}}$,
we therefore raise the following DQ formulas for the left-hand fractional derivatives:
\begin{align}\label{gs09}
    \frac{\partial^{\alpha}_+ u(\textbf{x}_i,t)}{\partial x^{\alpha}_{l,+}} &\approx\sum_{k=0}^M\delta_k(t) \frac{\partial^{\alpha}_+
        \phi_k(\textbf{x}_i)}{\partial x^{\alpha}_{l,+}} \\
       &=\sum_{k=0}^M\delta_k(t)\sum\limits_{j=0}^M {w_{ij}^{\alpha^+_l}
       \phi_k(\textbf{x}_j)}\approx \sum\limits_{j=0}^M {w_{ij}^{\alpha^+_l}u(\textbf{x}_j,t)},\nonumber
\end{align}
in which, $w_{ij}^{\alpha^+_l}$, $i,j=0,1,\ldots,M$, fulfill
\begin{align}
\frac{\partial^{\alpha}_+ \phi_k(\textbf{x}_i)}{\partial x^{\alpha}_{l,+}}
     =\sum\limits_{j=0}^M {w_{ij}^{\alpha^+_l}\phi_k(\textbf{x}_j)},\ \ i,k=0,1,\ldots, M,\label{gs12}
\end{align}
with $\{\phi_k(\textbf{x})\}_{k=0}^M$ being the test functions. 
The ones for the fractional derivatives {\color{red}in other coordinate directions can be defined analogously}.

Since the interpolation matrices of Inverse Multiquadrics, Inverse Quadratics {\color{red}and Gaussians are
invertible,} eliminate the polynomial term in Eq. (\ref{gs05}) and let $\phi_k(\textbf{x})=\varphi(r_{k})$ in
Eq. (\ref{gs12}). As for Multiquadrics, we employ
\begin{equation*}
u(\textbf{x},t)\approx\sum_{i=0}^{M}\lambda_i(t)\varphi(r_i)+\mu_{1}(t).
\end{equation*}
Besides, from Eq. (\ref{gs06}), it follows that
$\lambda_0(t)=-\sum_{i=1}^{M}\lambda_i(t)$, which leads to
\begin{equation*}
    u(\textbf{x},t)\approx\sum_{i=1}^{M}\lambda_i(t)\big\{\varphi(r_i)-\varphi(r_0)\big\}+\mu_{1}(t).
\end{equation*}

By careful observation, we regard $\phi_0(\textbf{x})=1$, $\phi_i(\textbf{x})=\varphi(r_i)-\varphi(r_0)$,
$i=1,2,\ldots,M$, as test functions. {\color{red}After these discussions,} we can get
the weighted coefficients via a linear system of $M+1$ equations
resulting from the above equations for the given $\{\textbf{x}_i\}_{i=0}^M$.


\subsection{Fully discrete RBFs-based DQ scheme}
\indent

{\color{red}In this subsection,  we develop} a fully discrete DQ scheme for the multi-term TSFPDEs, which
utilizes the operator $\mathscr{A}^{\theta}_q$  to handle the Caputo derivative in time and the RBFs-based
DQ formulas to discretize the fractional derivatives in space.
Let $\{\textbf{x}_i\}_{i=0}^{M}$ be a set of nodes in $\Omega\subset \mathbb{R}^d$.
Replacing the space-fractional derivatives by RBFs-based DQ formulas in Eq. (\ref{gs28}), we have
\begin{equation}\label{gs280}
\begin{small}
\begin{aligned}
  &\sum_{r=1}^{s}a_r\mathscr{A}^{\theta_r}_qu(\textbf{x}_i,t_n)-\sum_{l=1}^d
    \varepsilon^+_{\alpha_l}(\textbf{x}_i)\sum\limits_{j=0}^M {w_{ij}^{\alpha_l^+}u(\textbf{x}_j,t_n)} \\
  &  -\sum_{l=1}^d\varepsilon^-_{\alpha_l}(\textbf{x}_i)\sum\limits_{j=0}^M {w_{ij}^{\alpha_l^-}u(\textbf{x}_j,t_n)}
       =f(\textbf{x}_i,t_n)+\mathscr{O}(\tau^q).
\end{aligned}
\end{small}
\end{equation}
where $i=0,1,\ldots,M$. Omitting $\mathscr{O}(\tau^q)$ and enforcing Eq. (\ref{gs280})
to be exactly ture at $\{\textbf{x}_i\}_{i=0}^{M}$, we obtain the following fully discrete DQ scheme:
\begin{align}\label{gs29}
\left\{
     \begin{array}{lll}
    &\displaystyle\sum_{r=1}^{s}\frac{a_r\omega^{q,\theta_r}_0}{\tau^{\theta_r}}u(\textbf{x}_i,t_{n})
      -\sum_{l=1}^d\varepsilon^+_{\alpha_l}(\textbf{x}_i)\sum\limits_{j=0}^M {w_{ij}^{\alpha_l^+}u(\textbf{x}_j,t_n)}\\
     & \displaystyle\quad -\sum_{l=1}^d\varepsilon^-_{\alpha_l}(\textbf{x}_i)\sum\limits_{j=0}^M {w_{ij}^{\alpha_l^-}u(\textbf{x}_j,t_n)}
      =f(\textbf{x}_i,t_n) \\
     & \displaystyle\quad-\sum_{r=1}^{s}\frac{a_r}{\tau^{\theta_r}}\sum_{k=1}^{n-1}\omega^{q,\theta_r}_ku(\textbf{x}_i,t_{n-k}) \\
     & \displaystyle\quad +\sum_{r=1}^{s}\frac{a_r}{\tau^{\theta_r}}\sum_{k=0}^{n-1}\omega^{q,\theta_r}_ku(\textbf{x}_i,0),
       \quad (\textbf{x}_i;t_n)\in\Omega\times(0,T],\\
   & u(\textbf{x}_i,0)=u_0(\textbf{x}_i),\quad \textbf{x}_i\in\Omega, \\
   & u(\textbf{x}_i,t_n)=g(\textbf{x}_i,t_n), \quad (\textbf{x}_i;t_n)\in\partial\Omega\times(0,T],
             \end{array}
        \right.
\end{align}
where $i=0,1,\ldots,M$ and $n=1,2,\ldots,N$.

For the ease of expression, we employ the notations $U_i^n=u(\textbf{x}_i,t_{n})$, $\varepsilon^\pm_{\alpha_l,i}=\varepsilon^\pm_{\alpha_l}(\textbf{x}_i)$,
$f_i^n=f(\textbf{x}_i,t_{n})$, $g_i^n=g(\textbf{x}_i,t_n)$, and $\textbf{U}^n$, $\textbf{F}^n$, $\textbf{g}^n$ are the
column vectors consisting of $U^n_i$, $f^n_i$ and $g^n_i$ in the ascending order of subscript $i$, respectively. Also, we adopt
$\boldsymbol{\varepsilon}^\pm_{\alpha_l}=diag(\varepsilon^\pm_{\alpha_l,0},\varepsilon^\pm_{\alpha_l,1},
\ldots,\varepsilon^\pm_{\alpha_l,M})$. Reforming Eqs. (\ref{gs29}) in a matrix-vector form finally leads to
\begin{align}\label{gs30}
\left\{
     \begin{array}{lll}
    \displaystyle\sum_{r=1}^{s}\frac{a_r\omega^{q,\theta_r}_0}{\tau^{\theta_r}}\textbf{U}^n
      -\sum_{l=1}^d\boldsymbol{\varepsilon}^+_{\alpha_l} \textbf{W}_{\alpha_l}^+\textbf{U}^n
      -\sum_{l=1}^d\boldsymbol{\varepsilon}^-_{\alpha_l}\textbf{W}_{\alpha_l}^-\textbf{U}^n\\
    \displaystyle\ {\color{red}=\textbf{F}^n\!-\!\sum_{r=1}^{s}\frac{a_r}{\tau^{\theta_r}}\sum_{k=1}^{n-1}\omega^{q,\theta_r}_k\textbf{U}^{n-k}
       \!+\!\sum_{r=1}^{s}\frac{a_r}{\tau^{\theta_r}}\sum_{k=0}^{n-1}\omega^{q,\theta_r}_k\textbf{U}^0, }\\
    U^n_i=g^n_i, \quad \textrm{for}\ \ \textbf{x}_i \in \partial\Omega,
             \end{array}
        \right.
\end{align}
where $n=0,1,\ldots,N$ and $\textbf{W}_{\alpha_l}^\pm$ are the weighted coefficient matrices, given by
\begin{align*}
\textbf{W}_{\alpha_l}^\pm=\left( \begin{array}{cccc}
w^{\alpha_l^\pm}_{00}&w^{\alpha_l^\pm}_{01} &\cdots &w^{\alpha_l^\pm}_{0,M}\\
w^{\alpha_l^\pm}_{10}&w^{\alpha_l^\pm}_{11} &\cdots &w^{\alpha_l^\pm}_{1,M}  \\
 \vdots &\vdots &\ddots&\vdots \\
w^{\alpha_l^\pm}_{M,0}&w^{\alpha_l^\pm}_{M,1}&\cdots &w^{\alpha_l^\pm}_{M,M}
\end{array} \right), \quad l=1,2,\ldots,d.
\end{align*}

In addition, before we can get the numerical solutions, another problem having to be
addressed is how to compute the fractional derivatives of RBFs and determine the corresponding integral paths.
It is acknowledged that the fractional derivative of a general function
like RBFs is quite difficult to compute. We will be engaged in this point.
For a node $\textbf{x}_i\in \Omega$ and the RBF $\varphi(r_i)$, letting
$\xi=x_l-\big(x_l-X_{l,L}\big)(1+\zeta)/2$, one has
\begin{align*}
\begin{small}
\begin{aligned}
\frac{\partial^{\alpha}_+\varphi_k(r_i)}{\partial x^{\alpha}_{l,+}}&=\frac{1}{\Gamma(2-\alpha)}
     \Bigg(\frac{x_l-X_{l,L}}{2}\Bigg)^{2-\alpha}\int^1_{-1}\Bigg\{(1+\zeta)^{1-\alpha} \\
    &\quad \cdot\frac{\partial^2\varphi_k\big(r_i|_{x_l=\xi}\big)}{\partial \xi^2}\bigg|_{\xi=x_l-(x_l-X_{l,L})(1+\zeta)/2}\Bigg\}d\zeta,
\end{aligned}
\end{small}
\end{align*}
where $X_{l,L}=\min\big\{\eta_l: (\eta_1,\eta_2,\ldots,\eta_d),\eta_i=x_i,i\neq l\big\}$.
Then, it can be tackled by Gauss-Jacobi quadrature formula \cite{da51}.
The line segment along the $x_l$-coordinate axis from the left-side of boundary to the node $\textbf{x}_i$
is called the integral path, which need to be determined before the calculation of fractional derivatives.

{\color{red}Assume the space $V_h$ is spanned by the RBFs used as test functions,  then all functions
in $V_h$ can be approximated by RBFs interpolation. If Eq. (\ref{gs12}) is satisfied, we can
examine Eq. (\ref{gs09}) by the linearity of fractional calculus, i.e., when the RBFs-based DQ formula is applied in the
discretization of a fractional PDE in space, the solution can be effectively approximated by Eq. (\ref{gs09}),
and therefore the space discretized method is convergent. }
By solving the algebraic equations resulting from the RBFs-based DQ scheme (\ref{gs30}),
we obtain the desirable solutions for the considered problems at each time layer through an iterative procedure.
From the constructive process of the scheme, we observe that the treatment of boundary conditions
is pretty simple and no special treatment is required, which is one of the advantages of our approach.
{\color{red}
\begin{remark}
The other RBFs-based methods like Kansa's method require to expand the unknown function into the form of
RBFs interpolation, then put it into the considered equation and enforce this equation holds at collocation
nodes so as to determine the RBFs weights. Our DQ method is not pure RBFs collocation method since
the numerical solution at the given nodes are directly obtained by the scheme (\ref{gs30}). We do not
have to compute these RBFs weights like Kansa's method and do not need to rearrange the function in the form of
RBFs interpolation at last either. It is actually a simple algorithm, which not only has the advantages of high accuracy,
small computation amount, but only can easily be applied to the nonlinear fractional problems.
\end{remark} }

\section{Illustrative examples} \label{s5}
\indent

{\color{red}
Let $u^n_{i}$, $U^n_{i}$ be the analytical} and numerical solutions at the time level $n$. Denoting
\begin{align*}
  &||u-U||_{L^2}=\Bigg[h^*\sum_{i=0}^M\big(u_{i}^{n}-U_{i}^{n}\big)^2\Bigg]^{\frac{1}{2}},\ \ h^*=\frac{1}{M+1},\\
  &||u-U||_{L^\infty}=\max\limits_{0\leq i\leq M}\big|u_{i}^n-U_{i}^n\big|,
\end{align*}
some illustrative examples will be performed to reveal the validity and convergence of the proposed
RBFs-based DQ method on different domains. We test our codes by choosing the values of {\color{red}the shape parameter $c$ with reference
to the formula $c=\nu/(M+1)^{\sigma/4}$, where $\nu>0$, $\sigma\in \mathbb{Z}^+$ are the user-choosable parameters.}
We compute the convergent rates by
\begin{align*}
Cov.\ rate=\left\{
\begin{aligned}
&\frac{\log_2\Big(||u-U||_{L^\nu}/||u-U||^*_{L^\nu}\Big)}{\log_2(\tau_2/\tau_1)},\quad \textrm{in time},\\
&\frac{\log_2\Big(||u-U||_{L^\nu}/||u-U||^*_{L^\nu}\Big)}{\log_2(M_2/M_1)},\quad \textrm{in space},
\end{aligned}
\right.
\end{align*}
where $\nu=2,\ \infty$, $\tau_1$, $\tau_2$, $M_1$, $M_2$ are the temporal stepsizes and nodal parameters,
$||u-U||_{L^\nu}$, $||u-U||_{L^\nu}^*$ are the numerical errors, which correspond to $\tau_1$, $M_1$ and
$\tau_2$, $M_2$, respectively. $\tau_1$, $\tau_2$, $M_1$, and $M_2$ should satisfy the conditions $\tau_1\neq \tau_2$ and $M_1\neq M_2$.

{\color{red}The nodal distribution has an important influence on the computational results of DQ methods,
and in general, equispaced nodal distribution is often not adopted
because irregular nodal distribution would be more helpful to obtain fast convergent speed and high accuracy \cite{rbf10}.
Here, we tend to adopt the shifted Chebyshev-Gauss-Lobatto nodes $x_i=\big(1-\cos\frac{i\pi}{M}\big)(b-a)/2+a$ as done by \cite{dq40},
$i=0,1,\cdots,M$, for one-dimensional problems, where $\Omega=[a,b]$, since it is commonly used.
For high-dimensional problems, we should avoid using the nodal distribution which
allows possible nodes overlapping because it may result in an ill-conditioned algebraic equations for interpolation and
numerical solutions based on such nodes are instable, i.e., the uniform irregular nodal distribution is perferred.}

{\color{red}In the sequel,} we will numerically study the convergent accuracy of the proposed RBFs-based DQ method in time and space by comparing with
some existing algorithms like FD and FE methods. According to the theoretical analysis, we anticipate the errors yielded by
the DQ method would show a gradual decline as $M$ increases
and the convergent rate in time would strictly be $q$ as $\tau$ decreases. Moreover, all the numerical tests are
carried out by MATLAB R2014a on a personal PC with WIN 7 Pro., AMD Athlon(tm) II $\times2$ 250
3.00 GHz Processor and 4 GB DDR3 RAM.

\subsection{One-dimensional problems}
\noindent
\textbf{Example 6.1.}  To test the validity of the DQ formula, {\color{blue}we firstly discretize
${^C_{-1}}D_x^{\alpha}(x+1)^2$ on $[-1,1]$ by}
Gr\"{u}nwald-Letnikov (GL) difference operator and DQ formula. 
We choose $\alpha=1.5$ and Inverse Multiquadrics as the test functions with {\color{red}the free parameters}
$\nu=1.3$ and $\sigma=1$. The comparison of numerical errors and convergent rates
for these two methods are shown in Table \ref{tab2}. It is clear from this table that the convergent rates of
GL-operator {\color{blue}are nearly 1} throughout the whole computation, while DQ formula has far better performance than GL-operator on aspects of both
convergent rates and mean square errors. We then come to the
conclusion that our DQ formulas {\color{blue}are more efficient} than GL-operator.
{\color{blue}Secondly, we consider the function with
low smoothness and discretize ${^C_{-1}}D_x^{\alpha}(x+1)^\mu$ on $[-1,1]$ by retaking $\alpha=1.1$ and $\nu=0.9$.
As $\mu$ decreases, the regularity of the function becomes weaker and its fractional derivative shows
the singularity at $x=-1$ if $\mu=\alpha$. The convergent results with different $\mu$ are given in Table \ref{tab2b}.
It can be seen that the DQ formula is still effective when $\mu=\alpha$, but with the decreasing of $\mu$,
its convergent speed would decrease gradually.} \\

\begin{table*}[!htb]
\centering
\caption{The comparison of errors for GL-operator and DQ formula when $\alpha=1.5$.}\label{tab2}
{\color{blue}\begin{tabular}{lccccc}
\toprule
\multicolumn{1}{l}{\multirow{2}{0.6cm}{$M$}} &\multicolumn{2}{l}{GL-operator} &\multicolumn{2}{l}{DQ formula} \\
\cline{2-5}& $||u-U||_{L^2}$  &$Cov.\ rate$  &$||u-U||_{L^2}$  & $Cov.\ rate$ \\
\midrule 10    &5.8888e-02  &-    &  1.8198e-01  &-      \\
         15    &4.3139e-02  &0.77 &  3.0272e-02  &4.42   \\
         20    &3.4252e-02  &0.80 &  7.5563e-03  &4.82   \\
         25    &2.8513e-02  &0.82 &  2.4133e-03  &5.12   \\
\bottomrule
\end{tabular}  }
\end{table*}

\begin{table*}[!htb]
\centering
\caption{{\color{blue}The convergent results of DQ formula for ${^C_{-1}}D_x^{1.1}(x+1)^\mu$ with weak regularity.}}\label{tab2b}
{\color{blue}\begin{tabular}{lccccccc}
\toprule
\multicolumn{1}{l}{\multirow{2}{0.6cm}{$M$}} &\multicolumn{2}{c}{$\mu=1.5$} &\multicolumn{2}{c}{$\mu=1.3$}
                &\multicolumn{2}{c}{$\mu=1.1$}\\
\cline{2-7}& $||u-U||_{L^2}$  &$Cov.\ rate$  &$||u-U||_{L^2}$  & $Cov.\ rate$ &$||u-U||_{L^2}$  & $Cov.\ rate$ \\
\midrule 10    &1.7201e-01  &-    &  3.2198e-01  &-    &  6.4756e-01  &-  \\
         15    &9.4085e-02  &1.49 &  2.4504e-01  &0.67 &  6.1528e-01  &0.13  \\
         20    &6.6809e-02  &1.19 &  2.0472e-01  &0.63 &  5.8697e-01  &0.16  \\
         25    &5.8687e-02  &0.58 &  1.8525e-01  &0.45 &  5.6680e-01  &0.16 \\
\bottomrule
\end{tabular}  }
\end{table*}

{\color{red}
\noindent
\textbf{Example 6.2.} Consider the one-dimensional space-\!\! fractional diffusion equation:
\begin{align*}
\left\{
             \begin{array}{lll}
     &\displaystyle\frac{\partial u(x,t)}{\partial t}-\frac{\Gamma(5-\alpha)x^\alpha}{4}\frac{\partial^{\alpha}_+ u(x,t)}{\partial x^{\alpha}_+}
       = -28e^{-t}x^2(2-x)^2\\
     &\quad\displaystyle -8e^{-t}x^2(\alpha(\alpha-7)+3\alpha x) \quad (x;t)\in[0,2]\times(0,T], \\
     &  u(x,0)=4x^2(2-x)^2,\quad x\in[0,2], \\
     &  u(0,t)=u(2,t)=0, \ t\in(0,T],
             \end{array}
        \right.
\end{align*}
with $\alpha=1.8$ and its analytical solution being $u(x,t)=4e^{-t}x^2(2-x)^2$.

We discuss the the influence of nodal distributions on the computational results of DQ method and compare the
errors yielded by the RBFs-based collocation and DQ methods. For this purpose,
selecting Multiquadrics as the test functions and letting $\nu=0.4$, $\sigma=1$ and $q=1$,
we firstly run the algorithm on equispaced nodal distributions and then run the algorithm by shifted
Chebyshev-Gauss-Lobatto nodes with the same discrete parameters. We compare
the convergent results in space at $t=1$ with those obtained by the RBFs collocation method \cite{rbf11},
where $\tau=1.6667\times10^{-3}$, and all the results are presented in Table \ref{tabq1}.

From this table, we observe that the errors computed based on the shifted Chebyshev-Gauss-Lobatto nodes
are smaller than those based on equispaced nodes and the accuracy of RBFs-based DQ method
is higher than that of collocation technique. All of these phenomena tells us that
the nodal distribution has some effect on the numerical results
and our DQ method can be more accurate than RBFs collocation approach. } \\

\begin{table*}[!htb]
\centering
\caption{{\color{red}The comparison of mean-square errors for RBFs-based collocation and DQ methods at $t=1$.}}\label{tabq1}
{\color{red}\begin{tabular}{lccccccc}
\toprule
\multicolumn{1}{l}{\multirow{3}{0.6cm}{$M$}} & \multicolumn{2}{c}{\multirow{2}{4.2cm}{RBFs collocation method \cite{rbf11}}}
                            &\multicolumn{4}{c}{ MQ-based DQ method} \\
 \cline{4-7}     & &    &\multicolumn{2}{l}{ equispaced nodes} &\multicolumn{2}{l}{shifted Chebyshev-Gauss-Lobatto nodes} \\
\cline{2-7}      &$||u-U||_{L^2}$  &$Cov.\ rate$& $||u-U||_{L^2}$  &$Cov.\ rate$  &$||u-U||_{L^2}$  & $Cov.\ rate$ \\
\midrule  3    &1.28      &-    &  1.7163e-01  &-      &4.0201e-01  &- \\
          7    &3.87e-02  &5.05 &  5.0342e-02  &1.45   &1.3893e-01  &1.25 \\
         15    &2.08e-03  &4.22 &  1.9651e-02  &1.23   &9.0354e-03  &3.59 \\
         31    &6.37e-04  &1.71 &  7.4084e-03  &1.34   &4.6974e-04  &4.07 \\
\bottomrule
\end{tabular}  }
\end{table*}

\noindent
\textbf{Example 6.3.} Consider the one-dimensional multi-term TSFDE:
{\color{blue}\begin{align*}
\left\{
             \begin{array}{lll}
     &\displaystyle\sum_{r=1}^{4}a_r{^C_0}D_t^{\theta_r}u(x,t)\!\!-\!\!\frac{\partial^{\alpha}_+ u(x,t)}{\partial x^{\alpha}_+}
       \!\!=\!\! \sum_{r=1}^{4}\frac{a_r\Gamma{(\mu+1)}t^{\mu-\theta_r}x^4}{\Gamma{(\mu+1-\theta_r)}} \\
     &\quad\displaystyle  -\frac{24t^\mu x^{4-\alpha}}{\Gamma(5-\alpha)},  \quad (x;t)\in[0,1]\times(0,T], \\
     &  u(x,0)=0,\quad x\in[0,1], \\
     &  u(0,t)=0, \ u(1,t)=t^\mu, \ t\in(0,T],
             \end{array}
        \right.
\end{align*}
with $a_1=a_2=a_3=a_4=1$, $\theta_1=0.3$, $\theta_2=0.5$, $\theta_3=0.7$, $\theta_4=0.9$ and $\alpha=1.2$.

The analytical solution is $u(x,t)=t^\mu x^4$.} Using Multiquadrics and Inverse Quadratics as the test functions and
choosing {\color{red}the corresponding free parameters} {\color{blue}$\nu=0.1$, $\sigma=1$, $\mu=1$ and $\nu=0.3$, $\sigma=1$,}
respectively, the convergent results in space at $t=0.5$ with
$q=1$, $\tau=5.0\times10^{-5}$ are listed in Table \ref{tab3}.
The comparison of numerical, analytical solutions and the resulting absolute error yielded
by IQ-based DQ method at $t=0.5$ when $M=25$ are plotted in Fig. \ref{fig0}.
As expected, we obviously see that the errors are quite small and the convergent rates are almost up to 4,
which implies that the proposed DQ method achieves high accuracy by using a small number of nodes.
{\color{red}Besides, the convergent rate of IQ-based DQ method seems higher than that of MQ-based DQ method in this problem}
{\color{blue}and the results in the figure show good consistency between the numerical and analytical solutions.}
Furthermore, we examine the convergent rates of MQ-based DQ method
in time with different $q$. For this purpose, {\color{blue}retaking $\nu=0.15$, $\sigma=1$, $\mu=4$, $q=1,2,3,4$, and $M=50$,} the numerical errors and
corresponding convergent rates at $t=0.5$ are all shown in Table \ref{tab4}. From the data of this table, we observe that
the convergent rates are nearly $\mathscr{O}(\tau^q)$ in time,
{\color{red}which are basically} what we anticipated in theory.

\begin{table*}[!htb]
\centering
\caption{The convergent results at $t=0.5$ with $q=1$, $\tau=5.0\times10^{-5}$ and different $M$ for Example 6.3.}\label{tab3}
{\color{blue}\begin{tabular}{cccccc}
\toprule
\multicolumn{1}{l}{\multirow{2}{0.6cm}{$M$}} &\multicolumn{2}{l}{ MQ-based DQ method}
     &\multicolumn{2}{l}{IQ-based DQ method,\ } \\
\cline{2-5}& $||u-U||_{L^2}$  &$Cov.\ rate$  &$||u-U||_{L^2}$  & $Cov.\ rate$ \\
\midrule 10    &  1.4636e-03  &-      &2.3080e-03  &- \\
         15    &  3.9252e-04  &3.25   &5.2202e-04  &3.67 \\
         20    &  1.4515e-04  &3.46   &1.5203e-04  &4.29 \\
         25    &  5.8617e-05  &4.06   &5.1782e-05  &4.83 \\
\bottomrule
\end{tabular} }
\end{table*}

\begin{figure*}[!htb]
\begin{minipage}[t]{0.46\linewidth}
\includegraphics[width=2.5in]{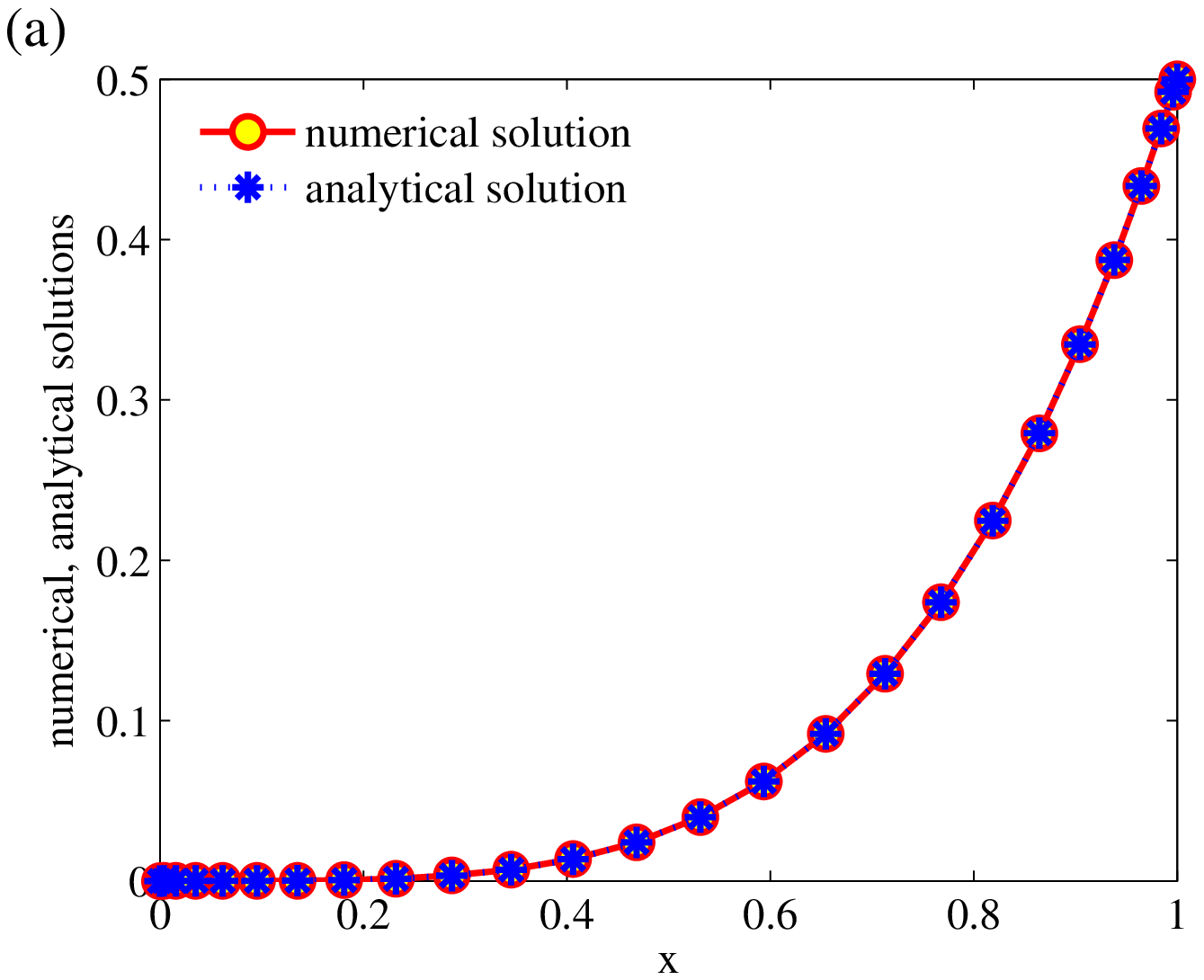}
\end{minipage}
\begin{minipage}[t]{0.46\linewidth}
\includegraphics[width=2.5in]{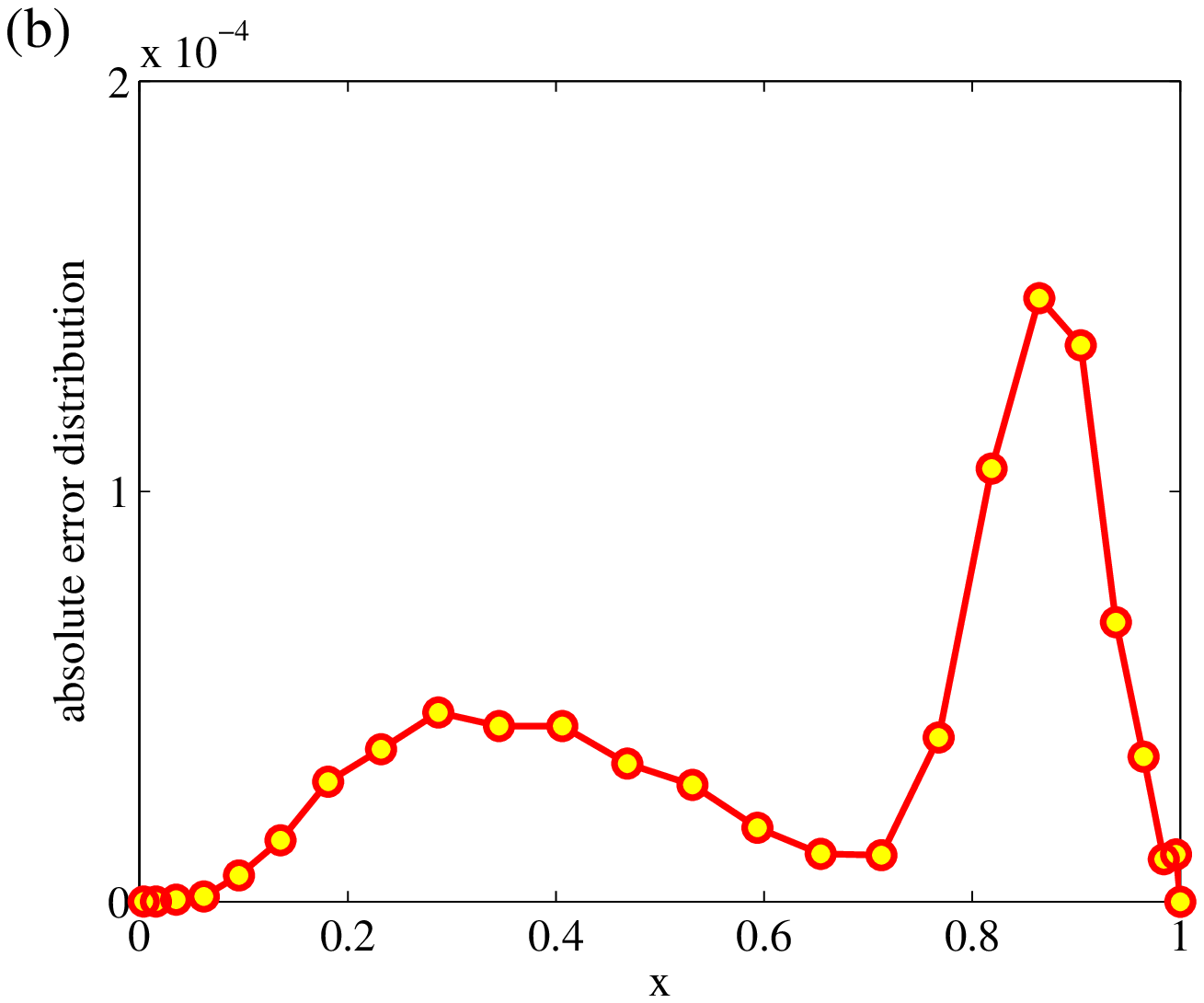}
\end{minipage}
\caption{{\color{blue}The comparison of numerical, analytical solutions and the absolute error by IQ-based DQ method at $t=0.5$.}}\label{fig0}
\end{figure*}

\begin{table*}[!htb]
\centering
\caption{The convergent results at $t=0.5$ of MQ-based DQ method with different $q$, $\tau$ and $M=50$ for Example 6.3.}\label{tab4}
{\color{blue}\begin{tabular}{ccccccc}
\toprule
    & $\tau$ & $||u-U||_{L^2}$  &$Cov.\ rate$  &$||u-U||_{L^\infty}$  & $Cov.\ rate$ \\
\midrule $q=1$
         &1/10    &  3.6590e-03  &-      &7.4663e-03  &- \\
         &1/20    &  1.8618e-03  &0.98   &3.8462e-03  &0.96 \\
         &1/30    &  1.2482e-03  &0.99   &2.5895e-03  &0.98 \\
         &1/40    &  9.3877e-04  &0.99   &1.9517e-03  &0.98 \\
 $q=2$
         &1/10    &  1.0570e-03  &-      &2.1731e-03  &- \\
         &1/20    &  3.0254e-04  &1.80   &6.2150e-04  &1.81 \\
         &1/30    &  1.4056e-04  &1.89   &2.8836e-04  &1.89 \\
         &1/40    &  8.0815e-05  &1.92   &1.6566e-04  &1.93 \\
 $q=3$
         &1/10    &  2.9515e-04  &-      &5.9951e-04  &- \\
         &1/20    &  4.0755e-05  &2.86   &8.1082e-05  &2.89 \\
         &1/30    &  1.2438e-05  &2.93   &2.4644e-05  &2.94 \\
         &1/40    &  5.3228e-06  &2.95   &1.0524e-05  &2.96 \\
 $q=4$
         &1/10    &  4.9886e-05  &-      &1.0407e-04  &- \\
         &1/20    &  2.7325e-06  &4.19   &5.1279e-06  &4.34 \\
         &1/30    &  5.2311e-07  &4.08   &9.7500e-07  &4.09 \\
         &1/40    &  1.6840e-07  &3.94   &3.0658e-07  &4.02 \\
\bottomrule
\end{tabular}  }
\end{table*}

\subsection{Two-dimensional problems}
\noindent
\textbf{Example 6.4.} Consider the two-dimensional single-term TSFPDE:
\begin{align*}
\small
\left\{
             \begin{array}{lll}\displaystyle
\begin{aligned}
 &{^C_0}D^\theta_tu(x,y,t)+\frac{1}{2\cos(\alpha\pi/2)}\Bigg(\frac{\partial^{\alpha}_+ u(x,y,t)}{\partial x^{\alpha}_+}
     +\frac{\partial^{\alpha}_- u(x,y,t)}{\partial x^{\alpha}_-}\Bigg)\\
 &    \quad\displaystyle  +\frac{1}{2\cos(\beta\pi/2)}\Bigg(\frac{\partial^{\beta}_+ u(x,y,t)}{\partial y^{\beta}_+}
       +\frac{\partial^{\beta}_- u(x,y,t)}{\partial y^{\beta}_-}\Bigg) \\
 & \quad    =f(x,y,t),  \quad (x,y;t)\in\Omega\times(0,T],\\
 &     u(x,y,0)=x^2(1-x)^2y^2(1-y)^2,\quad (x,y)\in\Omega, \\
 &     u(x,y,t)=0, \quad (x,y;t)\in\partial\Omega\times(0,T],
\end{aligned}
             \end{array}
        \right.
\end{align*}
{\color{red}on $\Omega=[0,1]\times[0,1]$ with $\alpha=\beta=1.6$} and the source term
\begin{align*}
    f(&x,y,t)=\frac{2t^{2-\theta}}{\Gamma(3-\theta)}x^2(1-x)^2y^2(1-y)^2\\
           &{\color{red}+\frac{(1+t^2)y^2(1-y)^2}{\cos(\alpha\pi/2)}\Bigg\{  \frac{x^{2-\alpha}+(1-x)^{2-\alpha}}{\Gamma(3-\alpha)} } \\
           &{\color{red}   - \frac{6(x^{3-\alpha}+(1-x)^{3-\alpha})}{\Gamma(4-\alpha)}+ \frac{12(x^{4-\alpha}+(1-x)^{4-\alpha})}{\Gamma(5-\alpha)}\Bigg\} } \\
           &{\color{red}+\frac{(1+t^2)x^2(1-x)^2}{\cos(\beta\pi/2)}\Bigg\{  \frac{y^{2-\beta}+(1-y)^{2-\beta}}{\Gamma(3-\beta)}  } \\
           &{\color{red}   - \frac{6(y^{3-\beta}+(1-y)^{3-\beta})}{\Gamma(4-\beta)}+ \frac{12(y^{4-\beta}+(1-y)^{4-\beta})}{\Gamma(5-\beta)}\Bigg\}.  }
\end{align*}

The analytical solution is $u(x,y,t)=(1+t^2)x^2(1-x)^2y^2(1-y)^2$. To show the superiority of our methods as much as possible,
we compare the spatial accuracy of the weighted and shifted GL (WSGL) methods \cite{dq11} ($(p,q)=(1,0)$ for WSGL1 and $(p,q)=(1,-1)$
for WSGL2, $p$, $q$ are weighted coefficients), FD method \cite{dq12}, FE method based on bilinear rectangular
finite elements \cite{da54} and our DQ method based on Inverse Multiquadrics in terms of $||u-U||_{L^2}$.
Here, we note that the structured quadrilateral meshes with the mesh sizes $h=1/4$, $h=1/8$, $h=1/16$ and $h=1/32$ are employed
{\color{red}in the implementation of the above FE method, i.e., the corresponding nodal numbers are $25$, $81$, $289$ and $1089$, respectively.}
Taking $\theta=0.5$, $\nu=8$, $\sigma=q=2$ and $\tau=1.0\times10^{-3}$, the numerical errors at $t=1$
of all the above-mentioned methods are reported in Table \ref{tab5}. From this table, we observe that the proposed
DQ method is beyond all doubt more excellent than WSGL and FD methods in computational accuracy.
Although the error magnitude of DQ method is about the same as bilinear rectangular FE method,
it is certain that our method would be more flexible in implementation and have a less amount of
computation than bilinear rectangular FE method because DQ methods do not require meshes generation
in practical computation, nor is variational principle.

Furthermore, to get more insight into the computational efficiency of the proposed DQ method, we compare the spatial accuracy and CPU
times of FE method based on linear unstructured triangular finite elements \cite{dq17} and our DQ method with the same nodal numbers.
Retaking $\theta=1$, the numerical errors at $t=1$ and the CPU time costs are all tabulated in Table \ref{tab50}, where
the unstructured triangular meshes of the mesh sizes $h\approx1/2$, $h\approx1/4$, $h\approx1/8$ and $h\approx1/12$ are used in FE method
with their nodal numbers exactly being $12$, $37$, $121$ and $199$, respectively.
In Fig. \ref{fig1}, the used triangular meshes with $h\approx1/8$ and $h\approx1/12$ are displayed.
The configuration of 199 nodes used in IMQ-based DQ method and  its corresponding absolute error
of IMQ-based DQ method at $t=1$ are presented in Fig. \ref{fig01}, and
{\color{red}the analytical, numerical solutions are presented in Fig. \ref{figxz3}.}
From these table and figures,
we observe that DQ method achieves better accuracy than FE method does, and
not only that, the CPU time of our DQ method is almost a third of that of FE method in
the best situation, which further manifests that our method causes the total computational amount much less than FE method.
On the other hand, {\color{red}the numerical solution is in good agreement with analytical solution,
and the absolute error reaches a maximum in the center of the square domain
and gradually decreases along the center towards the boundary, which coincides with the shape of the analytical solution.} \\

\begin{table*}[!htb]
\centering
\caption{ The comparison of mean-square errors at $t=1$ for WSGL, FD, FE and DQ methods when $q=2$, $\theta=0.5$ and $\alpha=\beta=1.6$}\label{tab5}
\begin{tabular}{lcccccc}
\toprule
$M$ &WSGL1 \cite{dq11}&WSGL2 \cite{dq11}&FD method \cite{dq12}&FE method \cite{da54}&IMQ-based DQ method \\
\midrule 24   &8.2885e-04   &7.1069e-04  &5.1926e-04   &2.8397e-04  &2.4378e-04 \\
         80   &2.0070e-04   &2.4232e-04  &1.1762e-05   &8.6729e-05  &7.7419e-05 \\
         288  &4.8865e-05   &6.9518e-05  &2.6904e-05   &2.2360e-05  &2.1523e-05 \\
         1088 &1.1904e-05   &1.8377e-05  &6.2224e-06   &5.1208e-06  &6.5475e-06\\
\bottomrule
\end{tabular}
\end{table*}

\begin{table*}[!htb]
\centering
\caption{ The comparison of errors at $t=1$ and CPU times for FE and DQ methods when $q=2$, $\theta=1$ and $\alpha=\beta=1.6$}\label{tab50}
\begin{tabular}{lclcclcc}
\toprule
\multicolumn{1}{l}{\multirow{2}{0.6cm}{$M$}} &\multicolumn{3}{l}{FE method \cite{dq17}}
    &\multicolumn{3}{l}{IMQ-based DQ method} \\
\cline{2-7}& $||u-U||_{L^2}$  &$||u-U||_{L^\infty}$ &CPU times/s  &$||u-U||_{L^2}$  & $||u-U||_{L^\infty}$ & CPU times/s \\
\midrule 11     &  1.4364e-03  &1.9251e-03 & 8.4470  &5.7675e-04  &1.1361e-03 &5.2219 \\
         36     &  3.8724e-04  &4.1695e-04 & 34.0283   &1.1968e-04  &2.6712e-04 &15.1634 \\
         120    &  1.2548e-04  &1.9499e-04 & 138.7742   &2.1899e-05  &4.7802e-05 & 46.0256\\
         198    &  7.8313e-05  &1.2380e-04 & 240.2891   &1.0404e-05   &2.2870e-05 &77.1730 \\
\bottomrule
\end{tabular}
\end{table*}

\begin{figure*}[!htb]
\begin{minipage}[t]{0.43\linewidth}
\includegraphics[width=2.7in]{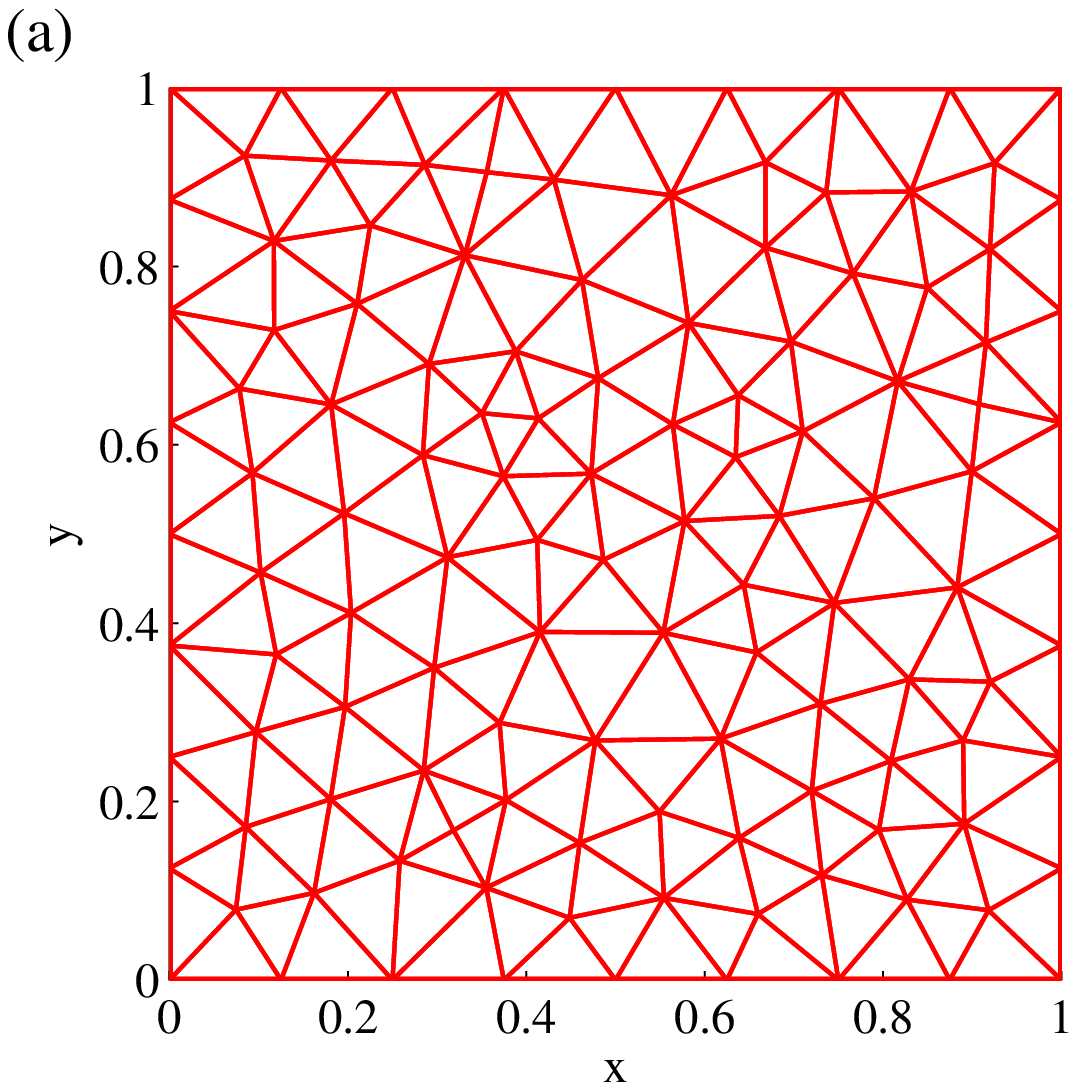}
\end{minipage}
\begin{minipage}[t]{0.43\linewidth}
\includegraphics[width=2.7in]{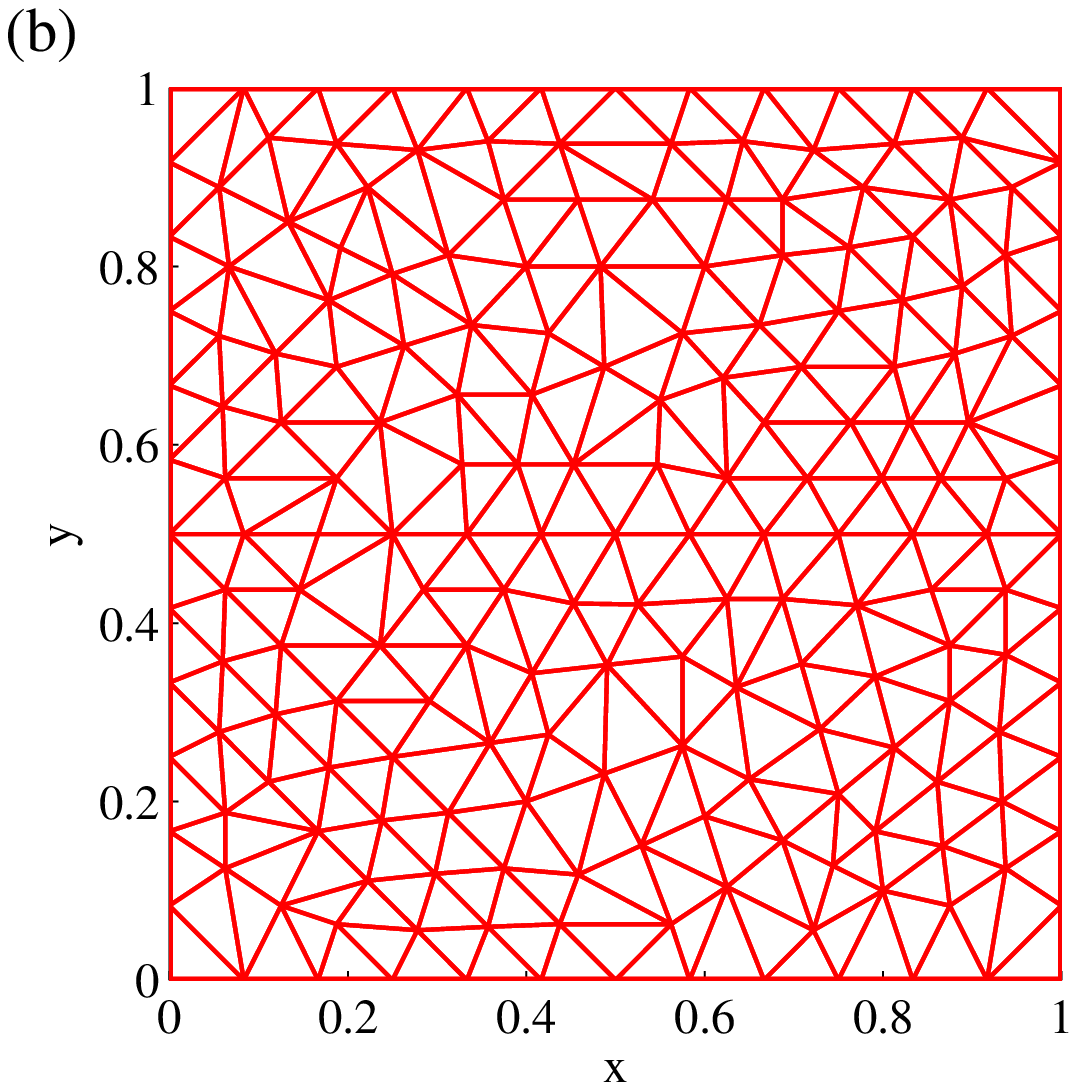}
\end{minipage}
\caption{The unstructured triangular meshes used by FE method with the mesh sizes $h\approx 1/8$ and $h\approx1/12$, respectively.}\label{fig1}
\end{figure*}

\begin{figure*}[!htb]
\begin{minipage}[t]{0.46\linewidth}
\includegraphics[width=2.7in]{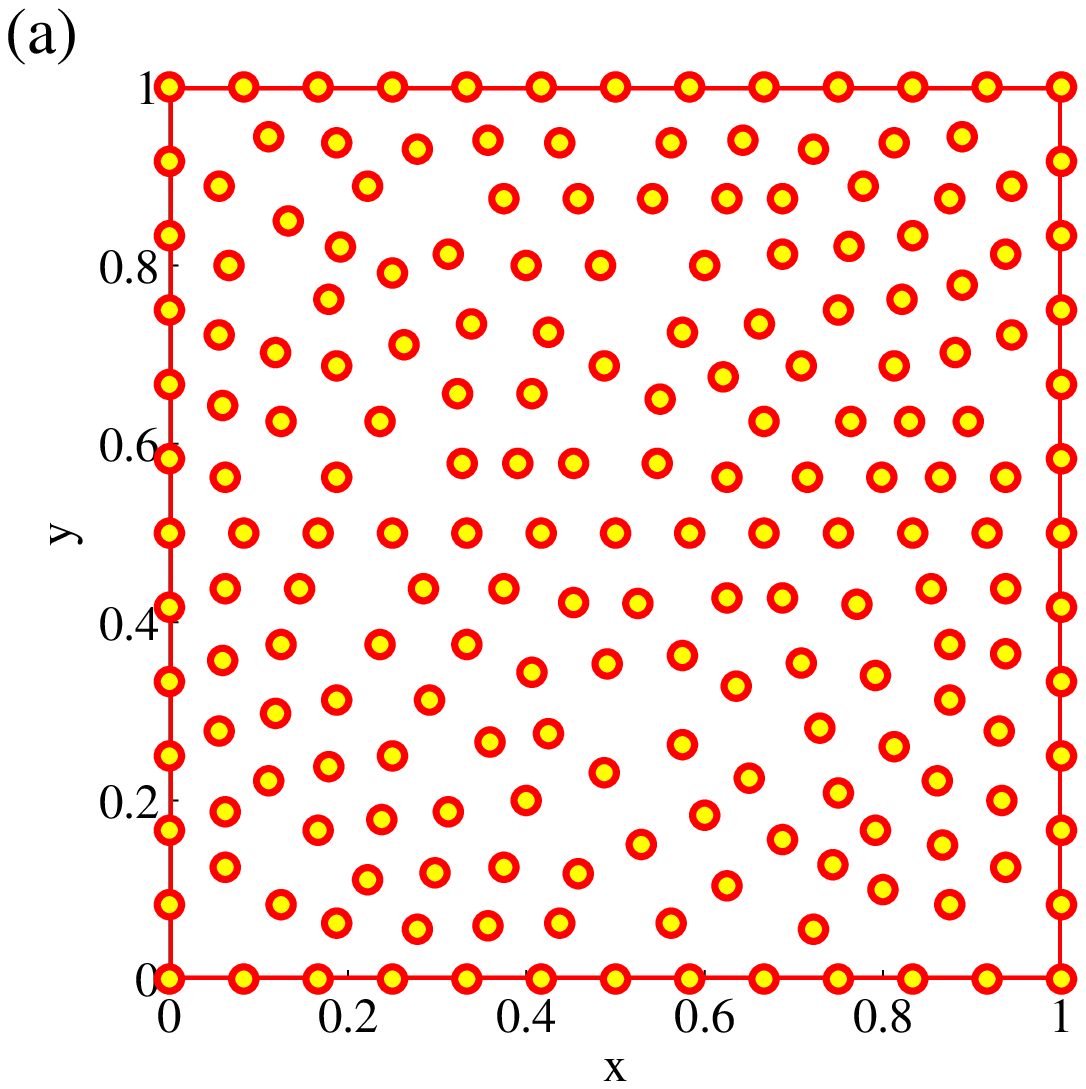}
\end{minipage}
\begin{minipage}[t]{0.46\linewidth}
\includegraphics[width=2.7in]{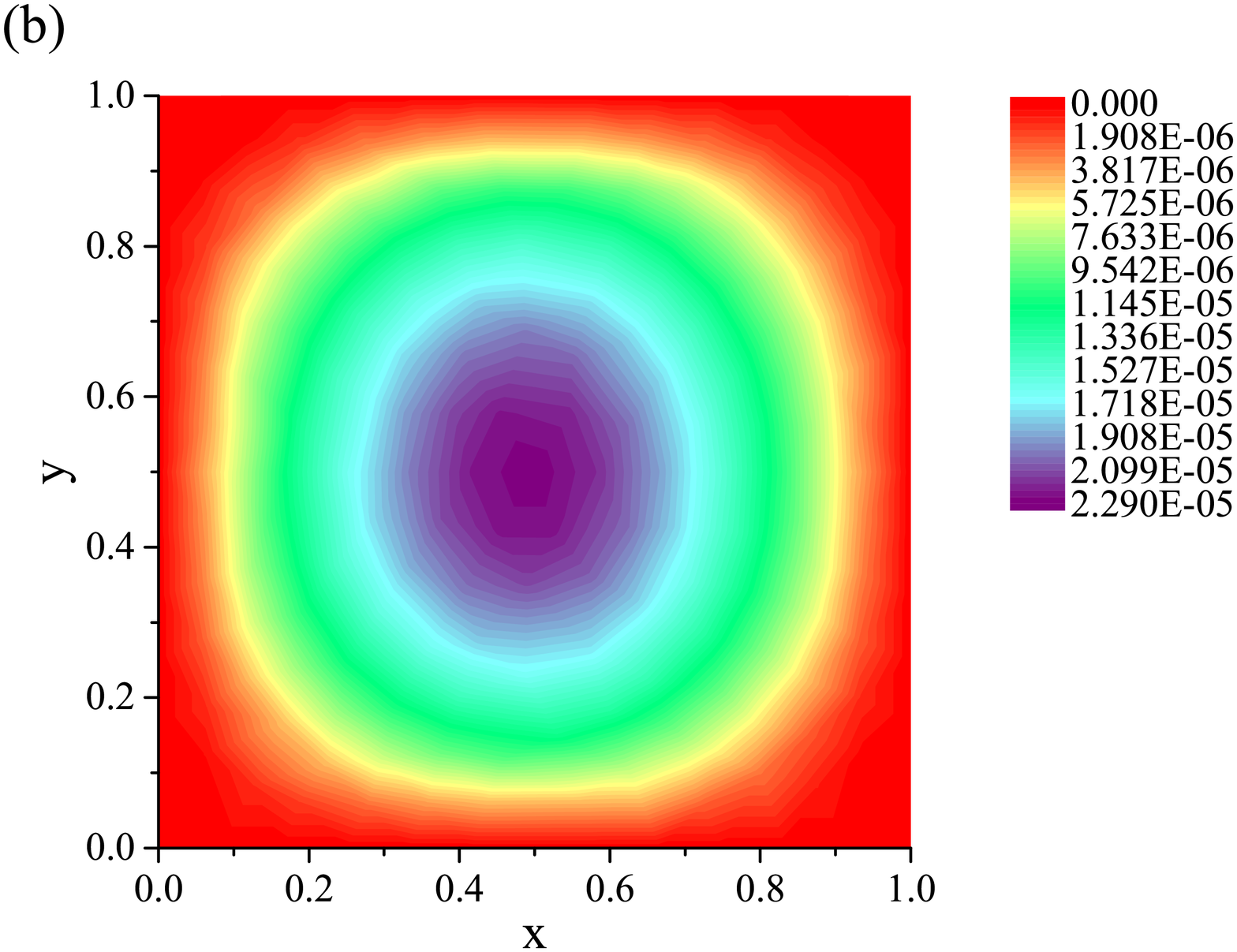}
\end{minipage}
\caption{The configuration of 199 nodes and its absolute error by IMQ-based DQ method at $t=1$.}\label{fig01}
\end{figure*}

\begin{figure*}[!htb]
\begin{minipage}[t]{0.49\linewidth}
\includegraphics[width=3in]{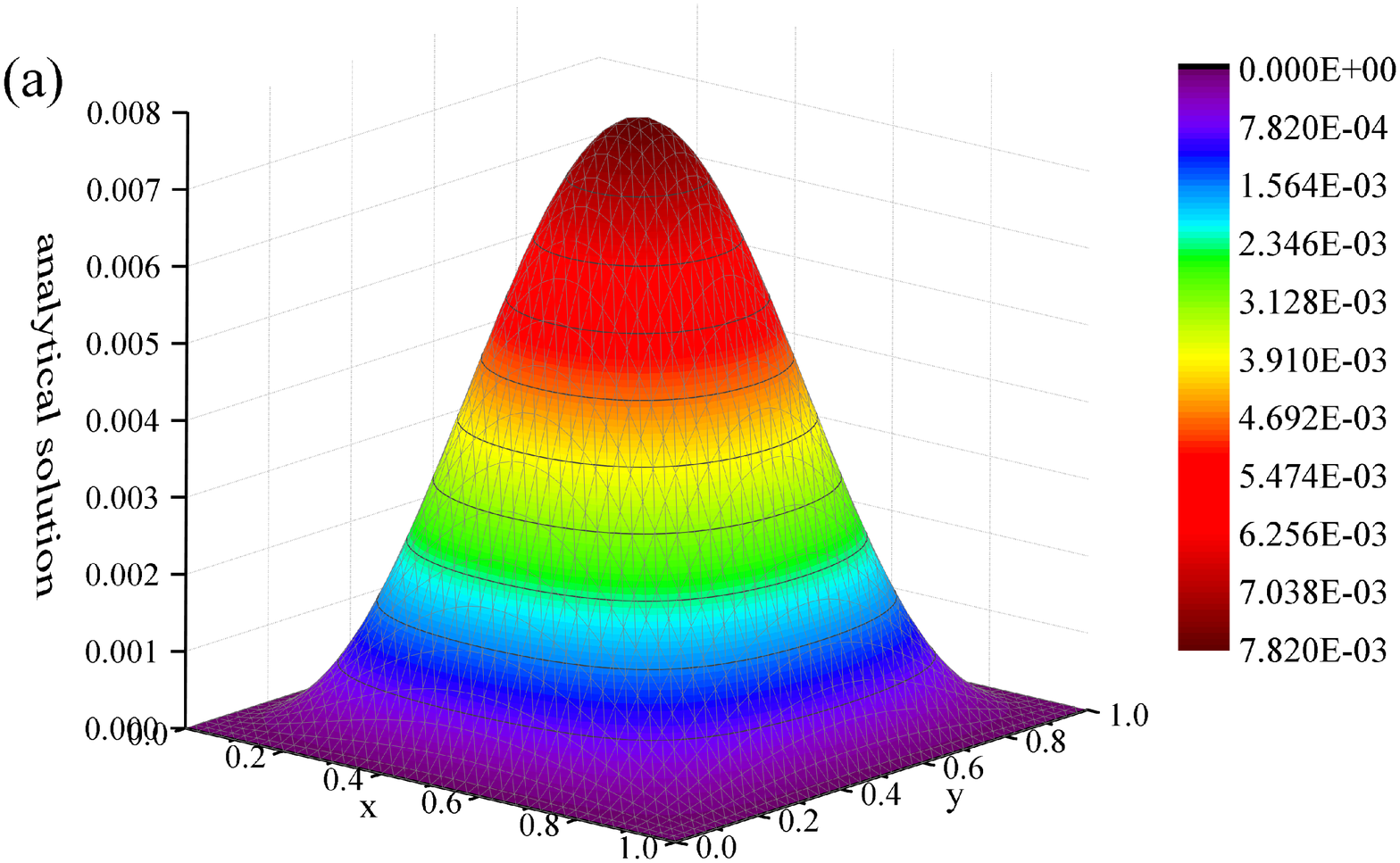}
\end{minipage}
\begin{minipage}[t]{0.49\linewidth}
\includegraphics[width=3in]{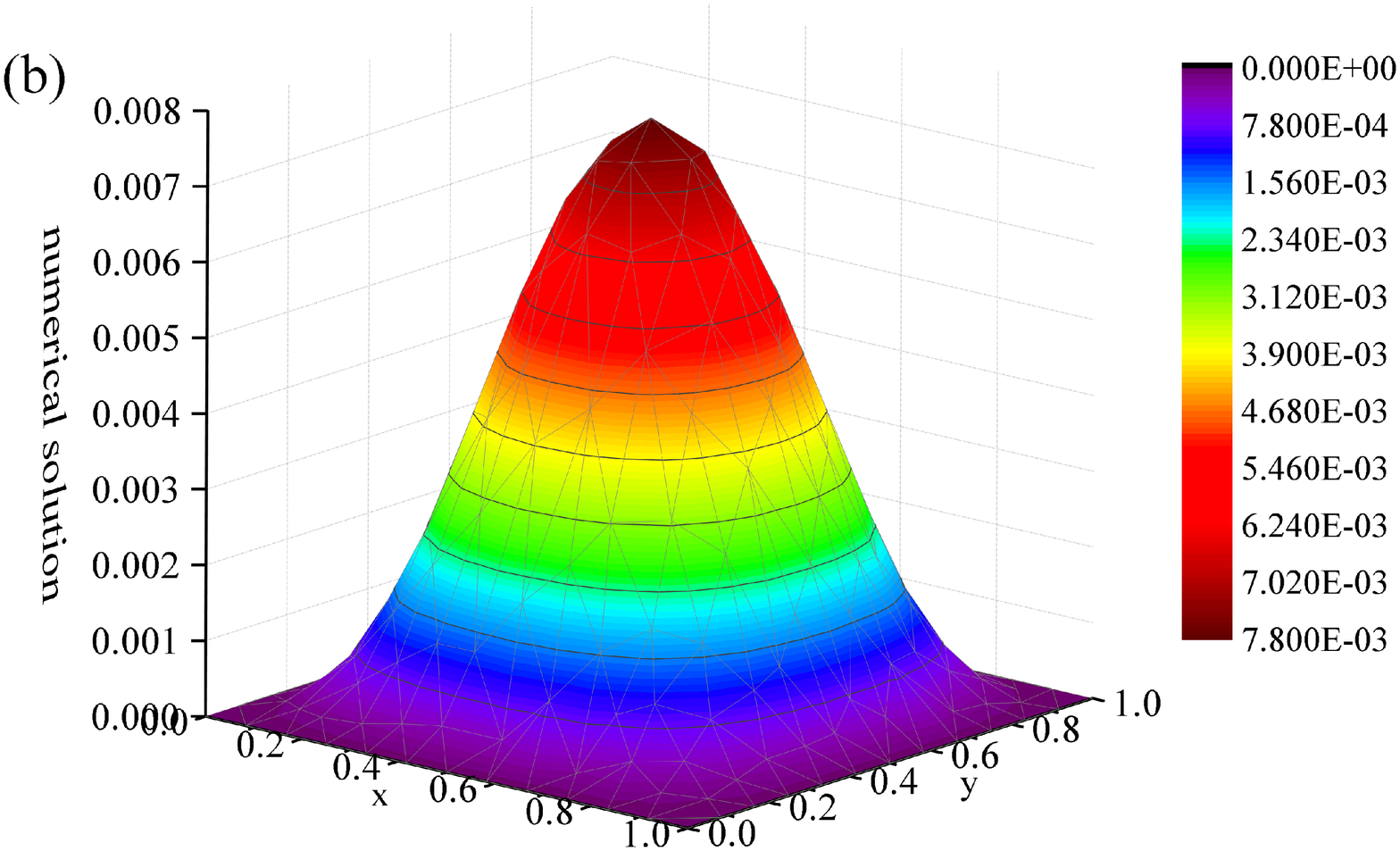}
\end{minipage}
\caption{{\color{red}The analytical solution and the numerical solution by IMQ-based DQ method at $t=1$.}}\label{figxz3}
\end{figure*}

\noindent
\textbf{Example 6.5.} Consider the two-dimensional multi-term TSFPDE:
\begin{align*}
\left\{
             \begin{array}{lll}
     &\displaystyle\sum_{r=1}^{3}a_r{^C_0}D_t^{\theta_r}u(x,y,t)-\frac{x^\alpha}{2}\frac{\partial^{\alpha}_+ u(x,y,t)}{\partial x^{\alpha}_+} \\
     &\displaystyle\quad -\frac{y^\beta}{2}\frac{\partial^{\beta}_+ u(x,y,t)}{\partial y^{\beta}_+}=f(x,y,t),  \quad (x,y;t)\in\Omega\times(0,T],\\
     &  u(x,y,0)=0,\quad (x,y)\in\Omega, \\
     &  u(x,y,t)=t^3x^2y^2, \quad (x,y;t)\in\partial\Omega\times(0,T],
             \end{array}
        \right.
\end{align*}
on the triangular domain $\Omega=\{(x,y)|0\leq x\leq1,0\leq y\leq 1-x\}$ with $a_1=2$, $a_2=0.5$, $a_3=3$, $\theta_1=0.1$,
$\theta_2=0.3$, $\theta_3=0.5$, $\alpha=1.5$, $\beta=1.8$ and the source term
\begin{align*}
   f(x,y,t)&=\sum_{r=1}^{3}\frac{6a_rt^{3-\theta_r}x^2y^2}{\Gamma{(4-\theta_r)}} \\
           &\quad -t^3x^2y^2\Bigg\{ \frac{1}{\Gamma(3-\alpha)}+\frac{1}{\Gamma(3-\beta)}\Bigg\}.
\end{align*}

The analytical solution is $u(x,y,t)=t^3x^2y^2$. We use Multiquadrics and Gaussian as the test functions
and choose {\color{red}the corresponding free parameters} to be $\nu=0.3$, $\sigma=1$ and $\nu=0.5$, $\sigma=1$, respectively.
The convergent results at $t=0.5$ with $q=3$ and $\tau=1.0\times10^{-3}$ are reported in Table \ref{tab6}.
In Fig. \ref{fig2}, we show the used configuration of 152 nodes and its corresponding absolute error of
GA-based DQ method at $t=0.5$. {\color{red}In Fig. \ref{figxz4}, we give a comparison of the analytical and numerical solutions
at $t=0.5$.} It is clear from the above table and figure that
{\color{red}a good convergent property has been observed for
both two DQ methods either in sense of the mean-square error or maximum error as the distances of these nodes decreases.}
{\color{red}The numerical solution is also in perfect agreement with analytical solution.}
Moreover, the absolute error of GA-based DQ method is relatively large around the hypotenuse of this triangular domain
and gradually decreases along the hypotenuse towards its two right-angle sides. {\color{red}
This is possibly caused by the fact that the surface of analytical solution is steep near the hypotenuse while becomes flat
near the right-angle sides since the large error usually occurs at the place where the gradient of true solution is relatively high.}\\

\begin{table*}[!htb]
\centering
\caption{ The convergent results at $t=0.5$ with $q=3$, $\tau=1.0\times10^{-3}$ and different $M$ for Example 6.5}\label{tab6}
\begin{tabular}{lclclc}
\toprule
\multicolumn{1}{l}{\multirow{2}{0.6cm}{$M$}} &\multicolumn{2}{l}{MQ-based DQ method}
    &\multicolumn{2}{l}{GA-based DQ method} \\
\cline{2-5}& $||u-U||_{L^2}$  &$||u-U||_{L^\infty}$  &$||u-U||_{L^2}$  & $||u-U||_{L^\infty}$ \\
\midrule 55     &  5.7083e-05  &1.8789e-04   &8.2862e-05  &3.0100e-04 \\
         79     &  4.4719e-05  &1.6652e-04   &5.9078e-05  &2.2586e-04 \\
         114    &  3.8698e-05  &1.5015e-04   &5.2482e-05  &2.0574e-04 \\
         152    &  2.9558e-05  &1.1855e-04   &3.6037e-05  &1.5203e-04 \\
\bottomrule
\end{tabular}
\end{table*}

\begin{figure*}[!htb]
\begin{minipage}[t]{0.48\linewidth}
\includegraphics[width=3.2in]{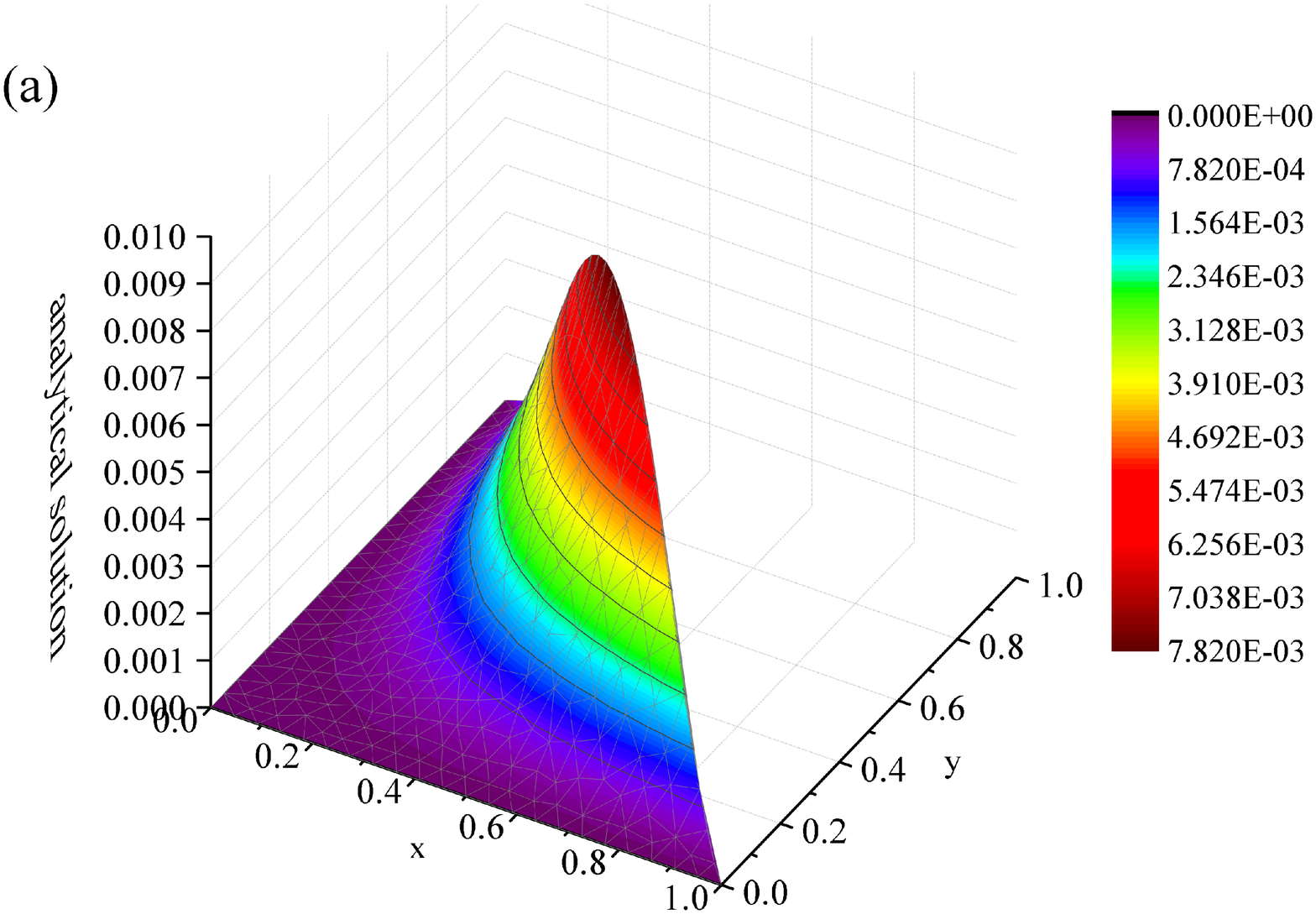}
\end{minipage}
\begin{minipage}[t]{0.48\linewidth}
\includegraphics[width=3.2in]{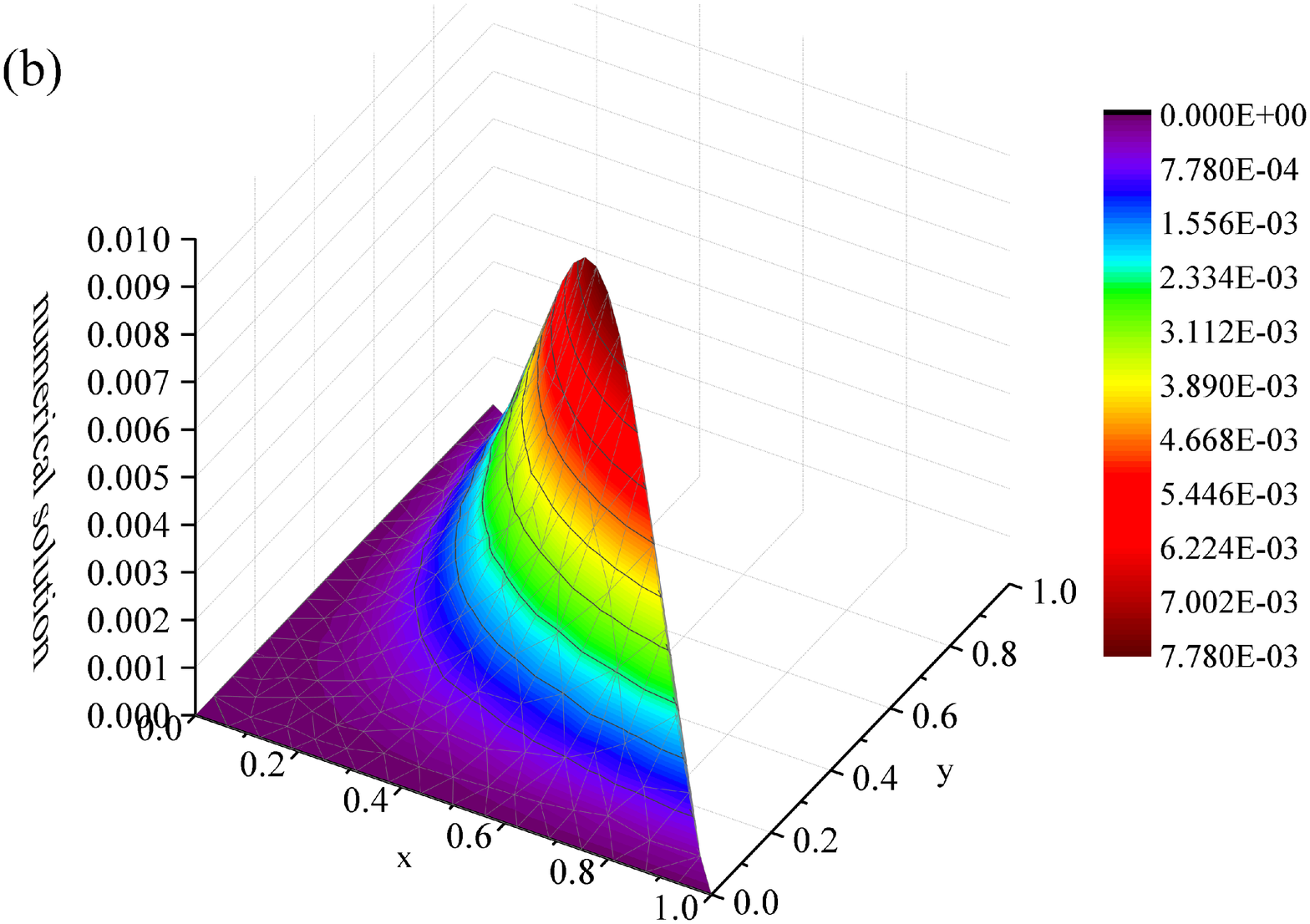}
\end{minipage}
\caption{{\color{red}The analytical solution and the numerical solution by GA-based DQ method at $t=0.5$.}}\label{figxz4}
\end{figure*}

\begin{figure*}[!htb]
\begin{minipage}[t]{0.46\linewidth}
\includegraphics[width=2.7in]{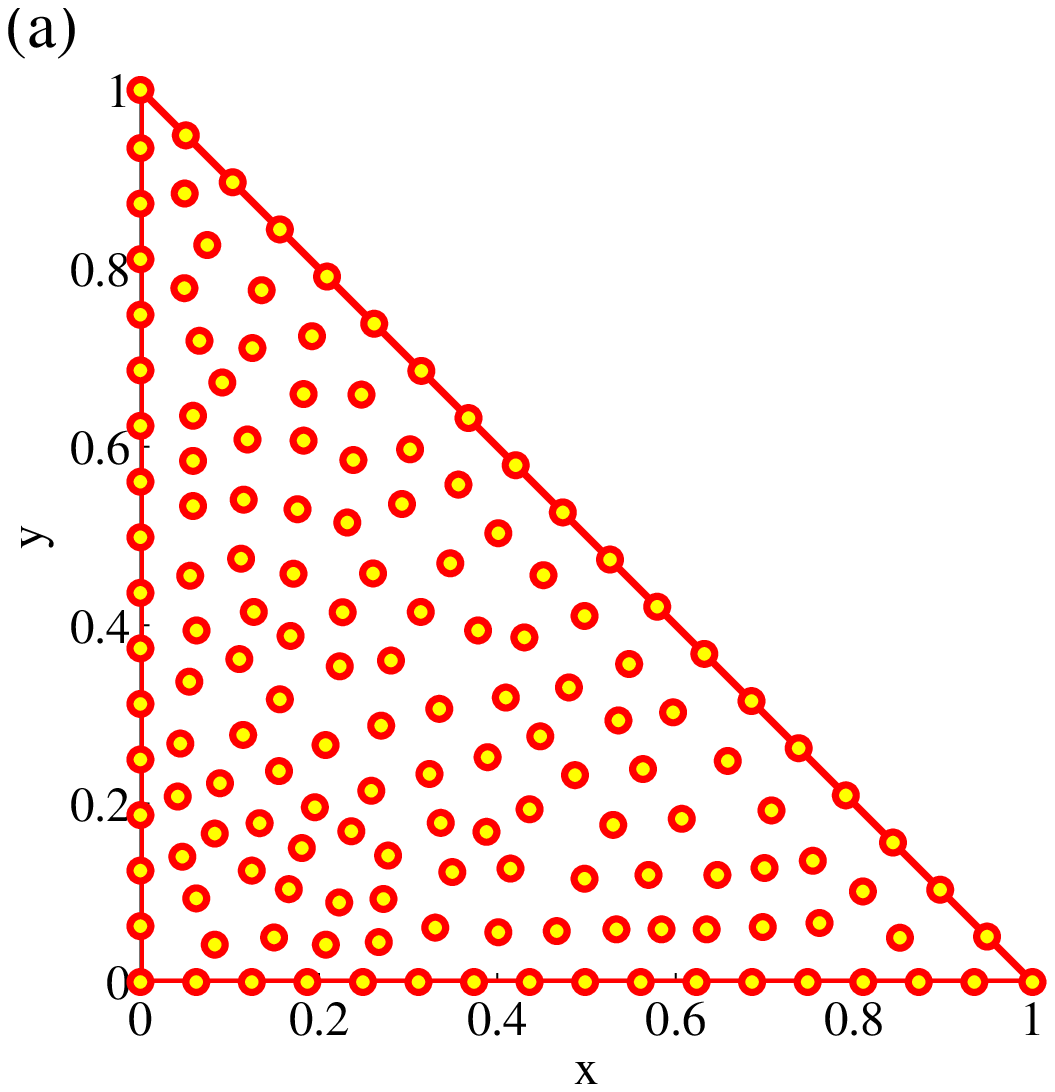}
\end{minipage}
\begin{minipage}[t]{0.46\linewidth}
\includegraphics[width=3.0in]{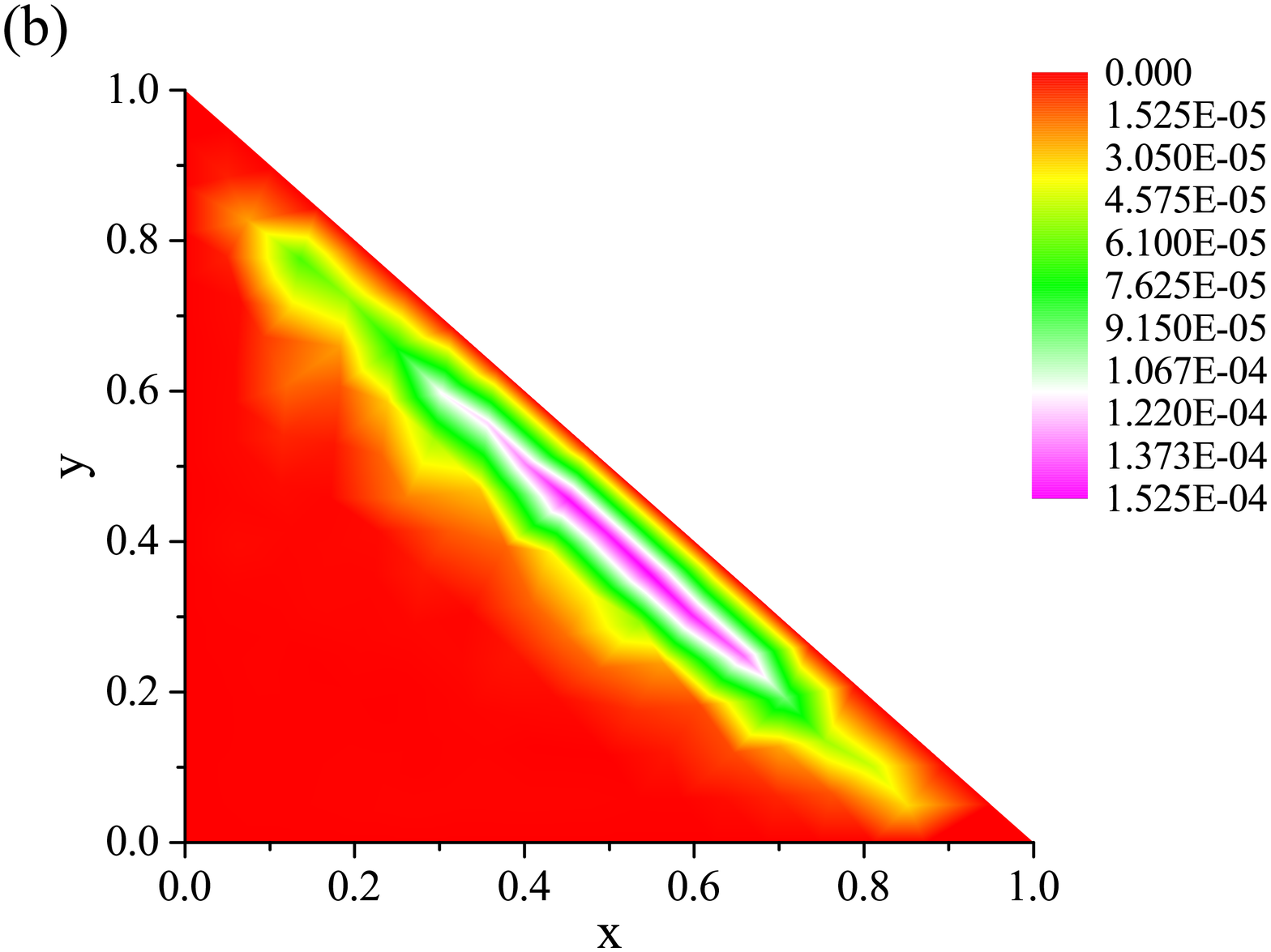}
\end{minipage}
\caption{The configuration of 153 nodes and its absolute error by GA-based DQ method at $t=0.5$.}\label{fig2}
\end{figure*}

\noindent
\textbf{Example 6.6.} Consider the two-dimensional multi-term TSFPDE:
\begin{align*}
\left\{
             \begin{array}{lll}\displaystyle
\begin{aligned}
    & \sum_{r=1}^{3}a_r{^C_0}D_t^{\theta_r}u(x,y,t)+\frac{1}{2\cos(\alpha\pi/2)} \\
    & \quad\displaystyle \cdot\Bigg(\frac{\partial^{\alpha}_+ u(x,y,t)}{\partial x^{\alpha}_+}
     +\frac{\partial^{\alpha}_- u(x,y,t)}{\partial x^{\alpha}_-}\Bigg)\\
    & \quad\displaystyle  +\frac{1}{2\cos(\beta\pi/2)}\Bigg(\frac{\partial^{\beta}_+ u(x,y,t)}{\partial y^{\beta}_+}
       +\frac{\partial^{\beta}_- u(x,y,t)}{\partial y^{\beta}_-}\Bigg)\\
    & \quad =f(x,y,t),  \quad (x,y;t)\in\Omega\times(0,T],\\
    &   u(x,y,0)=(4x^2+y^2-1)^2/10,\quad (x,y)\in\Omega, \\
    &   u(x,y,t)=0, \quad (x,y;t)\in\partial\Omega\times(0,T],
\end{aligned}
             \end{array}
        \right.
\end{align*}
on the elliptical domain $\Omega=\{(x,y)|4x^2+y^2\leq 1\}$ with $a_1=a_2=a_3=1$, $\theta_1=0.6$,
$\theta_2=0.7$, $\theta_3=1$, $\alpha=\beta=1.6$ and the source term
\begin{align*}
   f(x,y,t)&=\sum_{r=1}^{3}\frac{a_rt^{2-\theta_r}}{5\Gamma(3-\theta_r)}(4x^2+y^2-1)^2\\
           &+\frac{(1+t^2)}{20\cos(\alpha\pi/2)}\Bigg\{ (192x^2_l+16y^2-16)\frac{(x-x_l)^{2-\alpha}}{\Gamma(3-\alpha)} \\
           & +\frac{384x_l(x-x_l)^{3-\alpha}}{\Gamma(4-\alpha)}+ \frac{384(x-x_l)^{4-\alpha}}{\Gamma(5-\alpha)} \\
           &+(192x^2_r+16y^2-16)\frac{(x_r-x)^{2-\alpha}}{\Gamma(3-\alpha)} \\
           &+\frac{384x_r(x_r-x)^{3-\alpha}}{\Gamma(4-\alpha)} + \frac{384(x_r-x)^{4-\alpha}}{\Gamma(5-\alpha)}\Bigg\}\\
           &+\frac{(1+t^2)}{20\cos(\beta\pi/2)}\Bigg\{ (192y^2_l+16x^2-16)\frac{(y-y_l)^{2-\beta}}{\Gamma(3-\beta)}\\
           &+\frac{384y_l(y-y_l)^{3-\beta}}{\Gamma(4-\beta)}+ \frac{384(y-y_l)^{4-\beta}}{\Gamma(5-\beta)} \\
           &+(192y^2_r+16x^2-16)\frac{(y_r-y)^{2-\beta}}{\Gamma(3-\beta)}\\
           &+\frac{384y_r(y_r-y)^{3-\beta}}{\Gamma(4-\beta)}+ \frac{384(y_r-y)^{4-\beta}}{\Gamma(5-\beta)}\Bigg\},
\end{align*}
with $x_l=-\sqrt{1-y^2}/2$, $x_r=\sqrt{1-y^2}/2$, $y_l=-\sqrt{1-4x^2}$ and $y_r=\sqrt{1-4x^2}$.

The analytical solution is $u(x,y,t)=\frac{1+t^2}{10}(4x^2+y^2-1)^2$. We compare the spatial accuracy of FE method \cite{da56} and
our DQ methods based on Inverse Quadratics and Gaussians in terms of $||u-U||_{L^2}$, where we choose
{\color{red}the free parameters} $\nu=11$, $\sigma=2$ for Inverse Quadratics while $\nu=4.5$, $\sigma=2$ for Gaussians.
Taking $q=2$ and $\tau=1.0\times10^{-3}$, the numerical errors at $t=1$ of FE and DQ methods are all tabulated in Table \ref{tab7}.
The used configurations of 113, 239 and 413 nodes and their corresponding absolute errors of IQ-based
DQ method at $t=1$ are displayed in Fig. \ref{fig3}. In the computation, $N_e$ is the total number of triangles
of the meshes used by the FE method. In Table \ref{tab7}, it can be seen that $N_e$ is taken to be 70, 468, 1142 and 1738, respectively,
then according to the proportional relationship between the numbers of nodes and triangles
in a triangular mesh (the proportion of numbers of nodes to triangles is about 3:1), we know that the nodal number is bound to be much larger than $N_e$.
Thus, we conclude that our DQ methods generate the same error magnitude
as FE method with less nodes, which shows that our DQ methods are fairly efficient
and equipped with some advantages FE method does not have.
Besides, the absolute error reaches a maximum in the center of this elliptical domain and gradually
decreases along the center towards its boundary, this is because the change of analytical solution
is relatively rapid around the center of this domain and becomes gently in other places.

\begin{table*}[!htb]
\centering
\caption{The comparison of  mean-square errors at $t=1$ for FE and DQ methods when $\theta_1=0.6$,
$\theta_2=0.7$, $\theta_3=1$ and $\alpha=\beta=1.6$.}\label{tab7}
\begin{tabular}{lclclclc}
\toprule
\multicolumn{1}{l}{\multirow{2}{0.6cm}{$N_e$}} &\multicolumn{1}{l}{FE method \cite{da56}}
    &\multicolumn{1}{l}{\multirow{2}{0.6cm}{$M$}} &\multicolumn{2}{l}{IQ-based DQ method}
    &\multicolumn{2}{l}{GA-based DQ method} \\
\cline{2-2}\cline{4-7}
        & $||u-U||_{L^2}$& & $||u-U||_{L^2}$  &$||u-U||_{L^\infty}$  &$||u-U||_{L^2}$  & $||u-U||_{L^\infty}$ \\
\midrule70  &8.6311e-03 &34  &   7.4516e-03  &1.5952e-02   &6.7396e-03  &1.4335e-02 \\
        468 &1.4508e-03 &112 &   1.2389e-03  &3.1296e-03   &1.3671e-03  &3.5605e-03 \\
        1142&5.4919e-04 &238 &   6.6484e-04  &1.4906e-03   &7.3477e-04  &1.7147e-03 \\
        1738&3.6969e-04 &412 &   4.2034e-04  &9.4854e-04   &4.1947e-04  &1.0285e-03 \\
\bottomrule
\end{tabular}
\end{table*}

\begin{figure*}[!htb]
\centering
\begin{minipage}[t]{0.3\linewidth}
\includegraphics[width=1.6in]{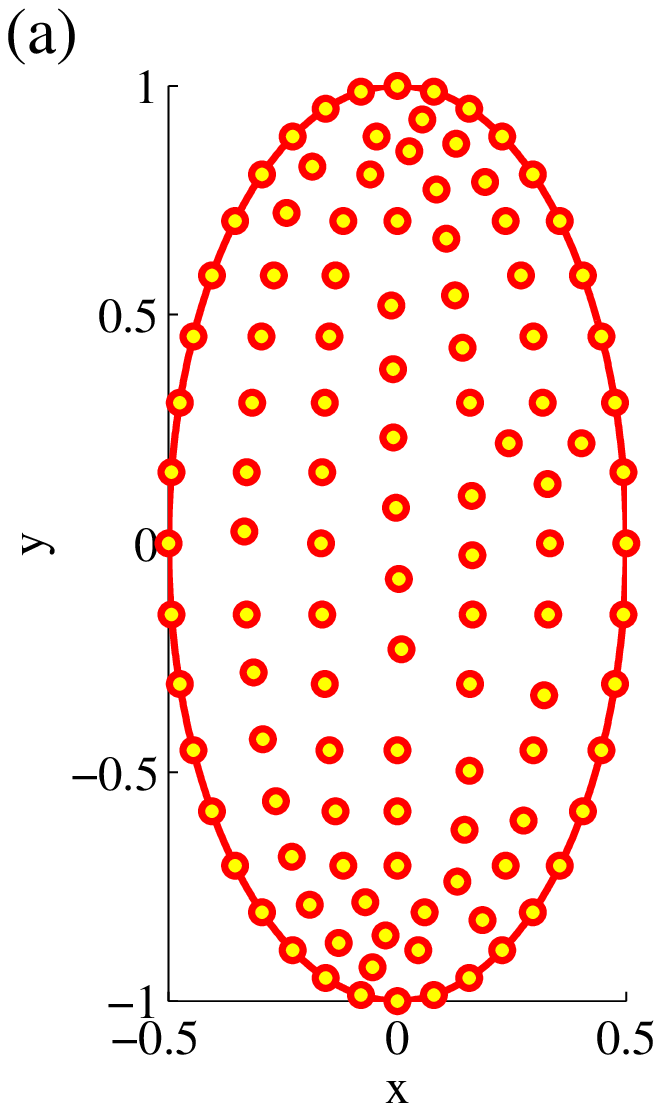}
\end{minipage}
\begin{minipage}[t]{0.3\linewidth}
\includegraphics[width=1.6in]{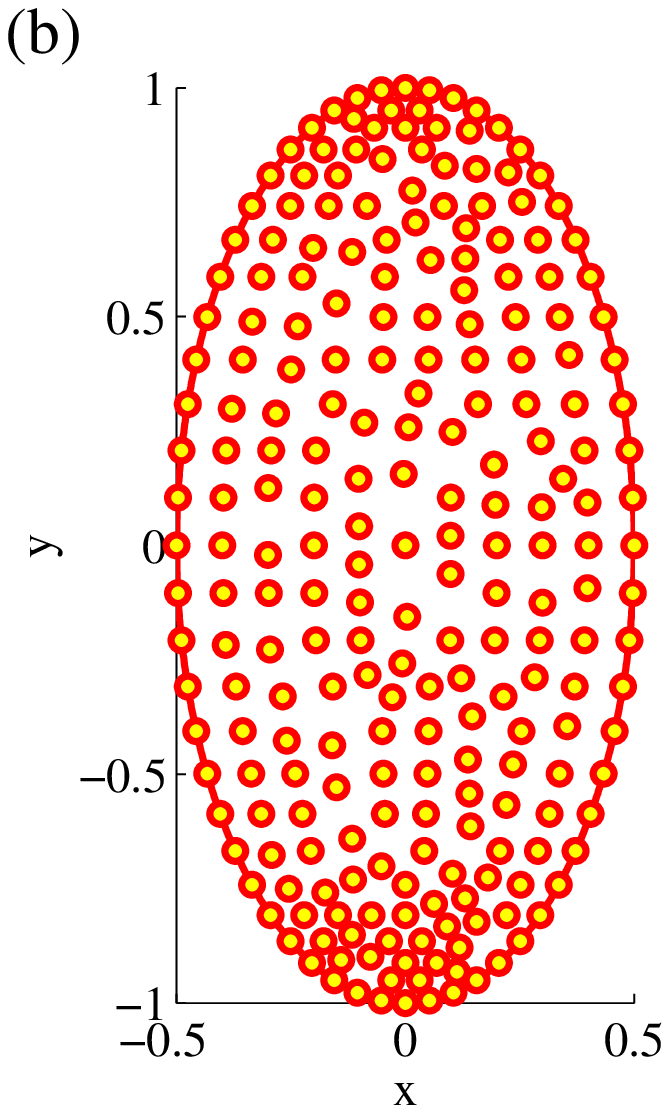}
\end{minipage}
\begin{minipage}[t]{0.3\linewidth}
\includegraphics[width=1.6in]{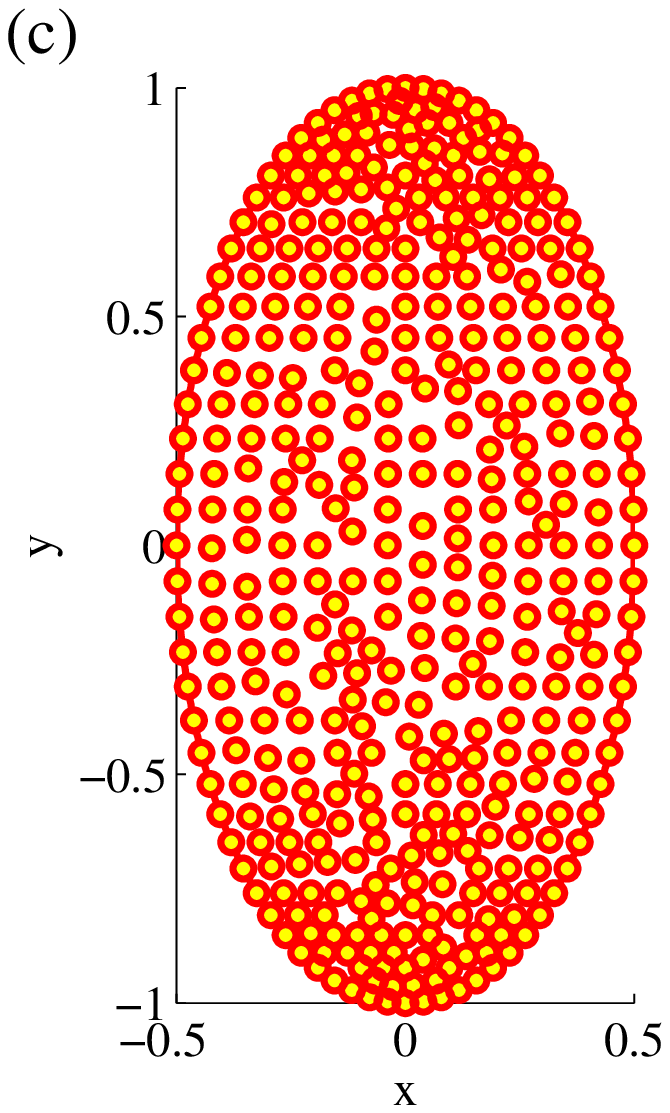}
\end{minipage}\\
\begin{minipage}[t]{0.3\linewidth}
\includegraphics[width=3.2in]{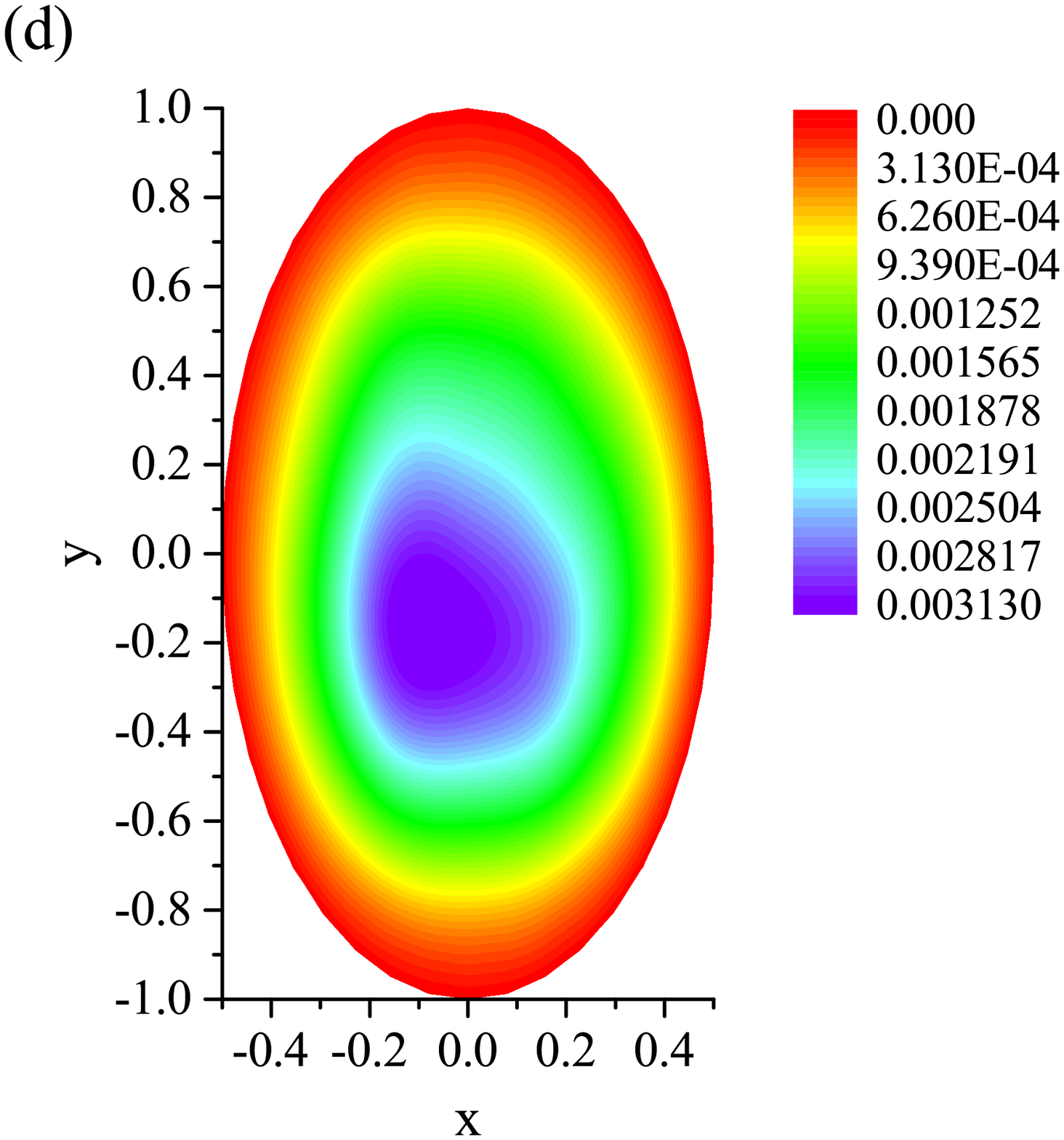}
\end{minipage}
\begin{minipage}[t]{0.3\linewidth}
\includegraphics[width=3.2in]{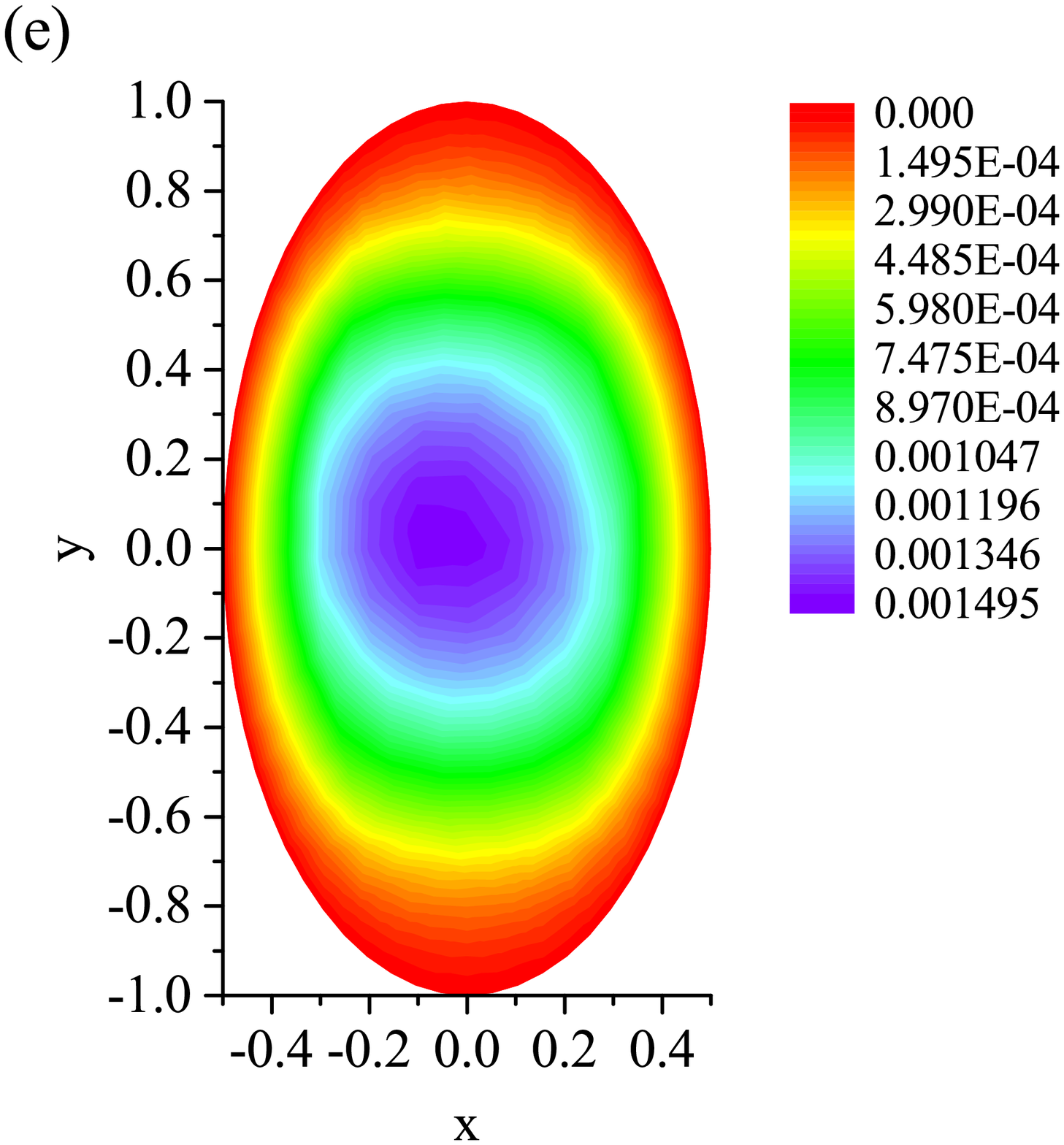}\\[10pt]
\end{minipage}
\begin{minipage}[t]{0.3\linewidth}
\includegraphics[width=3.2in]{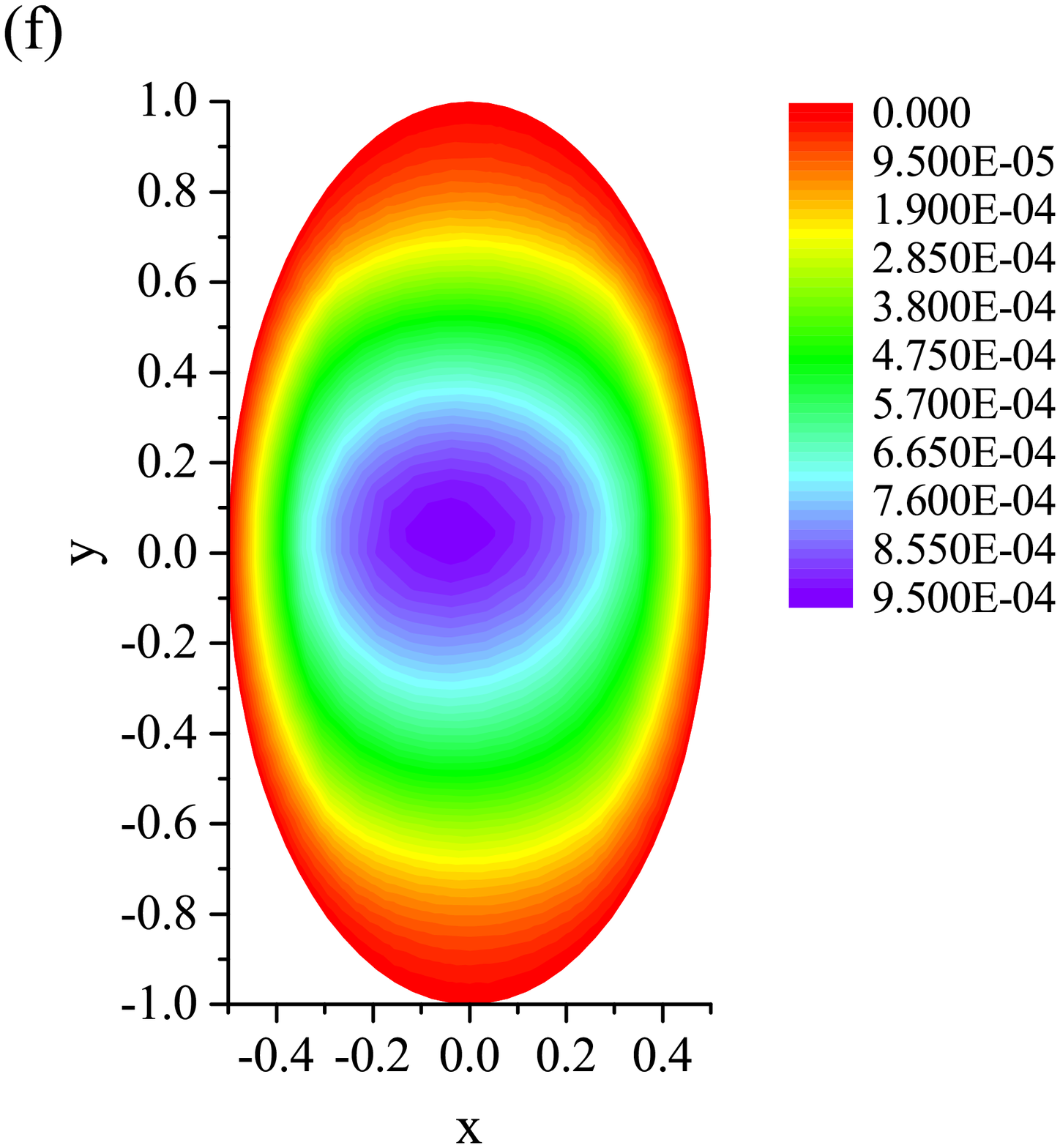}
\end{minipage}\\
\caption{The configurations of 113, 239 and 413 nodes and their absolute errors by IQ-based DQ method at $t=1$.}\label{fig3}
\end{figure*}

\subsection{Three-dimensional problems}
\noindent
\textbf{Example 6.7.} Consider the three-dimensional multi-term TSFDE:
\begin{align*}
\left\{
             \begin{array}{lll}\displaystyle
\begin{aligned}
       &\sum_{r=1}^{4}a_rD_t^{\theta_r}u(x,y,z,t)
       -\frac{y^\gamma}{2}\displaystyle\frac{\partial^{\gamma}_- u(x,y,z,t)}{\partial z^{\gamma}_-}\\
        &  \quad =f(x,y,z,t), \quad (x,y,z;t)\in\Omega\times(0,T],\\
      & u(x,y,z,0)=0,\quad (x,y,z)\in\Omega, \\
      & u(x,y,z,t)=t^3y^{2-\gamma}(1-x-y-z)^2, \\
      &\qquad\qquad\qquad\quad\quad \ (x,y,z;t)\in\partial\Omega\times(0,T],
\end{aligned}
             \end{array}
        \right.
\end{align*}
on the triangular pyramid domain $\Omega=\{(x,y,z)|0\leq x,y\leq 1, 0\leq z\leq 1-x-y\}$ with
$\theta_1=2$, $\theta_2=1.5$, $\theta_3=2$, $\theta_4=1$, $a_1=0.2$, $a_2=0.4$, $a_3=0.7$, $a_4=0.9$,
$\gamma=1.8$ and the source term
\begin{align*}
   f(x,y,z,t)&=\sum_{r=1}^{4}\frac{6a_rt^{3-\theta_r}}{\Gamma(4-\theta_r)}y^{2-\gamma}(1-x-y-z)^2\\
             &\quad -\frac{t^3y^2(1-x-y-z)^{2-\gamma}}{\Gamma(3-\gamma)},
\end{align*}

The analytical solution is given by $u(x,y,z,t)=t^3y^{2-\gamma}(1-x-y-z)^2$.  We test the convergent
behavior in space by using Inverse Multiquadrics and Inverse Quadratics as the test functions.
{\color{red}The algorithm is run on the nodal distributions of 91, 119, 165 and 214 nodes, and for Inverse Multiquadrics, the
corresponding shape parameters $c$ is chosen to be 0.1, 0.2, 0.9 and 1, while for Inverse Multiquadrics,
we choose $c$ to be 0.08, 0.16, 0.7 and 0.8, respectively. }
The convergent results at $t=0.5$ with $q=2$ and $\tau=5.0\times10^{-4}$ are listed in Table \ref{tab9}.
The used configuration of 119 nodes and its corresponding absolute error of
IMQ-based DQ method at $t=0.5$ are displayed in Fig. \ref{fig5}, respectively. It is clear from
the above table and figure that the errors of these two DQ methods present a stable downward trend
as we continually refine the nodes and this fact shows that our DQ methods are stable and convergent
in dealing with the three-dimensional multi-term TSFPDEs.
The absolute error of IMQ-based DQ method is relatively large at the nodes near the center of this triangular pyramid domain and
the further the node from its center, the smaller the error looks likes, which coincides with the
change of analytical solution. Moreover, IMQ-based DQ method produces nearly the same accuracy
as IQ-based DQ method by using the above parameters. \\

\begin{table*}[!htb]
\centering
\caption{The convergent results at $t=0.5$ with $q=2$ and $\tau=5.0\times10^{-4}$ for Example 6.7}\label{tab9}
\begin{tabular}{lclclc}
\toprule
\multicolumn{1}{l}{\multirow{2}{0.6cm}{$M$}} &\multicolumn{2}{l}{IMQ-based DQ method}
    &\multicolumn{2}{l}{IQ-based DQ method} \\
\cline{2-5}& $||u-U||_{L^2}$  &$||u-U||_{L^\infty}$  &$||u-U||_{L^2}$  & $||u-U||_{L^\infty}$ \\
\midrule 90     &  1.1210e-04  &7.5967e-04   &1.0397e-04  &6.2086e-04 \\
         118    &  5.9533e-05  &4.3038e-04   &7.8823e-05  &5.0199e-04 \\
         164    &  4.6240e-05  &3.2299e-04   &4.7980e-05  &3.3329e-04 \\
         213    &  4.1000e-05  &2.4174e-04   &4.2061e-05  &2.6229e-04 \\
\bottomrule
\end{tabular}
\end{table*}

\begin{figure*}[!htb]
\begin{minipage}[t]{0.43\linewidth}
\includegraphics[width=3.2in]{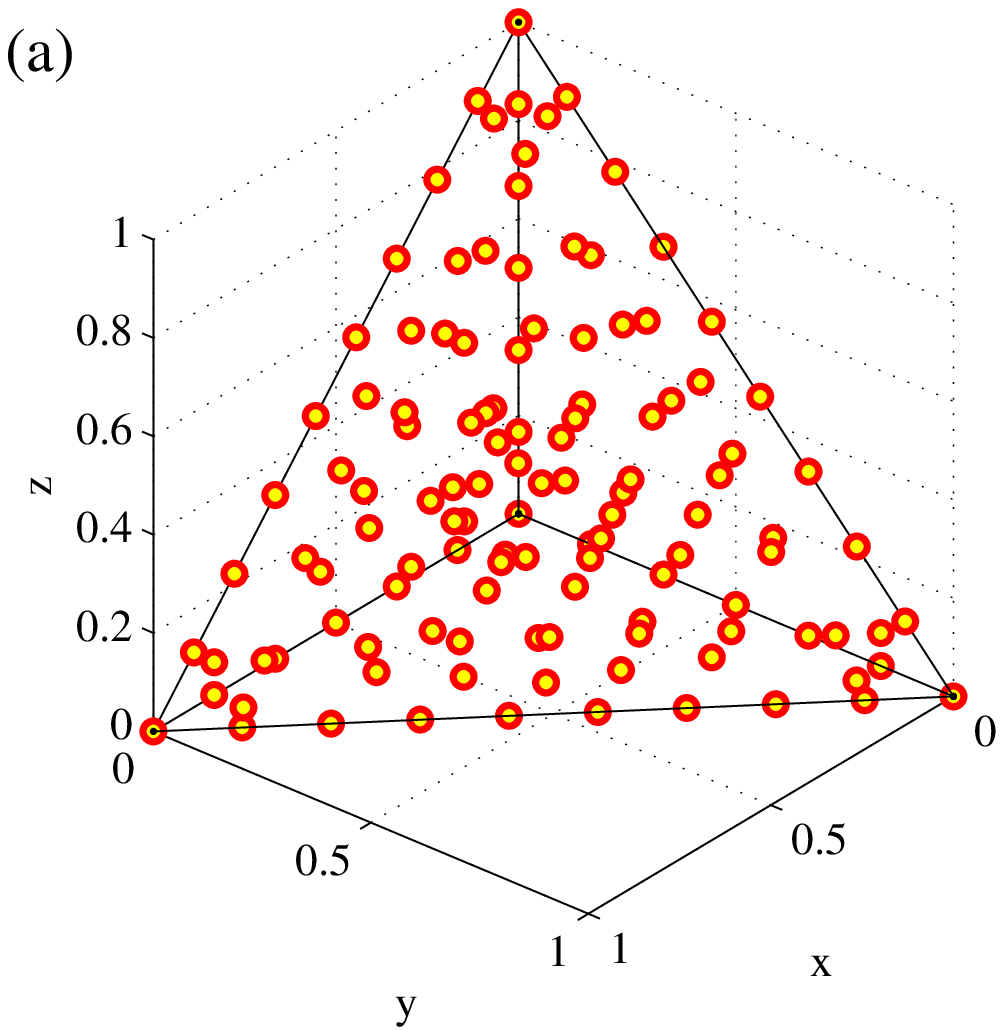}
\end{minipage}
\begin{minipage}[t]{0.43\linewidth}
\includegraphics[width=3.2in]{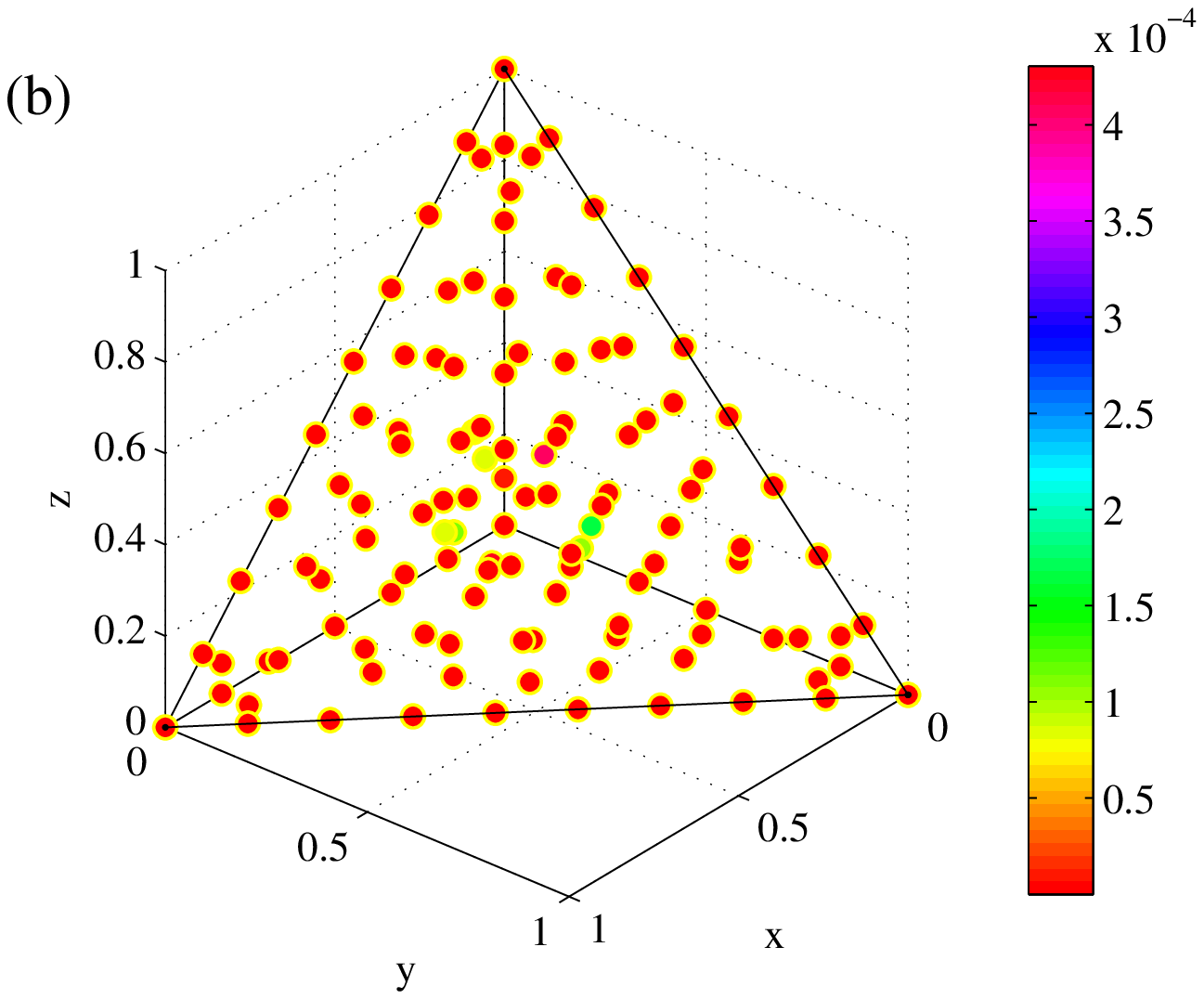}
\end{minipage}
\caption{The configuration of 119 nodes and its absolute error by IMQ-based DQ method at $t=0.5$.}\label{fig5}
\end{figure*}

\noindent
\textbf{Example 6.8.} Consider the three-dimensional multi-term TSFDE:
\begin{align*}
\small
\left\{
             \begin{array}{lll}\displaystyle
\begin{aligned}
     &\frac{{^C_0}D_t^{\theta_1}u(x,y,z,t)+3{^C_0}D_t^{\theta_2}u(x,y,z,t)}{2}\\
     &  -\frac{(y-0.5+ \sqrt{0.25-(x-0.5)^2-z^2})^\beta}{6}
        \displaystyle\frac{\partial^{\beta}_+ u(x,y,z,t)}{\partial y^{\beta}_+} \\
     &   \qquad=f(x,y,z,t),
        \quad (x,y,z;t)\in\Omega\times(0,T],\\
     &  u(x,y,z,0)=0,\quad (x,y,z)\in\Omega, \\
     &  u(x,y,z,t)=(1+t^4)\big(y-0.5+ \sqrt{0.25-(x-0.5)^2-z^2}\big)^3z^3,\\
     & \qquad\qquad\qquad\qquad (x,y,z;t)\in\partial\Omega\times(0,T],
\end{aligned}
             \end{array}
        \right.
\end{align*}
on the sphere domain $\Omega=\{(x,y,z)|(x-0.5)^2+(y-0.5)^2+z^2\leq0.25\}$ with $\theta_1=0.3$,
$\theta_2=0.8$ and $\beta=1.9$. The source term is given as follows
\begin{small}
\begin{align*}
   f(x,y,z,t)&=\Bigg\{\frac{12t^{4-\theta_1}}{\Gamma(5-\theta_1)}+\frac{36t^{4-\theta_2}}{\Gamma(5-\theta_2)}\Bigg\}\\
             &\quad    \cdot\big(y-0.5+ \sqrt{0.25-(x-0.5)^2-z^2}\big)^3z^3\\
             & -\frac{(1+t^4)\big(y-0.5+ \sqrt{0.25-(x-0.5)^2-z^2}\big)^3z^3}{\Gamma(4-\beta)},
\end{align*}
\end{small}
to enforce the analytical solution $u(x,y,z,t)=(1+t^4)\big(y-0.5+ \sqrt{0.25-(x-0.5)^2-z^2}\big)^3z^3$.

In this test, we select $q=4$, $\tau=1.0\times10^{-3}$ and report the convergent results at $t=0.5$ in Table \ref{tab8}, where
Multiquadrics and Inverse Multiquadrics are used as the test functions and {\color{red}the free parameters} are
chosen to be $\nu=3$, $\sigma=1$ and $\nu=3.3$, $\sigma=1$, respectively. Also, we highlight the used configuration
of 181 nodes and its corresponding absolute error of MQ-based DQ method at $t=0.5$ in Fig. \ref{fig4}. The results observed
from the table and figure illustrate that the proposed DQ methods generate the approximate solutions which
are seems to be very consistent with the analytical solution and the global errors in mean-square and
maximum norm sense by both two DQ methods have no big difference. Besides, the absolute error
of MQ-based DQ method coincides with the change of analytical solution. Therefore,
our DQ methods can be considered as a good choice in solving the three-dimensional multi-term TSFPDEs.

\begin{table*}[!htb]
\centering
\caption{The convergent results at $t=0.5$ with $q=4$ and $\tau=1.0\times10^{-3}$ for Example 6.8}\label{tab8}
\begin{tabular}{lclclc}
\toprule
\multicolumn{1}{l}{\multirow{2}{0.6cm}{$M$}} &\multicolumn{2}{l}{MQ-based DQ method}
    &\multicolumn{2}{l}{IMQ-based DQ method} \\
\cline{2-5}& $||u-U||_{L^2}$  &$||u-U||_{L^\infty}$  &$||u-U||_{L^2}$  & $||u-U||_{L^\infty}$ \\
\midrule 73     &  1.7998e-04  &8.8476e-04   &1.6390e-04  &8.4253e-04 \\
         101    &  1.2964e-04  &8.0668e-04   &1.2897e-04  &8.1281e-04 \\
         180    &  5.7266e-05  &4.6403e-04   &5.9258e-05  &4.7803e-04 \\
         336    &  2.1243e-05  &1.3452e-04   &2.2835e-05  &1.4298e-04 \\
\bottomrule
\end{tabular}
\end{table*}

\begin{figure*}[!htb]
\begin{minipage}[t]{0.43\linewidth}
\includegraphics[width=3.2in]{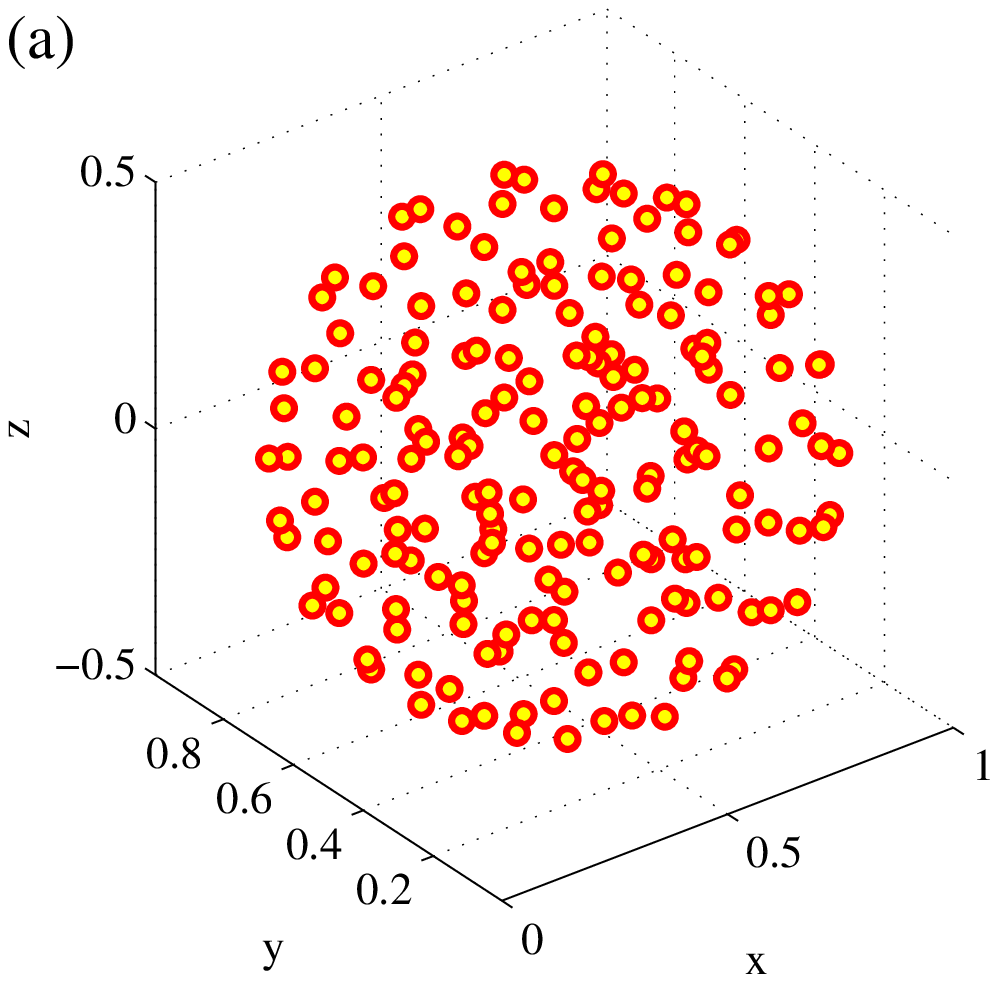}
\end{minipage}
\begin{minipage}[t]{0.43\linewidth}
\includegraphics[width=3.2in]{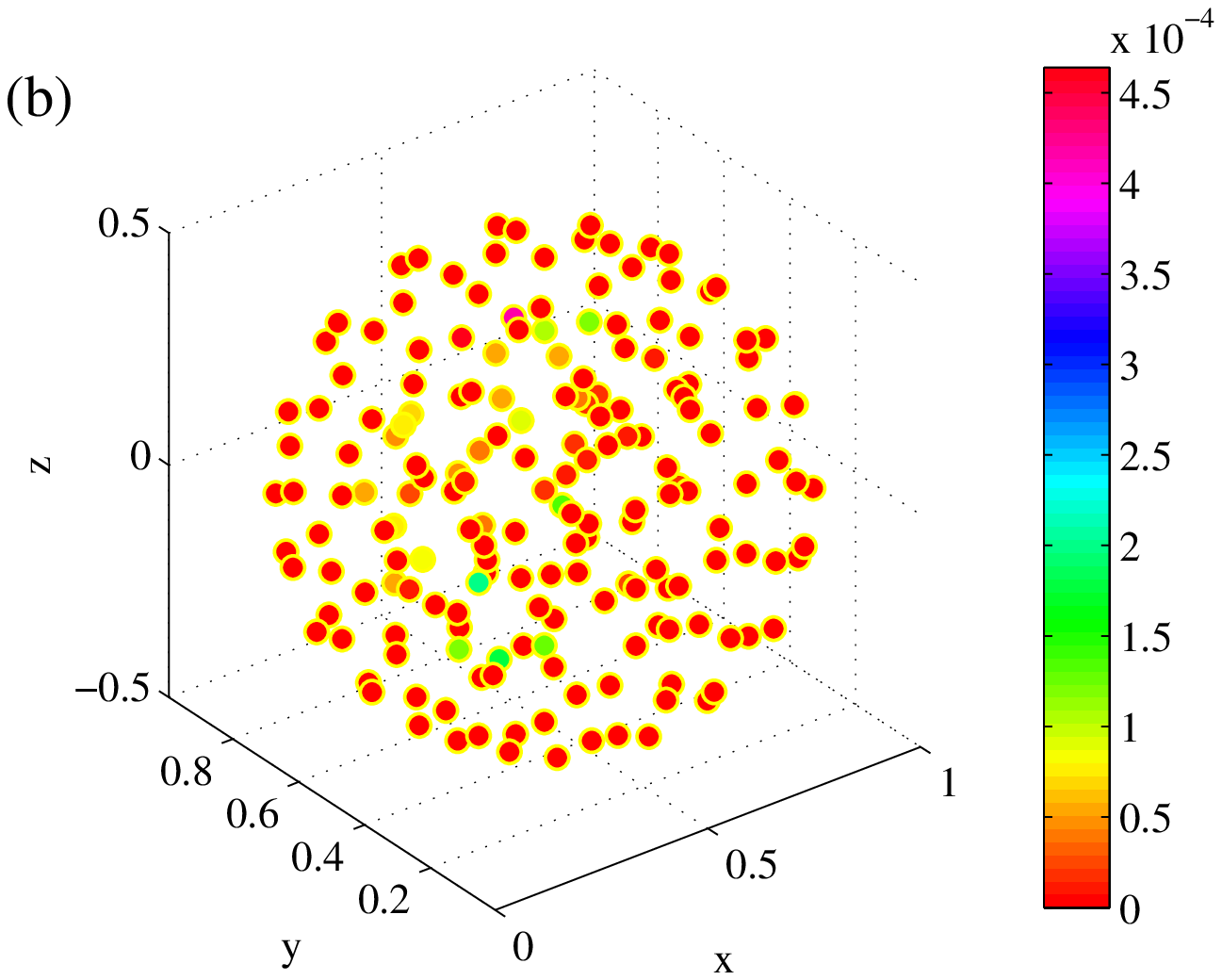}
\end{minipage}
\caption{The configuration of 181 nodes and its absolute error by MQ-based DQ method at $t=0.5$.}\label{fig4}
\end{figure*}

{\color{red}
\subsection{Numerical simulation}
\indent

In the last, we simulate the advection and diffusion of a plume of contaminant particles in groundwater,
governed by
\begin{align*}
\small
\left\{
             \begin{array}{lll}\displaystyle
\begin{aligned}
     &\frac{\partial u(x,y,t)}{\partial t}
        +\frac{0.03}{4\pi\cos(\alpha\pi/2)}\Bigg(\frac{\partial^{\alpha}_+ u(x,y,t)}{\partial x^{\alpha}_+}
     +\frac{\partial^{\alpha}_- u(x,y,t)}{\partial x^{\alpha}_-}\Bigg)\\
    &  \quad\displaystyle  +\frac{0.03}{4\pi\cos(\beta\pi/2)}\Bigg(\frac{\partial^{\beta}_+ u(x,y,t)}{\partial y^{\beta}_+}
       +\frac{\partial^{\beta}_- u(x,y,t)}{\partial y^{\beta}_-}\Bigg) \\
    &   \quad+0.012\frac{\partial u(x,y,t)}{\partial x}=0,\quad (x,y;t)\in\Omega\times(0,T],\\
    &  u(x,y,t)=0,\quad(x,y;t)\in\partial\Omega\times(0,T],
\end{aligned}
             \end{array}
        \right.
\end{align*}
on $\Omega=[0,4]\times[0,1]$, with the following compactly supported initial-value condition
\begin{align*}
u(x,y,0)=\left\{
\begin{array}{ccc}\displaystyle
\begin{aligned}
     &1000*2^{ 1-\frac{1}{1-\frac{(x-0.5)^2}{0.04}-\frac{(y-0.5)^2}{0.04}} }, \\
      & \qquad\quad \textrm{if}\quad 1-\frac{(x-0.5)^2}{0.04}-\frac{(y-0.5)^2}{0.04}>0, \\
     &  0,\qquad \textrm{otherwise}.
\end{aligned}
             \end{array}
        \right.
\end{align*}

This computational experiment was previously considered in \cite{dq34}, but with fractional directional
derivative in space. Here, we replace the fractional directional derivative by the left,
right Riemann-Liouville derivatives and simulate the dynamic process of the plume of contaminant particles
by using Inverse Multiquadrics as test functions. The spreading of the plume of
contaminants is governed by the diffusion term, while its drifting is governed by the advection term.
We irregularly place 1105 nodes on $\Omega$ and run the algorithm by $q=2$, $\tau=1.0\times10^{-3}$, $\nu=9$ and $\sigma=1$,
where $\frac{\partial u(x,y,t)}{\partial x}$ is approximated by Eq. (\ref{xxzz01}).
Letting $\alpha=\beta=1.9$, we report the contour plots of the concentration fields
at $t=0.02$, $1$, and $3$ in Fig. \ref{figns1}, respectively.

As we clearly see from these contour plots, the contaminant particles are
injected into the groundwater in a circular region $\{(x,y)|(x-0.5)^2+(y-0.5)^2\leq0.04\}$ at $t=0$, and then
slowly spread in all directions. As time goes on, the contaminated area gets bigger and bigger, while its
concentration gets lower and lower. Retaking $\alpha=\beta=1.3$, we present the contour plot
of the concentration field at $t=3$ in Fig. \ref{figns2}. As compared to the simulated result in Fig. \ref{figns1},
we find that the contaminant particles have much smaller diffusion velocity, i.e., the size of fractional
powers can affect the diffusion velocity of these contaminant particles and this phenomena may be useful in
some physical applications.

\begin{figure*}[!htb]
\begin{minipage}[t]{0.9\linewidth}
\includegraphics[width=6.0in]{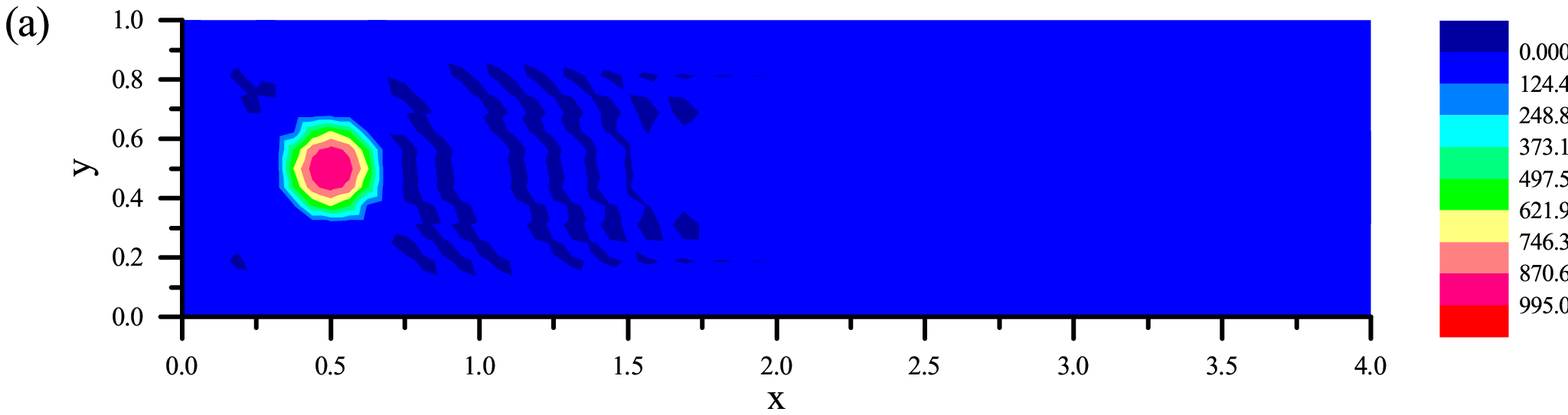}\\
\end{minipage}
\begin{minipage}[t]{0.9\linewidth}
\includegraphics[width=6.0in]{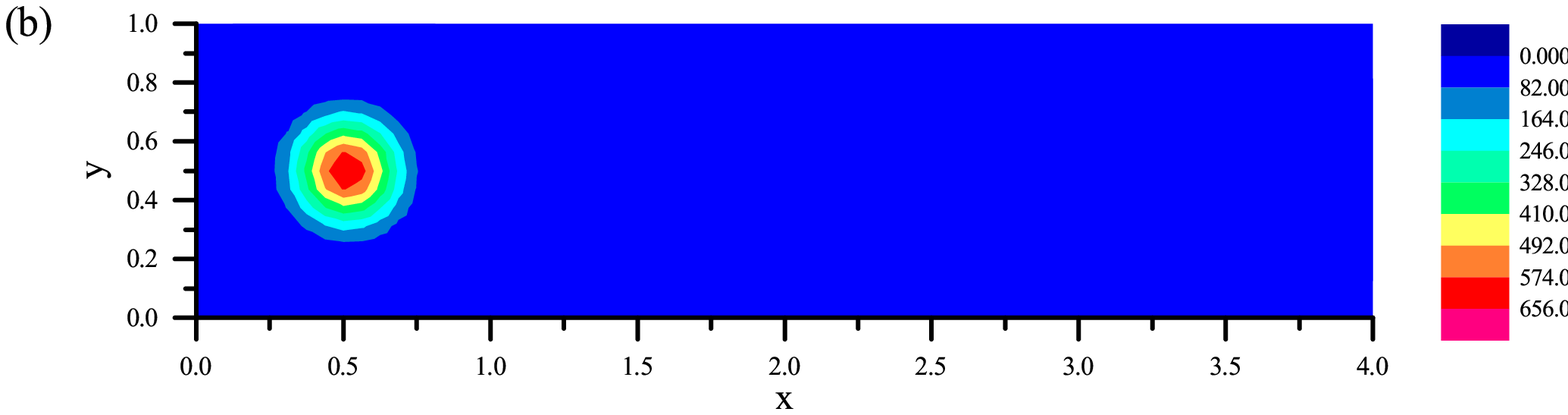}\\
\end{minipage}
\begin{minipage}[t]{0.9\linewidth}
\includegraphics[width=6.0in]{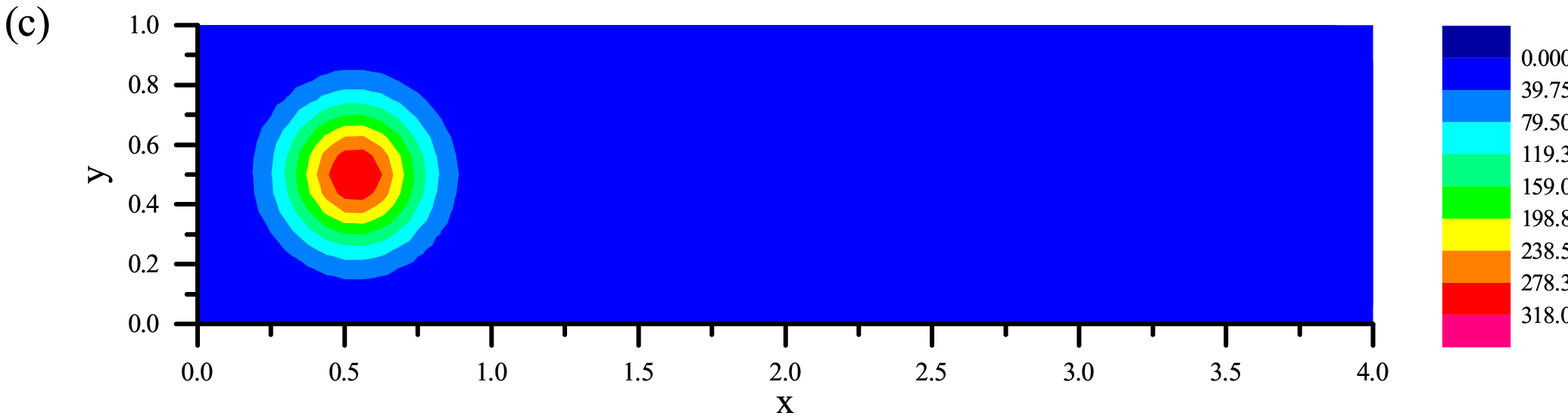}\\
\end{minipage}
\caption{The contour plots of of $u(x,y,t)$ at $t=0.02$, $1$ and $3$ when $\alpha=\beta=1.9$.}\label{figns1}
\end{figure*}

\begin{figure*}[!htb]
\begin{minipage}[t]{0.9\linewidth}
\includegraphics[width=6.0in]{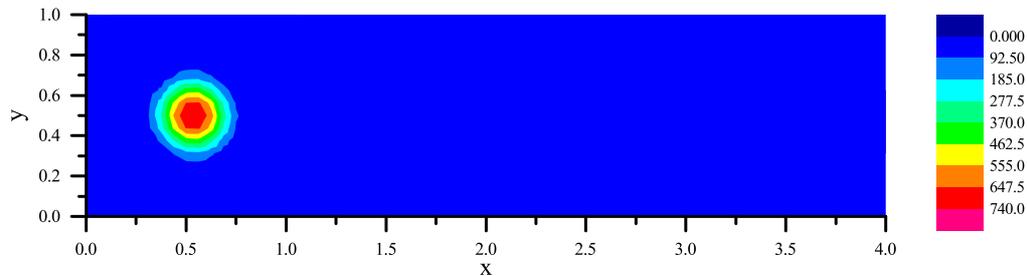}
\end{minipage}
\caption{The contour plot of $u(x,y,t)$ at $t=3$ when $\alpha=\beta=1.3$.}\label{figns2}
\end{figure*}

}
\section{Conclusion} \label{s6}
\indent

{\color{red}
This investigation has considered} an advanced DQ technique for the numerical solutions of the
multi-term TSFPDEs {\color{blue}on two- and three-dimensional convex domains}.
The time discretization was basically done by high-order difference schemes and the spatial
discretization was done by using a class of DQ formulas to approximate the fractional derivatives with the help
of some RBFs as test functions.
The unconditional stability and convergent analysis have been established for the time semi-discrete scheme.
In numerical study, we applied the proposed DQ method to several typical problems.
It produces very satisfactory accuracy both in time and space. The outcomes and comparisons with the true solutions or
the errors yielded by the other methods like FE method illustrate that our method is convergent
with theoretical orders and effective or even more favorable {\color{red}on aspect of} computational accuracy.

{\color{red}On the other hand, it is a common knowledge that
the condition number of RBFs interpolation matrix would be quite large as the nodal number increases.
This would be particularly evident when we apply it to large-scale real problems and this limitation can
affect the computational stability of our method in space.
Consequently, in our future researches, we will consider the local RBFs as test functions to further improve this algorithm
and extend it to some nonlinear physical problems with profound application background,
such as fractional Schr\"{o}dinger or Ginzburg-Landau equations.}\\

\begin{acknowledgements}
{\color{red}The authors thank the anonymous referees and editor for their comments and suggestions which have improved this work.}
This research was supported by the Natural Science Foundation of Hunan Province of China (Nos. 2020JJ5514, 2020JJ4554),
the Scientific Research Funds of Hunan Provincial Education Department (Nos. 19C1643, 19B509)
and  National Natural Science Foundation of China (No. 11971386).
\end{acknowledgements}
\small
\textbf{Compliance with ethical standards}\\

\noindent
\small
\textbf{Conflict of interest} The authors declare that they have no conflict of interest.


\end{document}